\documentclass[11pt]{amsart} \usepackage[left=1.3in, right=1.2in, top=1in, bottom=1in, includefoot]{geometry}

\usepackage{soul,xcolor}
\usepackage{bbm}

\usepackage{caption}
\usepackage{subcaption}

\usepackage[running]{lineno}

\newcommand*\patchAmsMathEnvironmentForLineno[1]{
\expandafter\let\csname old#1\expandafter\endcsname\csname #1\endcsname
\expandafter\let\csname oldend#1\expandafter\endcsname\csname end#1\endcsname
\renewenvironment{#1}
{\linenomath\csname old#1\endcsname}{\csname oldend#1\endcsname\endlinenomath}}
\newcommand*\patchBothAmsMathEnvironmentsForLineno[1]{
\patchAmsMathEnvironmentForLineno{#1}
\patchAmsMathEnvironmentForLineno{#1*}}
\AtBeginDocument{
\patchBothAmsMathEnvironmentsForLineno{equation}
\patchBothAmsMathEnvironmentsForLineno{align}
\patchBothAmsMathEnvironmentsForLineno{flalign}
\patchBothAmsMathEnvironmentsForLineno{alignat}
\patchBothAmsMathEnvironmentsForLineno{gather}
\patchBothAmsMathEnvironmentsForLineno{multline}
}

\usepackage{amssymb}
\usepackage{enumitem}

\usepackage[nocompress]{cite}

\usepackage{hyperref}

\hypersetup{
	colorlinks   = true, 
	urlcolor     = blue, 
	linkcolor    = blue, 
	citecolor   = red 
}

\usepackage{caption}

\newtheorem{thm}{Theorem}[section]
\newtheorem{prop}[thm]{Proposition}
\newtheorem{lem}[thm]{Lemma}
\newtheorem{cor}[thm]{Corollary}

\theoremstyle{definition}
\newtheorem{definition}[thm]{Definition}

\newtheorem{prob}{Problem}

\theoremstyle{remark}

\newtheorem{remark}[thm]{Remark}

\theoremstyle{claim}

\newtheorem{claim}{Claim}

\numberwithin{equation}{section}

\newcommand{\R}{\mathbb{R}}  
\newcommand{\N}{\mathbb{N}}  

\usepackage{float}
\usepackage{amsmath}
\usepackage{amssymb} 
\usepackage{mathtools}

\newcommand\norm[1]{\lVert#1\rVert}

\renewcommand\H[1]{H(#1)}

\newcommand{\inner}[2]{\langle{#1},{#2}\rangle}
\newcommand{\scalarprod}[2]{\big({#1},{#2}\big)}

\newcounter{rtaskno}
\newcommand{\rtask}[1]{\refstepcounter{rtaskno}\label{#1}}

\newcounter{rsubtaskno}
\newcommand{\rsubtask}[1]{\refstepcounter{rsubtaskno}\label{#1}}
\newcounter{rsubsubtaskno}

\DeclareRobustCommand{\gobblefive}[5]{}
\newcommand*{\SkipTocEntry}{\addtocontents{toc}{\gobblefive}}

\usepackage{xcolor}

\begin{document}


\title[Well-posedness and applications]{Well-posedness of evolutionary differential variational-hemivariational inequalities and applications to frictional contact mechanics}


\author[N. Skoglund Taki and K. Kumar]{Nadia Skoglund Taki and Kundan Kumar}
\address{Center for Modeling of Coupled Subsurface Dynamics \\ Department of Mathematics\\ University of Bergen\\ Postbox 7800\\ 5020 Bergen\\ Norway}
\date{\today}
\email{Nadia.Taki@uib.no, Kundan.Kumar@uib.no}






\begin{abstract}
    In this paper, we study the well-posedness of a class of evolutionary variational-hemivariational inequalities coupled with a nonlinear ordinary differential equation in Banach spaces. 
    The proof is based on an iterative approximation scheme showing that the problem has a unique mild solution. In addition, we established the continuity of the flow map with respect to the initial data.
    Under the general framework, we consider two new applications for modelling of frictional contact for viscoelastic materials. In the first application, we consider Coulomb friction with normal compliance, and in the second, normal damped response. The structure of the friction coefficient $\mu$ is new with motivation from geophysical applications in earth sciences with dependence on an external state variable $\alpha$ and the slip rate $|\Dot{u}_\tau|$.
\end{abstract}


\maketitle
\section{Introduction}

\noindent
This work concerns the study of an evolutionary differential variational-hemivariational inequality modelling mathematical problems from contact mechanics. These systems are relevant for many physical phenomena ranging from engineering to biology (see, e.g., \cite{Dancer2003, Duvaut1976,shillor2004} and the references therein). We are interested in frictional contact phenomena for viscoelastic materials with the linearized strain tensor, which have been studied intensively, see, e.g., some of the relevant books \cite{Duvaut1976,Migorski2012, shillor2004}. 
\\
\indent
Let $V$ and $Y$ be two Banach spaces, and for $T>0$, we let $[0,T]$ be the time interval of interest. Then, the Cauchy problem under consideration reads
\begin{subequations}\label{eq:prob}
	\begin{align}\label{eq:alpha_firsteq}
		&\Dot{\alpha}(t) = 	\mathcal{G}(t,\alpha(t),Mw(t)),\\
		&\inner{\Dot{w}(t) + A(t, w(t)) - f(t)+ \mathcal{R}w(t) }{v- w(t) } 
		+ \varphi(t, \alpha(t),\mathcal{S}_\varphi  w(t), Mw(t), Kv) \label{eq:w1}\\
		&-  \varphi(t,  \alpha(t),\mathcal{S}_\varphi  w(t), Mw(t), Kw(t))+
		j^\circ(t,\alpha(t), \mathcal{S}_j w(t),Nw(t); Nv -Nw(t))\geq 0, \notag
	\end{align}
	for all $v\in V$, a.e. $t\in (0,T)$ with
	\begin{equation}\label{eq:w2}
		w(0) = w_0, \ \ \alpha(0) = \alpha_0.
	\end{equation}
\end{subequations}
Here, $A$ and $\mathcal{R}$ are nonlinear operators related to the viscoelastic constitutive laws.   Further, $j^\circ$ is a generalized directional derivative of a functional $j$. The functionals $\varphi$ and $j$ are determined by contact boundary conditions. We require $\varphi$ to be convex in its last argument, while $j$ may be nonconvex with appropriate structures given later (see Section \ref{sec:problem_and_mainresult}). The operators $\mathcal{S}_\varphi$ and $\mathcal{S}_j$ relate to the contact conditions, and $\mathcal{G}$ is assumed to be a nonlinear operator related to the change in the external state variable $\alpha$. The data $f$ is related to the given body forces and surface traction, and $w_0$ and $\alpha_0$ represent the initial data. Lastly, $M$, $N$, and $K$ are bounded linear operators related to the tangential and normal trace operators.  The Cauchy problem \eqref{eq:w1}-\eqref{eq:w2} is called a hemivariational inequality if $\varphi\equiv 0$ and variational inequality if $j^\circ\equiv 0$. Moreover, a solution to \eqref{eq:prob} is understood in the mild sense.
\begin{definition}\label{def:sols}
	A pair of functions $(w,\alpha)$, where $\alpha\in C([0,T];Y)$ and $w : [0,T] \rightarrow V$ measurable, is said to be a mild solution of \eqref{eq:prob} if $\alpha$ and $w$, respectively, satisfy
	\begin{equation*}
		\alpha(t) = \alpha_0 + \int_0^t	\mathcal{G}(s,\alpha(s),Mw(s)) ds
	\end{equation*} 
	and \eqref{eq:w1}-\eqref{eq:w2}.
\end{definition}

The main purpose of this paper is to extend the results from \cite{Migorski2022, Patrulescu2017} to prove well-posedness of \eqref{eq:prob} with applications to rate-and-state frictional contact problems. We prove that the pair $(w,\alpha)$ is a solution to \eqref{eq:prob} in the sense of Definition \ref{def:sols} and that the flow map depends continuously on the initial data. The problem setting is motivated by \cite{Pipping2015_phd,shillor2004,Patrulescu2017,sofonea2017}, and the techniques have taken inspiration from \cite{Migorski2019}.
\subsection{Former well-posedness results}
Special cases of \eqref{eq:prob} have been investigated in literature.  The recent work \cite{Migorski2022} is closest to our setting. They prove well-posedness for an ordinary differential equation coupled with a variational-hemivariational inequality with applications to viscoplastic material and viscoelasticity with adhesion. In fact, if we let $\varphi$ be independent of $Mw$ in its third argument and
relax the more generalized structure of $\varphi$ (see Remark \ref{remark:mubounded} for more details), then \eqref{eq:prob} reduces to the problem studied in \cite{Migorski2022}. However, keeping the dependence of $Mw$ in $\varphi$ and a generalized structure of $\varphi$ (see Remark \ref{remark:mubounded}) allows us to include applications with a new structure of the friction coefficient. On the other hand, neglecting $\alpha$ and $Mw$ in $\varphi$ and $\alpha$ in $j^\circ$, existence and uniqueness is provided in \cite[Section 10.3]{sofonea2017}. If we let $\varphi \equiv 0$ and $j^\circ$ be independent of $\alpha$ and $\mathcal{S}_jw$, existence and uniqueness was proved in \cite[Section 6]{han}. 

In the quasi-static case tackled in \cite{Patrulescu2017} with $j^\circ \equiv 0$ and a simplified structure of $\varphi$ (see Remark \ref{remark:mubounded}), they proved existence and uniqueness of the solution pair by an implicit method, where they rewrite \eqref{eq:alpha_firsteq} to only depend on $w$. However, the setting of \cite{Patrulescu2017} is not applicable in our case as the inertial term restricts the space-time regularity for $w$. We refer to \cite[p.2]{Migorski2022} for further discussion.

\subsection{Physical setting}\label{sec:intro_physics}
A mathematical model in contact mechanics needs several relations: a constitutive law, a balance equation, boundary conditions, interface laws, and initial conditions. The constitutive laws help us describe the material's mechanical reactions (stress-strain type). In most cases, constitutive laws originate from experiments, though they are verified to satisfy certain invariance principles. We refer to \cite[Chapter 6]{han2002} for a general description of several diagnostic experiments which provide the needed information to construct constitutive laws for specific materials. The interface laws are prescribed on the possible contact surface. We refer to the interface laws in tangential direction as friction laws and in normal direction as contact conditions. The mathematical treatment of these problems gives rise to the variational-hemivariational inequalities of the form \eqref{eq:w1}-\eqref{eq:w2} where we put appropriate constraints on the operators to fit the applications of interest.
\\
\indent
We are mainly interested in studying frictional problems with the following dependencies:
\begin{align}\label{eq:mu_alpha}
	\mu &= \mu(|\Dot{u}_\tau(t)|,\alpha(t)), &\Dot{\alpha}(t) = G(\alpha(t),|\Dot{u}_\tau(t)|).
\end{align}
One application with the dependencies seen in \eqref{eq:mu_alpha} is 
a memory-dependent friction coefficient (see, e.g., \cite[Section 5.3]{Oden1985}), which is also referred to as rate-and-state friction law. This is modelled via an ODE, where the \emph{state variable} $\alpha$ 
tracks information of the contact surface using the slip rate $|\Dot{u}_\tau(t)|$ found from solving \eqref{eq:w1}-\eqref{eq:w2} and then updates the friction coefficient. Under certain constraints we may consider $\alpha$ as the surface temperature or humidity on the contact surface. 

In \cite{Patrulescu2017}, they assume that $\mu(|\Dot{u}_\tau(t)|,\alpha(t))$ is bounded and Lipschitz with respect to both arguments. Additionally to the boundeness assumption on the friction coefficient, \cite{Migorski2022} considers applications in the frictionless setting and $\mu = \mu(u_\nu(t))$. 
Our framework is therefore an extension of the frameworks in \cite{Patrulescu2017,Migorski2022}, where we need to use different techniques to prove the well-posedness of \eqref{eq:prob} in the sense of Definition \ref{def:sols}. Lastly, $\mu=\mu(|\Dot{u}_\tau|)$ is covered in \cite[Section 6.3 and 8.1]{Migorski2012} (see also \cite[p.185-187]{han}). A discussion on many different friction models can be found in \cite{Zmitrowicz}.

\subsection{Contributions and outline} The novelties of this paper are:
\begin{itemize}
    \item Well-posedness of \eqref{eq:prob} in the sense of Definition \ref{def:sols}, where the proof is based on an iterative decoupling approach that directly gives rise to a numerical method. 
    \item A more complicated structure of $\varphi$ (see Remark \ref{remark:mubounded}) that allows for a larger set of conditions on the contact surface.
    \item Two new applications with Coulomb's friction. One contact problem is with normal compliance, and the latter is with normal damped response. The general framework allows a new structure for the friction coefficient that can be unbounded. 
\end{itemize}

The paper is organized as follows. In Section \ref{sec:func_nonsmooth}, we introduce the function spaces and some basics of nonsmooth analysis in order to better understand the problem setting. In Section \ref{sec:problem_and_mainresult}, we present our problem statement and the assumptions on the data. The section ends with our main result, Theorem \ref{thm:mainresult}, that summarizes the well-posedness of \eqref{eq:prob} in the sense of Definition \ref{def:sols}. The proof of the theorem is presented in Section \ref{sec:proof_main_result} utilizing a preliminary result stated in Section \ref{sec:preliminary}. Next, two applications fitting our framework will be introduced in Section \ref{sec:application1}-\ref{sec:application2}. In Section \ref{sec:rateandstate}, we introduce an
application motivated by earth sciences. Lastly, in Appendix \ref{appendix:comments_app}, we include remarks on the assumptions needed in the proof of Theorem \ref{thm:mainresult}, \ref{thm:finalthm}, and \ref{thm:wellposed_app2}. Appendix \ref{appendix:proof_part2}-\ref{appendix:assumption_phiandj} contains proofs of results that are similar to ones found elsewhere but needed throughout the paper.

\subsection{Notation}
We now present some notations that will be used in this paper.
\begin{itemize}
	\item Let $0<T<\infty$ be the maximal time.
	\item Let $d$ denote the dimension. In the applications, $d=2,3$.
	\item A point in $\R^d$ is denoted by $x = (x_i)$,  $1 \leq i \leq d$.
	\item $\mathbb{S}^d$ denotes the space of second order symmetric tensors on $\R^d$.
	\item We denote 	$|\cdot|$ as the Euclidean norm. 
	\item $\Omega \subset \R^d$ is a bounded open connected subset with a Lipschitz boundary $\Gamma = \partial \Omega$. We split $\Gamma$ into three disjoint parts; $\Gamma_D$, $\Gamma_N$, and $\Gamma_C$ with $\mathrm{meas}(\Gamma_D)>0$, $\mathrm{meas}(\Gamma_C)>0$, i.e., nonzero Lebesgue measure, but $\Gamma_N$ is allowed to be empty.
	\item In the applications, $\Omega$ is the reference configuration of a viscoelastic deformable body sliding on a foundation. Moreover, $\Gamma_D$ denotes the  Dirichlet boundary,  $\Gamma_N$ the Neumann boundary, and $\Gamma_C$ is the contact boundary.
	\item $\nu$ denotes the outward normal on $\Gamma$.
	\item We denote $\Bar{\Omega} = \Omega \cup \Gamma$.
	\item  $L^p(\Omega)$	denotes the space of Lebesgue $p$-integrable functions equipped with the norm $ \norm{v}_{L^p(\Omega)}  = \big( \int_\Omega |v|^p dx \big)^{1/p}$
	for $1\leq p <\infty$. 
	With the usual modifications for $L^\infty(\Omega)$.
	\item $C^\infty_c(\Omega)$ denotes the space of infinitely differentiable functions with compact support.
	\item We will denote $c$ as a positive constant, which might change from line to line.
	\item Let $h(t)$ be a function, then we denote the time derivative of $h(t)$ by $\Dot{h}(t)$ and the double time derivative as $\Ddot{h}(t)$. Assuming here that $h(t)$ has enough regularity such that it makes sense to take the time derivative of it twice.
	\item $(\Tilde{X},\norm{\cdot}_{\Tilde{X}})$ and $(\Tilde{Y},\norm{\cdot}_{\Tilde{Y}})$ denote arbitrary (separable reflexive) Banach spaces. For the convenience of the reader, this notation will be used when introducing general theory. If the theory is only needed for a specific (separable reflexive) Banach space, we write the specific space.
	\item The dual space of $\Tilde{X}$ will be denoted by $\Tilde{X}^\ast$.
	\item $2^{\Tilde{X}^\ast}$ denotes the set containing all subsets of $\Tilde{X}^\ast$. 
	\item The dual product between the spaces $\Tilde{X}$ and $\Tilde{X}^\ast$ will be denoted by $\inner{\cdot}{\cdot}_{\Tilde{X}^\ast \times \Tilde{X}}$.
	\item In the particular case, when $\Tilde{X}$ is an inner product space, we denote $\scalarprod{\cdot}{\cdot}_{\Tilde{X}}$ as its inner product.
	\item  $U$,	$V$, $X$, $Y$, and $Z$ denote real separable reflexive Banach spaces, and $H$ is a real Hilbert space.
	\item The embedding $V \subset H \subset V^\ast$ is referred to as an evolution triple.	Here, the embedding $V \subset H$ is continuous and $V$ dense in $H$. It follows that $H \subset V^\ast$ is continuously embedded (see, e.g., \cite[Section 7.2]{Roubicek2006}). Moreover, $V \subset H$ is compactly embedded, leading to $H \subset V^\ast$ being compactly embedded (see, e.g., \cite[Remark 3.4.4]{Denkowski_application}).
	\item For simplicity of notation, the dual product between $V^\ast$ and $V$ is denoted by $\inner{\cdot}{\cdot} = \inner{\cdot}{\cdot}_{V^\ast \times V} $. 
	\item $\mathcal{L}(\Tilde{X},\Tilde{Y})$ denotes the set of all bounded linear maps from $\Tilde{X}$ into $\Tilde{Y}$.
	\item We denote the operator norm of the operators $M: V \rightarrow U$, $N : V \rightarrow X$, and $K : V \rightarrow Z$  as  $\norm{M} = \norm{M}_{\mathcal{L}(V,U)}$, $\norm{N} = \norm{N}_{\mathcal{L}(V,X)}$, and $\norm{K} = \norm{K}_{\mathcal{L}(V,Z)}$, respectively.	
\end{itemize}

\tableofcontents

\section{Function spaces and basics of nonsmooth analysis} \label{sec:func_nonsmooth}
\noindent
In this section, we present the function spaces and fundamental results. For further information, we refer to standard textbooks, e.g., \cite{Cheney2001,evans,Pedersen1989,Roubicek2006}.

\subsection{Sobolev spaces}\label{sec:sobolev_spaces}
This section defines the solution spaces and the usual Sobolev spaces, which will become useful in Section \ref{sec:problem_and_mainresult} and in the applications, i.e., Section \ref{sec:applications}. 
We define
\begin{align}\label{eq:H_space}
	L^p(\Omega; \mathbb{R}^d ) &= \{ v = (v_i) : v_i \in L^p(\Omega), \ 1\leq i \leq d\},\\
	L^p(\Omega; \mathbb{S}^d ) &= \{ \tau = (\tau_{ik})  : \tau_{ik} = \tau_{ki} \in L^p(\Omega), \ 1\leq i,k \leq d\},
	\label{eq:Q_space}
\end{align}
for $1\leq p \leq \infty$. The associated norms will be denoted $\norm{\cdot}_{ L^p(\Omega; \mathbb{R}^d )}$ and $\norm{\cdot}_{L^p(\Omega; \mathbb{S}^d )}$, respectively. For $p=2$, \eqref{eq:H_space}-\eqref{eq:Q_space} are Hilbert spaces with the canonical inner products
\begin{align*}
	\scalarprod{u}{v}_{ L^2(\Omega; \mathbb{R}^d )} &= \int_\Omega u_iv_i dx = \int_\Omega u\cdot v dx, &\scalarprod{\sigma}{\tau}_{L^2(\Omega; \mathbb{S}^d )} &= \int_\Omega \sigma_{ik} \tau_{ik} dx =  \int_\Omega \sigma : \tau dx
\end{align*}
for all $u,v \in  L^2(\Omega; \mathbb{R}^d) $, $\sigma,\tau \in L^2(\Omega; \mathbb{S}^d)$. 
Moreover, we define the spaces
\begin{equation*}
	H^1(\Omega) = \{v \in L^2(\Omega) : \text{ the weak derivatives } \frac{\partial v }{\partial x_i}\text{ exist in } L^2(\Omega) , \ 1\leq i \leq d\}
\end{equation*}
and 
\begin{equation*}
	H^1(\Omega;\mathbb{R}^d) = \{ v = (v_i) : v_i \in H^1(\Omega), \ 1\leq i \leq d\}.
\end{equation*}
With abuse of notation, the trace of functions $v \in  H^1(\Omega;\mathbb{R}^d)$ on $\Gamma$ will still be denoted by $v$. For the displacement, we use the space 
\begin{equation}\label{eq:V_space}
	V = \{v \in H^1(\Omega;\mathbb{R}^d) :  v =  0 \ \text{ on }\Gamma_D \}.
\end{equation}
As a consequence of $\mathrm{meas}(\Gamma_D) > 0$, it follows by Korn's inequality, i.e., $\norm{\varepsilon(v)}_{L^2(\Omega; \mathbb{S}^d )}\geq c\norm{v}_{H^1(\Omega;\mathbb{R}^d)}$ (see, e.g., \cite[Lemma 6.2]{Kikuchi1988}), that $V$ is a Hilbert space with the canonical inner product
\begin{equation*}
	\scalarprod{u}{v}_V = \int_\Omega \varepsilon(u) : \varepsilon(v) dx.
\end{equation*}
Here, $\varepsilon : H^1(\Omega;\mathbb{R}^d) \rightarrow L^2(\Omega; \mathbb{S}^d )$ is the deformation operator defined by
\begin{equation*}
\varepsilon(u) = (\varepsilon_{ik}(u)), \ \ \  \varepsilon_{ik}(u) = \frac{1}{2} \Big(\frac{\partial u_k }{\partial x_i} + \frac{\partial u_i}{\partial x_k}\Big).
\end{equation*}
We denote the associated norm on $V$ by $\norm{\cdot}_V$. 	Moreover, if $\sigma$ is a regular function, say $\sigma \in C^1(\Bar{\Omega}; \mathbb{S}^d)$, the following Green's formula holds
\begin{align}\label{eq:greensformulasigma}
	\int_\Omega (\nabla \cdot \sigma) v dx = \int_\Gamma \sigma \nu \cdot v da - \int_\Omega \sigma : \varepsilon(v) dx \ \text{ for all } v\in H^1(\Omega;\mathbb{R}^d).
\end{align}
We also need the following trace theorem from \cite[Theorem 4.12]{Adams2003}.
\begin{thm}\label{thm:trace}
    Let $\Omega \subset \R^d$ be bounded with Lipschitz boundary $\Gamma$ and $\Gamma_C \subset \Gamma$ be such that $\mathrm{meas}(\Gamma_C)>0$.
    Then there exists a linear continuous operator $\gamma : V \rightarrow L^q(\Gamma_C;\mathbb{R}^d)$
    satisfying 
    \begin{align*}
        \norm{\gamma v}_{L^q(\Gamma_C;\R^d)} &\leq \norm{\gamma}_{\mathcal{L}(V,L^q(\Gamma_C;\R^d))} \norm{v}_V
    \end{align*}
    for all $v\in V$.  If $d=2$, then $2\leq q<\infty$. And  $d>2$,
    $2\leq q\leq p^\ast = \frac{4}{d-2}$. 
\end{thm}
\indent We denote $\mathbb{X} = \prod_{i=1}^\ell \Tilde{X}_i$ as a Cartesian product space, for some $\ell\in \mathbb{Z}_+$, where $(\Tilde{X}_i,\norm{\cdot}_{\Tilde{X}_i})$ are normed spaces for $i=1,\dots, \ell$. Then, $\mathbb{X}$ is equipped with the norm
\begin{equation*}
	\norm{(v_1,\dots,v_\ell)}_\mathbb{X} = \sum_{i=1}^\ell \norm{v_i}_{\Tilde{X}_i} \ \ \text{ for all $v_i \in \Tilde{X}_i$, $i=1,\dots,\ell$}.
\end{equation*}
Equivalently, we may equip $\mathbb{X}$ with the norm
\begin{equation*}
	\norm{(v_1,\dots,v_\ell)}_\mathbb{X}^2 = \sum_{i=1}^\ell \norm{v_i}_{\Tilde{X}_i}^2 \ \ \text{ for all $v_i \in \Tilde{X}_i$, $i=1,\dots,\ell$}.
\end{equation*}

\subsection{Time-dependent spaces}\label{sec:bochner_spaces}
Let $(V,H,V^\ast)$ be an evolution triple.
\begin{definition}\label{def:Bochner}
	Let $\Tilde{X}$ be a Banach space, and $T > 0$. The space $L^p(0,T; \Tilde{X})$, $1\leq p \leq \infty$, consists of all measurable functions $v : [0, T] \rightarrow \Tilde{X}$ such that
	\begin{align*}
		\int_0^T \norm{v(t)}_{\Tilde{X}}^p dt < \infty.
	\end{align*}
	With the usual modifications for $L^\infty(0, T; \Tilde{X})$. For brevity, we use the standard short-hand notation
	\begin{equation*}
		L^p_{t}\Tilde{X} = L^p(0, t; \Tilde{X})
	\end{equation*}
	for all $t \in [0,T]$ and $1\leq p \leq \infty$.
\end{definition}
We also introduce the solution space
\begin{align}\label{eq:spaces_VastW} 
	\mathcal{W}^{1,2}_T= \{ w \in L^2_{T}V  : \Dot{w} \in L^2_{T}V^\ast\},
\end{align}
equipped with the norm $\norm{w}_{\mathcal{W}^{1,2}_T}^2 = \norm{w}_{L^2_{T}V}^2 + \norm{\dot{w}}_{L^2_{T}V^\ast}^2$. The duality pairing between $L^2_TV^\ast$ and $L^2_TV$ is denoted by
\begin{equation*}
	\inner{\Tilde{v}}{v}_{L^2_TV^\ast \times L^2_TV}= \int_0^T \inner{\Tilde{v}(s)}{v(s)}  ds \ \ \text{ for all } \Tilde{v}\in L^2_TV^\ast, v\in L^2_TV.
\end{equation*}
We denote the space of continuous functions defined on $[0,T]$ with values in $\Tilde{X}$ by 
\begin{equation*}
	C([0,T]; \Tilde{X}) = \{h : [0,T] \rightarrow \Tilde{X} : h \text{ is continuous and } \norm{h}_{L^\infty_T{\Tilde{X}}} < \infty\}.
\end{equation*}
The next proposition can be found in, e.g., \cite[Proposition 3.4.14]{Denkowski_application} or \cite[Lemma 7.3]{Roubicek2006} and will help us provide estimates.
\begin{prop}\label{prop:integrationbypartsformula}
Let $(V,H,V^\ast)$ be an evolution triple in space, and $0<T< \infty$. Then, for any $v_1,v_2 \in \mathcal{W}^{1,2}_T$ (defined in \eqref{eq:spaces_VastW}), and for all $0 \leq s \leq t \leq T$, the following integration by parts formula holds:
\begin{equation*}
	\scalarprod{v_1(t)}{v_2(t)}_H -	\scalarprod{v_1(s)}{v_2(s)}_H = \int_s^t \big[ \inner{\Dot{v}_1(\tau)}{v_2(\tau)} + \inner{\Dot{v}_2(\tau)}{v_1(\tau)} \big] d \tau.
\end{equation*}
In addition, the embedding $\mathcal{W}^{1,2}_T \subset C([0,T];H)$ is continuous.
\end{prop}
For more on evolution spaces and other time-dependent spaces, which are also referred to as Bochner spaces, see, e.g., \cite[Chapter 7]{Roubicek2006}, \cite[Section 3.4]{Denkowski_application}, \cite[Chapter 23]{Zeidler1990}, \cite[Chapter V, Section 5]{Yosida1968}, or \cite[Section 5.9.2 and Appendix B.2]{evans}.\\
\indent 
We lastly introduce the following Bochner space needed in the applications, i.e., in Section \ref{sec:applications}:
\begin{definition}
	Let $V$ be a real Banach space, then the Bochner space $W^{1,2}(0,T;V)$ consists of all functions $u \in L^2_TV$ such that $\Dot{u}$ exists in the weak sense and belongs to $L^2_TV$.  The space $W^{1,2}(0,T;V)$ is equipped with the norm
	\begin{align*}
		\norm{u}_{W^{1,2}(0,T;V)}^2 = \norm{u}_{L^2_TV}^2 + \norm{\Dot{u}}_{L^2_TV}^2.
	\end{align*}
\end{definition}

\subsection{Generalized gradients}
Let $\Tilde{X}$ be a reflexive Banach space. In contact mechanics, we are often interested in contact conditions of the form $\zeta_\nu \in \partial h(u_\nu)$, where $\zeta_\nu$ represents an interface force, $u_\nu = u \cdot \nu$ the normal component of the displacement, and $\partial h(u_\nu)$ being the Clarke subdifferential of $h$ defined as below.
\begin{definition}\label{def:subdiffernetial}
Let $h : \Tilde{X} \rightarrow \mathbb{R}$ be a locally Lipschitz function. The generalized (Clarke) directional derivative of $h$ at $\Tilde{x} \in \Tilde{X}$ in the direction $v\in \Tilde{X}$, denoted $h^\circ(\Tilde{x};v)$, is defined by
\begin{equation*}
	h^\circ(\Tilde{x};v) = \limsup_{\Tilde{y}\rightarrow \Tilde{x}, \ \Tilde{\epsilon} \downarrow 0} \frac{h(\Tilde{y} + \Tilde{\epsilon} v) - h(\Tilde{y})}{\Tilde{\epsilon}}.
\end{equation*}
Moreover, the subdifferential in the sense of Clarke of $h$ at $\Tilde{x}$, denoted $\partial h(\Tilde{x})$, is a subset of $\Tilde{X}^\ast$ of the form
\begin{equation*}
	\partial h(\Tilde{x}) = \{ \zeta \in \Tilde{X}^\ast  :  h^\circ(\Tilde{x};v) \geq \inner{\zeta}{v}_{ \Tilde{X}^\ast \times  \Tilde{X}} \ \ \text{ for all } v\in \Tilde{X} \}.
\end{equation*}
\end{definition}
\begin{remark}
    To say that a function $h : \Tilde{X} \rightarrow \mathbb{R}$ is  locally Lipschitz on $\Tilde{X}$ means that $h(\Tilde{x})$ is Lipschitz continuous in a neighborhood of $\Tilde{x} \in \Tilde{X}$.
\end{remark}
\begin{remark}
We refer the reader to, e.g., \cite[p.185-187]{han} and \cite[Section 6.3]{Migorski2012} for examples related to contact mechanics. Other examples may be found in, e.g., \cite{clarke,Denkowski_theory}.
\end{remark}
\begin{prop}\label{prop:chainrule_subdiff}
	Let $\Tilde{X}$ be a Banach space and $h : \Tilde{X} \rightarrow \mathbb{R}$ be locally Lipschitz on $\Tilde{X}$.
	Then $(\Tilde{x} ,v) \mapsto h^\circ(\Tilde{x} ;v)$ is upper semicontinuous.
\end{prop}
\begin{prop}	\label{prop:convex_lsc_implies_locallyLipschitz}
	Let $\Tilde{X}$ be a Banach space.
	If $h: \Tilde{X} \rightarrow \mathbb{R}\cup\{\infty\}$ is proper, convex, and lower semicontinuous, then $h$ is locally Lipschitz on  the
	interior of the domain of $h$.
\end{prop}
Proposition \ref{prop:chainrule_subdiff} can be found in \cite[Proposition 5.6.6]{Denkowski_theory}, and 
 Proposition \ref{prop:convex_lsc_implies_locallyLipschitz} can be found in \cite[Proposition 5.2.10]{Denkowski_theory}.
\begin{definition}\label{def:convex_subdiffernetial}
	Let $h : \Tilde{X} \rightarrow \mathbb{R} \cup \{\infty\}$ be a proper and convex function. The (generally multivalued) mapping $\partial_c h : \Tilde{X} \rightarrow 2^{\Tilde{X}^\ast}$, written
\begin{equation*}
	\partial_c h(\Tilde{x} ) = \{ x^\ast \in \Tilde{X}^\ast  : h(v) -  h(\Tilde{x} ) \geq \inner{x^\ast}{v-\Tilde{x} }_{ \Tilde{X}^\ast \times  \Tilde{X}} \ \text{ for all } v\in \Tilde{X} \}
\end{equation*}
is called the convex subdifferential of $h$ at $\Tilde{x}  \in \Tilde{X}$.
\end{definition}
\begin{prop}
    Let  $h:\Tilde{X} \rightarrow \mathbb{R}$ be locally Lipschitz on $\Tilde{X}$. If $h$ is convex, then $\partial h$ coincides with $\partial_c h$.
\end{prop}
The proof of the above proposition is found in \cite[Proposition 2.2.7]{clarke}.
Lastly, to show that $(w,\alpha)$ is indeed a solution to Problem \ref{prob:fullproblem} (see Section \ref{sec:problem_and_mainresult}), we require the following result found in \cite[Lemma 3.43]{Migorski2012}.
\begin{lem}\label{lemma:phicontinuous}
	Let $\Tilde{X}$ and $\Tilde{Y}$ be two Banach spaces and $h :  \Tilde{X} \times \Tilde{Y} \rightarrow \mathbb{R}$ be such that
	\begin{enumerate}[labelindent=0pt,labelwidth=\widthof{\ref{last-item}},label=(\arabic*),itemindent=1em,leftmargin=!]
		\item $h(\cdot,\Tilde{y})$ is continuous on $\Tilde{X}$ for all $\Tilde{y} \in \Tilde{Y}$.
		\label{list:phicont1}
		\item  $h(\Tilde{x},\cdot)$ is locally Lipschitz on $\Tilde{Y}$ for all $\Tilde{x} \in \Tilde{X}$.
		\label{list:phicont2}
		\item There exists $c > 0$ such that for all $\zeta \in \partial h(\Tilde{x},\Tilde{y})$ we have
		\begin{align*}
			\norm{\zeta}_{\Tilde{Y}^\ast} \leq c (1 + \norm{\Tilde{x}}_{\Tilde{X}} + \norm{\Tilde{y}}_{\Tilde{Y}}),
		\end{align*}
		where $\partial h$ denotes the generalized gradient of $h(\Tilde{x},\cdot)$.
		\label{list:phicont3}
	\end{enumerate}
	Then $h$ is continuous on $ \Tilde{X} \times \Tilde{Y}$.
\end{lem}
To read more on the generalized directional derivatives, subdifferential, and nonsmooth analysis see, e.g., \cite[Chapter 2]{clarke}, \cite[Chapter 5]{Denkowski_theory}, and \cite[Chapter 1-3]{Hu_application}.

\section{Problem statement and main result}\label{sec:problem_and_mainresult}
In this section, we first introduce the problem and then present the main result.
\subsection{Problem statement}
Let $(V,H,V^\ast)$ be an evolution triple, and $U$, $X$, $Y$, $Z$ real separable reflexive Banach spaces, with the other function spaces defined in Section \ref{sec:bochner_spaces}. We only seek a solution of \eqref{eq:prob} in the sense of Definition \ref{def:sols}. We are therefore interested in the following evolutionary differential variational-hemivariational inequality:
\begin{prob}\label{prob:fullproblem}
	Find $w\in \mathcal{W}^{1,2}_T$ and $\alpha \in C([0,T];Y)$ such that
\begin{subequations}
	\begin{align}\label{eq:full_nonlinear_problem_1}
		&\alpha(t) = \alpha_0 + \int_0^t \mathcal{G}(s,\alpha(s),Mw(s))ds,\\
		&\inner{\Dot{w}(t) + A(t, w(t))- f(t) + \mathcal{R}w(t) }{v- w(t) }  
		+ \varphi(t, \alpha(t),\mathcal{S}_\varphi  w(t), Mw(t), Kv)  \label{eq:full_nonlinear_problem_2} \\
		&-  \varphi(t,  \alpha(t),\mathcal{S}_\varphi  w(t), Mw(t), Kw(t)) + j^\circ(t,\alpha(t), \mathcal{S}_jw(t),Nw(t); Nv -Nw(t)) \geq 0
		\notag
	\end{align}
	for all $v\in V$, a.e. $t\in (0,T)$ with
	\begin{equation}\label{eq:intialdata_originalprob}
		w(0) = w_0.
	\end{equation}
\end{subequations}
\end{prob}
We require the following assumptions on the operators and data:
\\
$\underline{\H{A}}$: \label{assumptionA}  $A : (0,T) \times V \rightarrow V^\ast \text{ is such that }$
\begin{enumerate}[labelindent=0pt,labelwidth=\widthof{\ref{last-item}},label=(\roman*),itemindent=1em,leftmargin=!]
	\item $A(\cdot, v)$ is measurable on $(0,T)$ for all $v \in V$.
	\label{list:A_measurable}
	\item $ A(t,\cdot) \text{ is demicontinuous on } V \text{ for a.e. } t \in (0,T)$, i.e., if $v^n \rightarrow v$ strongly in $V$, then $A(t,v^n) \rightarrow A(t,v)$ weakly in $V^\ast$ as $n \rightarrow \infty$ for a.e. $t\in (0,T)$.
	\label{list:A_demicont}
	\item $\norm{A(t,v)}_{V^\ast} \leq a_0(t) + a_1 \norm{v}_V$ for all  $v\in V$, a.e. $t\in (0,T)$ with  $a_0 \in L^2(0,T)$, $a_0  \geq 0, \ a_1 \geq 0$. \label{list:A_bounded} 
	\item There is a $m_A >0$ such that  $\inner{A(t,v_1) - A(t,v_2)}{v_1 - v_2} \geq m_A \norm{v_1 - v_2 }_V^2$ for all  $v_i\in V$, $i=1,2$, a.e. $t\in (0,T)$. \label{list:A_maximalmonotone}
\end{enumerate}
$\underline{\H{\varphi}}$: \label{assumptionphi}  $ \varphi :  (0,T) \times Y \times  X \times U \times Z \rightarrow \mathbb{R} \text{ is such that }$
\begin{enumerate}[labelindent=0pt,labelwidth=\widthof{\ref{last-item}},label=(\roman*),itemindent=1em,leftmargin=!]
\item $\varphi(\cdot,y,z,\Tilde{w},\Tilde{v}) \text{ is measurable on } (0,T) \text{ for all } y\in Y, \ z\in X, \ \Tilde{w} \in U, \ \Tilde{v}\in Z$.
\label{list:phi_measurable}
\item $ \varphi(t,\cdot,\cdot,\cdot,\Tilde{v}) \text{ is continuous on } Y \times X \times U \text{ for all } \Tilde{v}\in Z,  \text{ a.e. } t \in (0,T)$.
\label{list:phi_cont}
\item $\varphi(t,y,z,\Tilde{w},\cdot) \text{ is convex and lower semicontinuous on } Z \text{ for all}$ $y\in Y$, $z\in X$, $\Tilde{w}\in U$, a.e. $t\in (0,T)$.
\label{list:phi_convex_lsc}
\item $\norm{\partial_c \varphi (t,y,z,\Tilde{w},\Tilde{v})}_{Z^\ast}\leq c_{0\varphi} (t) +  c_{1\varphi} \norm{y}_Y  + c_{2\varphi} \norm{z}_X + c_{3\varphi} \norm{\Tilde{w}}_U + c_{4\varphi} \norm{\Tilde{v}}_Z$ $\text{ for all }$ $y\in Y, \ z\in X, \ \Tilde{w}\in U, \ \Tilde{v}\in Z, \text{ a.e. } t\in (0,T)$, $\text{ with } c_{0\varphi} \in L^2(0,T),\text{ and } c_{0\varphi}, c_{1\varphi},  c_{2\varphi}, c_{3\varphi}, c_{4\varphi} \geq 0.$
\label{list:phi_bounded}
\item $\text{There are $\beta_{i\varphi} \geq 0$ for $i=1,\dots,7$ such that }$ 
\begin{align*}
	& \varphi (t,y_1,z_1,\Tilde{w}_1, \Tilde{v}_2) - \varphi (t, y_1,z_1,\Tilde{w}_1, \Tilde{v}_1) +\varphi (t,y_2,z_2,\Tilde{w}_2, \Tilde{v}_1)  - \varphi (t, y_2, z_2,\Tilde{w}_2, \Tilde{v}_2) \\
	&\leq \beta_{1\varphi} \norm{\Tilde{w}_1}_U \norm{y_1 - y_2}_Y\norm{\Tilde{v}_1 - \Tilde{v}_2 }_Z 
    +\beta_{2\varphi}\norm{y_1 - y_2}_Y\norm{\Tilde{v}_1 - \Tilde{v}_2 }_Z
    \\
    &+ \beta_{3\varphi} \norm{z_1 - z_2}_X\norm{\Tilde{v}_1 - \Tilde{v}_2 }_Z  +  \beta_{4\varphi} \norm{\Tilde{w}_1 - \Tilde{w}_2}_U\norm{\Tilde{v}_1 - \Tilde{v}_2 }_Z \\
    &
    + \beta_{5\varphi} \norm{y_2}_Y\norm{\Tilde{w}_1 - \Tilde{w}_2}_U\norm{\Tilde{v}_1 - \Tilde{v}_2}_Z
+ \beta_{6\varphi} \norm{\Tilde{w}_1}_U\norm{z_1-z_2}_X\norm{\Tilde{v}_1 - \Tilde{v}_2}_Z 
    \\
    &+ \beta_{7\varphi}  \norm{y_1}_Y\norm{z_1-z_2}_X \norm{\Tilde{v}_1 - \Tilde{v}_2}_Z 
\end{align*}
$\text{for all } y_i \in Y$, $z_i \in X$, $\Tilde{w}_i \in U$, $\Tilde{v}_i \in Z$, $i=1,2$,  a.e. $t\in (0,T)$.
\label{list:phi_estimate}
\end{enumerate}
$\underline{\H{j}}$: \label{assumptionj}  $j : (0,T) \times Y \times X \times X \rightarrow \R \text{ is such that }$
\begin{enumerate}[labelindent=0pt,labelwidth=\widthof{\ref{last-item}},label=(\roman*),itemindent=1em,leftmargin=!]
\item $j(\cdot, y,z, \Tilde{v}) \text{ is measurable on } (0,T) \text{ for all } y\in Y, z,\Tilde{v} \in X$.
\label{list:j_measurable}
\item $j(t,y,z,\cdot) \text{ is locally Lipschitz on } X \text{ for all } y\in Y, z\in X, \text{ a.e. } t \in (0,T)$.
\label{list:j_locallyLipschitz}
\item For $\{y_n\}_{n\geq 1} \subset Y$ such that $y_n \rightarrow y$ strongly in $Y$, $\{z_n\}_{n\geq 1}  \subset X$ such that $z_n \rightarrow z$ strongly in $X$, $\{\Tilde{v}_n\}_{n\geq 1} \subset X$ satisfying $\Tilde{v}_n \rightarrow \Tilde{v}$ strongly in $X$, and $\{\Tilde{u}_n\}_{n\geq 1} \subset X$ such that $\Tilde{u}_n \rightarrow \Tilde{u}$ strongly in $X$, when $n\rightarrow \infty$, we have
\begin{equation*}
    \limsup_{n\rightarrow \infty} j^\circ(t, y_n,z_n, \Tilde{v}_n;\Tilde{u}_n) \leq j^\circ(t, y, z, \Tilde{v};\Tilde{u})
\end{equation*}
for a.e. $t\in (0,T)$. \label{list:j_convergence}
\item 	$ \norm{\partial j (t,y,z,\Tilde{v})}_{X^\ast}\leq c_{0j} (t) + c_{1j}\norm{y}_Y + c_{2j}\norm{z}_X  + c_{3j} \norm{\Tilde{v}}_X \text{ for all } y\in Y,  z,\Tilde{v}\in X$, $\text{ a.e. } t\in (0,T)$ with  $c_{0j} \in L^2(0,T)$ and $c_{0j},  c_{1j} , c_{2j}, c_{3j} \geq 0$.
\label{list:j_bounded}
\item $\text{There are $m_j,\Bar{m}_j \geq 0$ such that }$ 
$j^\circ(t,y_1,z_1,\Tilde{v}_1; \Tilde{v}_2 - \Tilde{v}_1) + j^\circ(t,y_2,z_2,\Tilde{v}_2; \Tilde{v}_1 - \Tilde{v}_2) \leq m_j \norm{\Tilde{v}_1 - \Tilde{v}_2 }_X^2 + \Bar{m}_j (\norm{y_1 - y_2 }_Y + \norm{z_1 - z_2 }_X)\norm{\Tilde{v}_1 - \Tilde{v}_2 }_X$ $\text{for all } y_i\in Y, \ z_i,\Tilde{v}_i \in X,\ i=1,2,  \text{ a.e. } t\in (0,T)$.
\label{list:j_estimate}
\end{enumerate}
We note that hypothesis \hyperref[assumptionj]{$\H{j}$}\ref{list:j_estimate} is equivalent to
\begin{align}\label{eq:jeq}
    &\inner{z^\ast_1- z^\ast_2}{\Tilde{v}_1 - \Tilde{v}_2}_{X^\ast \times X}\\ &\geq - m_j \norm{\Tilde{v}_1-\Tilde{v}_2}_X 
    - \bar{m}_j (\norm{y_1 - y_2}_Y + \norm{z_1 - z_2}_X)  \norm{\Tilde{v}_1-\Tilde{v}_2}_X \notag
\end{align}
for all $z^\ast_i \in \partial j(t,y_i,z_i,\Tilde{v}_i)$, $y_i \in Y$, $z_i,\Tilde{v}_i \in X$, $i=1,2$, a.e. $t\in (0,T)$ with $m_j \geq 0$, $\bar{m}_j \geq 0$. With small modifications, the equivalence follows by the proof in \cite[Lemma 7, p.124]{sofonea2017}.\\
\\
$\underline{\H{\mathcal{R} }}$: \label{assumptionR}  $ \mathcal{R}   :L^2_TV  \rightarrow L^2_TV^\ast$ is such that
\begin{enumerate}[labelindent=0pt,labelwidth=\widthof{\ref{last-item}},label=(\roman*),itemindent=1em,leftmargin=!]
	\item $\mathcal{R}  \text{ is a history-dependent operator, i.e., }$
	\begin{align*}
		\norm{ \mathcal{R}   v_1(t) -  \mathcal{R}  v_2 (t)}_{V^\ast} \leq c_{\mathcal{R}} \int_0^t \norm{v_1(s)-v_2(s)}_V ds
	\end{align*}
	$\text{ for all  }v_i \in L^2_TV$, $i=1,2$, a.e. $t\in (0,T) \text{ with } c_{\mathcal{R} } >0$. 
	\label{list:S_hist}
	\item $\mathcal{R}0$ belongs to a bounded subset of $L^2_TV^\ast$. 
	\label{list:R0}
\end{enumerate}
$\underline{\H{\mathcal{S}_\varphi }}$: \label{assumptionS1}  $ \mathcal{S}_\varphi   :L^2_TV  \rightarrow L^2_TX$ is such that
\begin{enumerate}[labelindent=0pt,labelwidth=\widthof{\ref{last-item}},label=(\roman*),itemindent=1em,leftmargin=!]
	\item $\mathcal{S}_\varphi  \text{ is a history-dependent operator, i.e., }$
	\begin{align*}
		\norm{ \mathcal{S}_\varphi   v_1(t) -  \mathcal{S}_\varphi  v_2 (t)}_X \leq c_{\mathcal{S}_\varphi } \int_0^t \norm{v_1(s)-v_2(s)}_V ds
	\end{align*}
	$\text{ for all  }v_i \in L^2_TV$, $i=1,2$, a.e. $t\in (0,T) \text{ with } c_{\mathcal{S}_\varphi } >0$. 
	\label{list:S_hist11}
	\item $\mathcal{S}_\varphi0$ belongs to a bounded subset of $L^2_TX$. 
	\label{list:S011}
\end{enumerate}
$\underline{\H{\mathcal{S}_j}}$: \label{assumptionS2}  $ \mathcal{S}_j  :L^2_TV  \rightarrow L^2_TX$ is such that
\begin{enumerate}[labelindent=0pt,labelwidth=\widthof{\ref{last-item}},label=(\roman*),itemindent=1em,leftmargin=!]
	\item $\mathcal{S}_j \text{ is a history-dependent operator, i.e., }$
	\begin{align*}
		\norm{ \mathcal{S}_j  v_1(t) -  \mathcal{S}_j v_2 (t)}_X \leq c_ {\mathcal{S}_j} \int_0^t \norm{v_1(s)-v_2(s)}_V ds
	\end{align*}
	$\text{ for all  }v_i \in L^2_TV$, $i=1,2$, a.e. $t\in (0,T) \text{ with } c_{\mathcal{S}_j} >0$. 
	\label{list:S_hist22}
	\item $\mathcal{S}_j0$ belongs to a bounded subset of $L^2_TX$. 
	\label{list:S022}
\end{enumerate}
$\underline{\H{\mathcal{G}}}$: \label{assumptionG}  $ \mathcal{G} : (0,T) \times Y \times U \rightarrow Y \text{ is such that }$
\begin{enumerate}[labelindent=0pt,labelwidth=\widthof{\ref{last-item}},label=(\roman*),itemindent=1em,leftmargin=!]
\item $\mathcal{G}(\cdot,\alpha, \Tilde{v}) \text{ is measurable on } (0,T) \text{ for all } \alpha \in Y, \Tilde{v} \in U$.
\label{list:G_measurable}
\item There exists an $L_\mathcal{G}>0$ such that
\begin{align*}
	\norm{ \mathcal{G}(t,\alpha_1(t),\Tilde{v}_1(t)) - \mathcal{G}(t,\alpha_2(t), \Tilde{v}_2(t))}_Y \leq L_\mathcal{G} \Big(\norm{\alpha_1(t) - \alpha_2(t) }_Y +  \norm{\Tilde{v}_1(t) - \Tilde{v}_2(t) }_U \Big)
\end{align*}
$\text{for all } \alpha_i \in Y$, $\Tilde{v}_i \in U$, $i=1,2$,   a.e.  $t\in (0,T).$
\label{list:G_Lipschitz}
\item $\norm{\mathcal{G}(\cdot,0,0)}_{L^\infty_TY} < \infty$. 
\label{list:G_00}
\end{enumerate}
$\underline{\H{MNK}}$: \label{assumptionMNK} $M \in \mathcal{L}(V,U), \ \ N \in \mathcal{L}(V,X), \ \ K \in \mathcal{L}(V,Z)$.
\\
\\
\indent
We also assume the following regularity on the source term and initial data:
\begin{subequations}\label{eq:initaldata}
	\begin{align}\label{eq:initaldata1}
		f \in L^2_TV^\ast,& \ \ \ \ w_0 \in V,\\
		\alpha_0 &\in Y.
		\label{eq:initaldata2}
	\end{align}
\end{subequations}

Lastly, we require the following smallness-condition:
\begin{equation}\label{eq:assumptionbound_mA_max}
	m_A > 
	m_j\norm{N}^2  + \sqrt{2}(\beta_{4\varphi}+\beta_{5\varphi}\norm{\alpha_0}_{Y})\norm{K}\norm{M}.
\end{equation}
\begin{remark}
	Similar assumptions can be found in, e.g., \cite{Patrulescu2017,Migorski2022,Migorski2019,shillor2004,sofonea2017}. The same type of condition as \hyperref[assumptionj]{$\H{j}$}\ref{list:j_convergence} is found in, e.g., \cite{Migorski2022_2,Migorski2023}. If $j^0$ is independent of $\alpha$ and $\mathcal{S}_jw$ in \eqref{eq:w1}, we may relax the assumption \hyperref[assumptionj]{$\H{j}$}\ref{list:j_convergence}, and \hyperref[assumptionj]{$\H{j}$}\ref{list:j_locallyLipschitz} and Proposition \ref{prop:chainrule_subdiff} are enough.
\end{remark}
\begin{remark}\label{remark:mubounded}
    If $\beta_{1\varphi} =\beta_{4\varphi} =\beta_{5\varphi} = \beta_{6\varphi}=   \beta_{7\varphi} = 0$ in \hyperref[assumptionphi]{$\H{\varphi}$}\ref{list:phi_estimate}, then the friction coefficient $\mu$ is bounded. Consequently, Problem \ref{prob:fullproblem} reduces to the one found in \cite{Migorski2022}. In the quasi-static setting, taking $\beta_{1\varphi} = \beta_{5\varphi} = \beta_{6\varphi}=   \beta_{7\varphi} = 0$, the problem is covered in \cite{Patrulescu2017}. Lastly, taking $\beta_{3\varphi} = \beta_{6\varphi}=   \beta_{7\varphi} = 0$ reduces Problem \ref{prob:fullproblem} to the one found in \cite{Pipping2015,Pipping2019} for the first-order approximation of \eqref{eq:aginglaw} and \eqref{eq:regularized} introduced in Section \ref{sec:rateandstate}. 
\end{remark}

We make a brief remark on the assumptions in Appendix \ref{appendix:comments_app}.

\subsection{Main result}
We will now state the main result, i.e., Theorem \ref{thm:mainresult}; the first part is an existence and uniqueness result, and the latter provides that the flow map depends continuously on the initial data. The proof of Theorem \ref{thm:mainresult} is deferred to Section \ref{sec:proof_main_result} after the preparation in Section \ref{sec:preliminary}.
\begin{thm}\label{thm:mainresult}
	Assume \hyperref[assumptionA]{$\H{A}$},
	\hyperref[assumptionphi]{$\H{\varphi}$},
	\hyperref[assumptionj]{$\H{j}$},   \hyperref[assumptionR]{$\H{\mathcal{R}}$}, \hyperref[assumptionS1]{$\H{\mathcal{S}_\varphi}$}, \hyperref[assumptionS2]{$\H{\mathcal{S}_j}$},
	\hyperref[assumptionG]{$\H{\mathcal{G}}$},
	 \hyperref[assumptionMNK]{$\H{MNK}$}, and \eqref{eq:initaldata}-\eqref{eq:assumptionbound_mA_max} holds. 
	\begin{enumerate}[labelindent=0pt,labelwidth=\widthof{\ref{last-item}},label=(\alph*),itemindent=1em,leftmargin=1em]
		\item Then there exists a $T>0$ satisfying
        \begin{subequations}\label{eq:times_T}
            \begin{align}
                T= T(\norm{(f,w_0,\alpha_0)}_{L^2_T V^\ast \times V \times Y})
            \end{align}
            and
            \begin{align}
                T(a) > T(b) \quad   \text{if }\quad  a< b  
            \end{align}
        \end{subequations}
        so that $w\in \mathcal{W}^{1,2}_T \subset C([0,T];H)$ and $\alpha \in C([0,T];Y)$ is a unqiue solution to Problem \ref{prob:fullproblem}. 
		\item Moreover, there exists a neighborhood around $(w_0 , \alpha_0)$ so that the flow map 
            $F : V \times Y \rightarrow L^2_{T/2}V \times C([0,T/2];Y)$ defined by $(w_0 , \alpha_0) \mapsto (w,\alpha)$  is continuous. 
		\label{list:continuous_dependence_on_inital_data}
            \item If $\beta_{1\varphi} =\beta_{4\varphi} =\beta_{5\varphi} = \beta_{6\varphi}=   \beta_{7\varphi} = 0$ in \hyperref[assumptionphi]{$\H{\varphi}$}\ref{list:phi_estimate}, we obtain global time of existence, i.e., the existence of a solution holds for any finite time $T>0$. \label{list:global}
	\end{enumerate}
\end{thm}

\begin{remark}
	The theorem can easily be extended to include more than three history-dependent operators without needing any additional assumptions other than the once put on $\mathcal{R}$, $\mathcal{S}_\varphi$, and $\mathcal{S}_j$, i.e.,  \hyperref[assumptionR]{$\H{\mathcal{R}}$}, \hyperref[assumptionS1]{$\H{\mathcal{S}_\varphi}$}, and \hyperref[assumptionS2]{$\H{\mathcal{S}_j}$}, respectively.
\end{remark}
\subsection{Strategy of the proof of Theorem \ref{thm:mainresult}}
The proof of the theorem is divided into six steps. In the first step, we introduce an auxiliary problem to Problem \ref{prob:fullproblem}, calling this Problem \ref{prob:first_step}.
Specifically, we fix five of the functions in \eqref{eq:full_nonlinear_problem_2} and leave \eqref{eq:full_nonlinear_problem_1} intact. We recast the auxiliary problem as a differential inclusion (introduced in Section \ref{sec:preliminary}) and use existing results to prove that Problem \ref{prob:first_step} has a unique solution (Step \ref{linearization}-\ref{estimate1}). Next, we define an iterative scheme for Problem  \ref{prob:fullproblem} using Problem \ref{prob:first_step}. This iterative scheme decouples  \eqref{eq:full_nonlinear_problem_1} and Problem \ref{prob:first_step} at each step. Then, we study the difference between two successive iterates and show that these iterates are Cauchy sequences. We then pass to the limit to show that the iterative scheme converges to Problem \ref{prob:fullproblem} (Step \ref{convergence}). Finally, we show that the flow map continuously depends on the initial data (Step \ref{continuous}).

\section{Preliminary result}\label{sec:preliminary}
\noindent
Before proving Theorem \ref{thm:mainresult}, we present an existence and uniqueness result for a differential inclusion problem see, e.g., \cite{Aubin1984}. The forthcoming result will be used to prove existence of a solution to an auxiliary problem of \eqref{eq:full_nonlinear_problem_2}-\eqref{eq:intialdata_originalprob} in Problem \ref{prob:fullproblem}. To utilize this result, we need to introduce a differential inclusion which we relate to the auxiliary problem of \eqref{eq:full_nonlinear_problem_2}-\eqref{eq:intialdata_originalprob}. This will be made clear in Step \ref{linearization}-\ref{exunique} in the proof of Theorem \ref{thm:mainresult}.
\\
\indent
We begin by introducing the inclusion problem.
\begin{prob}\label{prob:inclusion}
Find $w\in \mathcal{W}^{1,2}_T$ such that
\begin{align*}
		\Dot{w}(t) &+ A(t,w(t)) + \partial \psi (t,w(t)) \ni f(t)  \ \text{ for a.e. } \ t\in (0,T),\\
		w(0) &= w_0.
	\end{align*}
\end{prob}
For clarity on how to work with the preceding problem, we include the following definition:
\begin{definition}\label{def:def_inclusion}
A function $w\in \mathcal{W}^{1,2}_T$ is called a solution to Problem \ref{prob:inclusion} if there exists $w^\ast \in L^2_TV^\ast$ such that
	\begin{align*}
		&\Dot{w}(t) + A(t,w(t)) + w^\ast(t) = f(t)   \ \ \text{ for a.e. } \ t\in (0,T),\\
		&w^\ast(t) \in \partial \psi (t,w(t)) 
	\end{align*}
  for a.e.  $t\in (0,T)$ with
  \begin{equation*}
      w(0) = w_0.
  \end{equation*}
\end{definition}
In the preliminary existence and uniqueness result, we consider the following assumptions:
\\
$\underline{\H{\psi}}$: \label{assumptionpsi} $\psi : (0,T) \times V \rightarrow \mathbb{R} \text{ is such that }$
\begin{enumerate}[labelindent=0pt,labelwidth=\widthof{\ref{last-item}},label=(\roman*),itemindent=1em,leftmargin=!]
\item $\psi(\cdot, v) $ is measurable on $(0,T)$ for all $v \in V$.
\label{list:psi_measurable}
\item $ \psi(t,\cdot) \text{ is locally Lipschitz on } V \text{ for a.e. } t \in (0,T)$.
\label{list:psi_locallyLipschitz}
\item 	$\norm{\partial \psi (t,v)}_{V^\ast}\leq c_0 (t) + c_1 \norm{v}_V \text{ for all } v\in V$, $\text{ a.e. } t\in (0,T) \text{ with } c_0 \in L^2(0,T)$, $c_0 \geq 0, \ c_1 \geq 0$.\label{list:psi_bounded}
\item $\text{There is a $m_\psi \geq 0$ such that }$ 
$ \inner{z_1 - z_2}{v_1 - v_2} \geq -m_\psi \norm{v_1 - v_2 }_V^2$ $\text{ for all } z_i \in \partial \psi(t,v_i)$, $z_i \in V^\ast, \ v_i\in V, \ i=1,2,  \text{ a.e. } t\in (0,T)$.
\label{list:psi_maximalmonotone}
\end{enumerate}
We further assume that the operator $A:(0,T) \times V \rightarrow V^\ast$ satisfies \hyperref[assumptionA]{$\H{A}$}, and the source term $f$ and the initial data $w_0$ satisfy \eqref{eq:initaldata1}. Additionally, we assume that the following smallness-condition holds: 
\begin{equation}\label{eq:assumptionbound_on_mA}
	m_A >m_\psi.
\end{equation}

\begin{thm}\label{thm:inclusionsol}
    Assume that \hyperref[assumptionA]{$\H{A}$}, \hyperref[assumptionpsi]{$\H{\psi}$},  \eqref{eq:initaldata1}, and \eqref{eq:assumptionbound_on_mA} hold. 
    Then  Problem \ref{prob:inclusion} has a unique solution $w\in \mathcal{W}^{1,2}_T$ in the sense of Definition \ref{def:def_inclusion} for any $T>0$.
\end{thm}
The theorem was proved in \cite[Theorem 3]{Migorski2019}. This result will be used in Step \ref{exunique} of the proof of Theorem \ref{thm:mainresult}.

\section{Proof of Theorem \ref{thm:mainresult}}\label{sec:proof_main_result}
With the preparation in Section \ref{sec:func_nonsmooth}-\ref{sec:preliminary}, we proceed to the proof of Theorem \ref{thm:mainresult}. For the convenience of the reader, the proof is established in several steps, and some of the proofs have been moved to the appendix. We recall that the function spaces are defined in Section \ref{sec:bochner_spaces}.
\\
\indent 
\textbf{Step \ref{linearization}}\rtask{linearization} \textit{(Auxiliary problem to the evolutionary hemivariational-variational inequality \eqref{eq:full_nonlinear_problem_2}-\eqref{eq:intialdata_originalprob})}.
Let $(\alpha,\xi, \eta, g,\chi) \in  C([0,T];Y) \times L^2_TV^\ast  \times L^2_TX \times L^2_TV \times L^2_TX$ be given, then we define an auxiliary problem to \eqref{eq:full_nonlinear_problem_2}-\eqref{eq:intialdata_originalprob} in Problem \ref{prob:fullproblem}.
\begin{prob}\label{prob:first_step}
Find $w_{\alpha\xi\eta g\chi} \in \mathcal{W}^{1,2}_T$ corresponding to $(\alpha,\xi,\eta, g,\chi) \in C([0,T];Y) \times L^2_TV^\ast  \times L^2_TX \times L^2_TV\times L^2_TX$ such that
\begin{align*}
	&\inner{\Dot{w}_{\alpha\xi\eta g\chi}(t) + A(t, w_{\alpha\xi\eta g\chi}(t)) - f(t) + \xi(t) }{v- w_{\alpha\xi\eta g\chi}(t) } + \varphi(t,\alpha(t), \eta(t), Mg(t), Kv)  \\
	&-  \varphi(t,\alpha(t), \eta(t), Mg(t), Kw_{\alpha\xi\eta g\chi}(t)) +j^\circ(t,\alpha(t), \chi(t), Nw_{\alpha\xi\eta g\chi}(t); Nv -Nw_{\alpha\xi\eta g\chi}(t) ) \geq 0
\end{align*}
for all $v\in V$, a.e. $t\in (0,T)$ with
\begin{equation*}
w_{\alpha\xi\eta g\chi} (0) = w_0.
\end{equation*}
\end{prob}
\begin{remark}
A glance at Problem \ref{prob:first_step} and \eqref{eq:full_nonlinear_problem_2} lets us see that the auxiliary problem keeps $\xi= \mathcal{R}w$, $\alpha$ (still denoted by $\alpha$), $\eta = \mathcal{S}_\varphi  w$, $g=w$, and $\chi = \mathcal{S}_j w$  known in contrast to \eqref{eq:full_nonlinear_problem_2}. We find it worth mentioning that we use the subscripts on $w$ to emphasize that a solution $w_{\alpha\xi\eta g\chi}$ to Problem \ref{prob:first_step} corresponds to $(\alpha,\xi,\eta, g,\chi) \in  C([0,T];Y) \times L^2_TV^\ast \times L^2_TX \times L^2_TV \times L^2_TX$. This also helps to distinguish between a solution to Problem \ref{prob:fullproblem} and a solution to Problem \ref{prob:first_step}.
\end{remark}
\textbf{Step \ref{exunique}}\rtask{exunique} \textit{(Existence of a solution to Problem \ref{prob:first_step})}.
Let $(\alpha,\xi,\eta, g,\chi) \in C([0,T];Y) \times L^2_TV^\ast  \times L^2_TX \times L^2_TV\times L^2_TX$ be given.
We wish to utilize Theorem \ref{thm:inclusionsol} in order to prove that Problem \ref{prob:first_step} has a solution. We therefore define the functional $\psi_{\alpha\xi\eta g\chi} : (0,T) \times V \rightarrow \mathbb{R}$ by
\begin{align}\label{eq:psixietag}
	\psi_{\alpha\xi\eta g\chi}(t,v) &=  \inner{\xi(t)}{v} + \varphi(t,\alpha(t),\eta(t),Mg(t),Kv) + j(t,\alpha(t),\chi(t),Nv)
\end{align}
for all $v\in V$, a.e. $t\in (0,T)$. 
Verification of the hypothesis of Theorem \ref{thm:inclusionsol} follows from the same approach as the first part of the proof in \cite[Theorem 5]{Migorski2019}. We investigate the assumption \hyperref[assumptionpsi]{$\H{\psi}$} and the smallness-condition \eqref{eq:assumptionbound_on_mA}, as there are some modifications in comparison to \cite[Theorem 5]{Migorski2019}. Keeping Proposition \ref{prop:convex_lsc_implies_locallyLipschitz} in mind, we  only comment on the changes and leave the reader to visit \cite[Theorem 5]{Migorski2019} for a detailed verification. Using \eqref{eq:jeq}, we find that \hyperref[assumptionpsi]{$\H{\psi}$} holds with $c_0(t) =  \norm{\xi(t)}_{V^\ast} + c_{0j} (t)\norm{N} +c_{2j} \norm{N} \norm{\chi(t)}_X + c_{0\varphi} (t)\norm{K} +  c_{2\varphi}\norm{K} \norm{\eta(t)}_X + (c_{1j}\norm{N}  + c_{1\varphi} \norm{K}) \norm{\alpha(t)}_Y + c_{3\varphi}\norm{K}\norm{M} \norm{g(t)}_V$,
    $c_1 = c_{3j}\norm{N}^2+ c_{4\varphi}\norm{K}^2$, and
   $m_\psi = m_j \norm{N}^2$. 
This, together with the smallness-condition \eqref{eq:assumptionbound_mA_max} leads to  \eqref{eq:assumptionbound_on_mA}. Thus, we conclude by Theorem \ref{thm:inclusionsol} that there exists a solution $w_{\alpha\xi\eta g\chi} \in \mathcal{W}^{1,2}_T$ of Problem \ref{prob:inclusion} with $\psi_{\alpha\xi\eta g\chi}$ defined in \eqref{eq:psixietag}. It remains to show that the existence of a solution to Problem \ref{prob:inclusion} implies the existence of a solution to Problem \ref{prob:first_step}. This is a consequence of Definition \ref{def:subdiffernetial} and \ref{def:convex_subdiffernetial}, and basic results of the generalized gradients, see, e.g., \cite[Theorem 3.7, Proposition 3.10-3.12]{han}, where they have summarized these properties, and \cite[Lemma 7, p.124]{sofonea2017}. A more detailed approach to this part can be found in, e.g., \cite[Section 6]{han} or \cite[p.190-192]{sofonea2017}.

\textbf{Step \ref{unique}}\rtask{unique} \textit{(Uniqueness of a solution to Problem \ref{prob:first_step})}.
Uniqueness is immediate from the proof in \cite[Theorem 98]{sofonea2017} with $\alpha_j = m_j\norm{N}^2$ and the smallness-condition \eqref{eq:assumptionbound_mA_max}, i.e., $m_A > m_j\norm{N}^2$.

\textbf{Step \ref{estimate1}}\rtask{estimate1} \textit{(Estimate on the solution to Problem \ref{prob:first_step}, that is, $w_{\alpha\xi\eta g\chi} \in \mathcal{W}^{1,2}_T \subset C([0,T];H)$)}.
We now find an estimate on the solution to Problem \ref{prob:first_step}, which will come in handy later:
\begin{prop}\label{prop:estimatewfirst}
	Under the assumptions of Theorem \ref{thm:mainresult}, for given $(\alpha,\xi,\eta,g,\chi) \in$ \\$C([0,T];Y)  \times L^2_TV^\ast \times L^2_TX\times L^2_TV  \times L^2_TX$, 
	let $w_{\alpha\xi\eta g\chi}$ be a solution to Problem \ref{prob:first_step}. Then, there exists a constant $c>0$ independent of $w_{\alpha\xi\eta g\chi}$ such that
 \begin{subequations}
     \begin{align}\label{eq:estwwww}
		&\norm{w_{\alpha\xi\eta g\chi}}_{L^\infty_TH}^2 + \norm{w_{\alpha\xi\eta g\chi}}_{L^2_TV}^2 \\ \notag
		&\leq c(1 +\norm{w_0}_V^2+  \norm{f}_{L^2_TV^\ast}^2 + \norm{\xi}_{L^2_TV^\ast}^2    
		 + \norm{\alpha}_{L^\infty_TY}^2 
           +  \norm{\eta}_{L^2_TX}^2
         +  \norm{\chi}_{L^2_TX}^2   )
        \\ \notag
        &+ \frac{\beta_{4\varphi} \norm{K}\norm{M} +\beta_{5\varphi} \norm{K}\norm{M} \norm{\alpha}_{L^\infty_TY}}{m_A  - m_j \norm{N}^2}
          \norm{g}_{L^2_TV}^2 \\
         &+  \frac{\beta_{4\varphi} \norm{K}\norm{M} +\beta_{5\varphi} \norm{K}\norm{M} \norm{\alpha}_{L^\infty_TY}}{m_A  - m_j \norm{N}^2} \norm{w_{\alpha\xi\eta g\chi}}_{L^2_TV}^2 \notag
    \end{align}
    and
    \begin{align}\label{eq:estwdot}
        \norm{\Dot{w}_{\alpha\xi\eta g\chi}}_{L^2_T V^\ast}^2
        &\leq c(1+  \norm{f}_{L^2_TV^\ast}^2 + \norm{\xi}_{L^2_TV^\ast}^2 
    		+\norm{\chi}_{L^2_TX}^2
    		+\norm{\alpha}_{L^\infty_TY}^2 +
               \norm{\eta}_{L^2_TX}^2 \\
               &+ \norm{g}_{L^2_TV}^2 + \norm{w_{\alpha\xi\eta g\chi}}_{L^2_TV}^2) + 2\beta_{5\varphi}^2 \norm{K}^2\norm{M}^2\norm{\alpha}_{L^\infty_TY}^2\norm{g}_{L^2_TV}^2 .\notag
    \end{align}
    \end{subequations}
\end{prop}
The proof of Proposition \ref{prop:estimatewfirst} is postponed to Appendix \ref{appendix:proof_part2}.

\textbf{Step \ref{convergence}}\rtask{convergence} \textit{(Scheme for the approximated solution to Problem \ref{prob:fullproblem})}.
For $n\in \mathbb{Z}_+$, let $\alpha^{n-1} \in C([0,T];Y)$, and  $w^{n-1} \in \mathcal{W}^{1,2}_T$ be known. We construct the approximated solutions $\\ \{(w^n,\alpha^n)\}_{n\geq 1} \subset \mathcal{W}^{1,2}_T \times C([0,T];Y)$ to Problem \ref{prob:fullproblem}, where $(w^{n},\alpha^n)$ is a solution of the scheme:
\begin{subequations}
\begin{align}\label{eq:algorithm11}
	&\inner{\Dot{w}^{n}(t) + A(t, 	w^{n}(t)) - f(t) + \mathcal{R}w^{n-1}(t) }{v- w^{n}(t) } \\ \notag
	&\hspace{1.1cm}+ \varphi(t,\alpha^{n-1}(t), \mathcal{S}_\varphi  w^{n-1}(t), 	Mw^{n-1}(t), Kv) \\ \notag
	&\hspace{1.1cm}-  \varphi(t,\alpha^{n-1}(t), \mathcal{S}_\varphi  w^{n-1}(t), Mw^{n-1}(t), Kw^{n}(t)) \\ 
	&\hspace{1.1cm}+ j^\circ(t,\alpha^{n-1}(t), \mathcal{S}_j w^{n-1}(t), Nw^{n}(t); Nv -Nw^{n}(t) )	  \geq 0 \notag
\end{align}
for all $v\in V$, a.e. $t \in (0,T)$, and 
\begin{equation}\label{eq:boundaryconditions_walphaxietag1} 
	w^{n}(0) = w_0.
\end{equation}
\begin{equation}\label{eq:algorithm12}
	\alpha^n (t) = \alpha_0 + \int_0^t \mathcal{G}(s,\alpha^n(s),Mw^n(s)) ds \ \ \text{ for a.e. } t\in (0,T),\\
\end{equation}
\end{subequations}
with $w^0 = w_0 \in V$ and $\alpha^0 = \alpha_0 \in Y$.
\\

\indent\textbf{Step \ref{convergence}.\ref{welldefinedscehem}}\rsubtask{welldefinedscehem} \textit{(Existence and uniqueness of $(w^n,\alpha^n) \in\mathcal{W}^{1,2}_T \times C([0,T];Y)$ to \eqref{eq:algorithm11}-\eqref{eq:algorithm12} for all $n\in \N$)}. We establish existence and uniqueness by induction on $n$. First, applying Minkowski's inequality, Young's inequality, integrating over the time interval $(0,t') \subset (0,T)$, and lastly applying the Cauchy-Schwarz inequality to hypothesis \hyperref[assumptionR]{$\H{\mathcal{R}}$}, \hyperref[assumptionS1]{$\H{\mathcal{S}_\varphi}$}, and \hyperref[assumptionS2]{$\H{\mathcal{S}_j}$}, respectively, yield
\begin{align}\label{eq:hist_est_R}
	\int_0^{t'}\norm{\mathcal{R}w^{n-1}(t)}_{V^\ast}^2dt 
	&\leq 2 \int_0^{t'}\norm{\mathcal{R}w^{n-1}(t) -\mathcal{R}0(t)}_{V^\ast}^2 dt + 2 \int_0^{t'} \norm{\mathcal{R}0(t)}_{V^\ast}^2 dt\\ \notag
	&\leq 2T^2c_\mathcal{R}^2 \int_0^{t'} \norm{w^{n-1}(t)}_{V}^2 dt + 2  \norm{\mathcal{R}0}_{L^2_{t'}V^\ast}^2
	\\
	\int_0^{t'}\norm{\mathcal{S}_\varphi w^{n-1}(t)}_X^2dt  &\leq 2T^2c_{\mathcal{S}_\varphi }^2 \int_0^{t'} \norm{w^{n-1}(t)}_{V}^2 dt + 2  \norm{\mathcal{S}_\varphi 0}_{L^2_{t'}X}^2
	\label{eq:hist_est_S1}\\
	\int_0^{t'}\norm{\mathcal{S}_jw^{n-1}(t)}_X^2dt  &\leq 2T^2c_{\mathcal{S}_j}^2 \int_0^{t'} \norm{w^{n-1}(t)}_{V}^2 dt + 2  \norm{\mathcal{S}_j0}_{L^2_{t'}X}^2
	\label{eq:hist_est_S2}
\end{align}
for all $t' \in [0,T]$. 
We combine  \eqref{eq:hist_est_R}-\eqref{eq:hist_est_S2} with the estimates \eqref{eq:estwwww}-\eqref{eq:estwdot} in Proposition \ref{prop:estimatewfirst} for $w_{\alpha\xi\eta g\chi} = w^n$, $\Dot{w}_{\alpha\xi\eta g\chi} = \Dot{w}^n$, $\xi = \mathcal{R}w^{n-1}$, $\alpha = \alpha^{n-1}$, $\eta = \mathcal{S}_\varphi  w^{n-1}$, $g=w^{n-1}$, and  $\chi = \mathcal{S}_j w^{n-1}$. This implies
\begin{subequations}
\begin{align}\label{eq:est_w_n}
      \norm{w^n}_{L^\infty_TH}^2 &+ \norm{w^n}_{\mathcal{W}^{1,2}_T}^2 
	\\
        &\leq c(1 +\norm{f}_{L^2_TV^\ast}^2 + \norm{w_0}_V^2  + \norm{\alpha^{n-1}}_{L^\infty_TY}^2  \notag
        ) \notag
        \\
        &+ \frac{\beta_{4\varphi} \norm{K}\norm{M} +\beta_{5\varphi} \norm{K}\norm{M} \norm{\alpha^{n-1}}_{L^\infty_TY}}{m_A  - m_j \norm{N}^2}\norm{w^{n-1}}_{L^2_TV}^2 \notag \\
         &+ \frac{\beta_{4\varphi} \norm{K}\norm{M} +\beta_{5\varphi} \norm{K}\norm{M} \norm{\alpha^{n-1}}_{L^\infty_TY}}{m_A  - m_j \norm{N}^2}\norm{w^{n}}_{L^2_TV}^2,
         \notag
\end{align}
\begin{align}\label{eq:est_wdot_n}
    \norm{\Dot{w}^n}_{L^2_T V^\ast}^2 &\leq c(1+  \norm{f}_{L^2_TV^\ast}^2 + \norm{w^{n-1}}_{L^2_TV}^2 
    		+\norm{\alpha^{n-1}}_{L^\infty_TY}^2 + \norm{w^n}_{L^2_TV}^2) \\
      &+ \big(\sqrt{2}\beta_{5\varphi} \norm{K}\norm{M}\norm{\alpha^{n-1}}_{L^\infty_TY}\big)^2\norm{w^{n-1}}_{L^2_TV}^2 .\notag
\end{align}
\end{subequations}
for all $n\in \mathbb{Z}_+$. Applying Minkowski's inequality to \eqref{eq:algorithm12} reads
\begin{align*}
    \norm{\alpha^n(t)}_Y \leq  \norm{\alpha_0}_Y +  \int_0^t \norm{\mathcal{G}(s,\alpha^n(s),Mw^n(s))}_Y  ds
\end{align*}
for a.e. $t\in (0,T)$. We observe that by  Minkowski's inequality, \hyperref[assumptionG]{$\H{\mathcal{G}}$}\ref{list:G_Lipschitz} and \hyperref[assumptionMNK]{$\H{MNK}$}
\begin{align}\label{eq:est_on_G}
    \norm{\mathcal{G}(s,\alpha^n(s),Mw^n(s))}_Y &\leq  \norm{\mathcal{G}(s,\alpha^n(s),Mw^n(s)) - \mathcal{G}(s,0,0)}_Y + \norm{\mathcal{G}(s,0,0)}_Y\\
    &\leq L_\mathcal{G} \norm{\alpha^n(s)}_Y + L_\mathcal{G}\norm{M}\norm{w^n(s)}_V + \norm{\mathcal{G}(s,0,0)}_{Y} \notag
\end{align}
for a.e. $s\in (0,t) \subset (0,T)$.
Accordingly, the Cauchy-Schwarz inequality and \hyperref[assumptionG]{$\H{\mathcal{G}}$}\ref{list:G_00} implies
\begin{align*}
    \norm{\alpha^n(t)}_Y &\leq  c(\norm{\alpha_0}_Y + T\norm{\mathcal{G}(\cdot,0,0)}_{L^\infty_TY} + T^{1/2} \norm{w^n}_{L^2_TV})  +  c\int_0^t\norm{\alpha^n(s)}_Y ds
\end{align*}
for a.e $t\in (0,T)$. By Gr\"{o}nwall's inequality (see, e.g., \cite{evans}), we have
\begin{subequations}
\begin{align}\label{eq:alpha_estimate1}
    \norm{\alpha^n}_{L^\infty_TY} \leq  c(\norm{\alpha_0}_Y  &+ T^k\norm{\mathcal{G}(\cdot,0,0)}_{L^\infty_TY} + T^k \norm{w^n}_{L^2_TV} )(1+cT\mathrm{e}^{cT}),
\end{align}
and from Young's inequality 
\begin{align}\label{eq:alpha_estimate2}
    \norm{\alpha^n}_{L^\infty_TY}^2 \leq  c(\norm{\alpha_0}_Y^2  &+ T^k(\norm{\mathcal{G}(\cdot,0,0)}_{L^\infty_TY}^2 +  T^k\norm{w^n}_{L^2_TV}^2  ))(1+cT^2\mathrm{e}^{2ct}) 
\end{align}
\end{subequations}
for $k\geq 1/2$. We will show the uniform bound by induction on $n$. 
\\
\indent
For $n=1$, with the initial guesses; $w^0 = w_0$ and $\alpha^0 = \alpha_0$, \eqref{eq:est_w_n}-\eqref{eq:est_wdot_n} becomes
\begin{align*}
   \norm{w^1}_{L^\infty_TH}^2 + \norm{w^1}_{L^2_TV}^2 
		&\leq c(1 +\norm{w_0}_V^2+  \norm{f}_{L^2_TV^\ast}^2    + \norm{\alpha_0}_{Y}^2 
  ) 
        \\ \notag
         &+ T \frac{\beta_{4\varphi} \norm{K}\norm{M} +\beta_{5\varphi} \norm{K}\norm{M} \norm{\alpha_0}_{Y}}{m_A  - m_j \norm{N}^2}
          \norm{w_0}_{V}^2 \\
         &+ \frac{\beta_{4\varphi} \norm{K}\norm{M} +\beta_{5\varphi} \norm{K}\norm{M} \norm{\alpha_0}_{Y}}{m_A  - m_j \norm{N}^2}\norm{w^1}_{L^2_TV}^2, \notag
\end{align*} 
\begin{align*}
    \norm{\Dot{w}^1}_{L^2_T V^\ast}^2
        &\leq c(1+  \norm{f}_{L^2_TV^\ast}^2 + \norm{w_0}_{V}^2 
    		+\norm{\alpha_0}_{Y}^2 + \norm{w^1}_{L^2_TV}^2)\\
      &+ T\big(\sqrt{2}\beta_{5\varphi}\norm{K}\norm{M}\norm{\alpha_0}_Y\big)^2\norm{w_0}_{V}^2 .\notag
\end{align*}
From the smallness-assumption \eqref{eq:assumptionbound_mA_max}, it follows that
\begin{align}\label{eq:estiamte_w_1}
    \norm{w^1}_{L^\infty_TH}^2+ \norm{w^1}_{\mathcal{W}^{1,2}_T}^2 &\leq c (1 +   \norm{f}_{L^2_TV^\ast}^2 + \norm{w_0}_V^2 + \norm{\alpha_0}_Y^2 ).
\end{align}
We next define the complete metric space
\begin{equation*}
	X_T(a) = \{ h \in C([0,T];Y) : \norm{h}_{L^\infty_TY} \leq a, \ \text{with } a\in \R_+\}
\end{equation*}
and the operator $\Lambda : X_T(a) \rightarrow X_T(a)$ by
\begin{align}\label{eq:Lambda}
	\Lambda\alpha^1 (t) = \alpha_0 + \int_0^t \mathcal{G}(s,\alpha^1(s),Mw^1(s)) ds
\end{align}
for $w^1 \in \mathcal{W}^{1,2}_T$. We verify that $\alpha^1$ is indeed a solution to \eqref{eq:algorithm12} for $n=1$ in the next lemma.
\begin{lem}\label{lemma:lambda_fixedpoint}
	Let $w^1 \in \mathcal{W}^{1,2}_T$ be a solution of \eqref{eq:algorithm11}-\eqref{eq:boundaryconditions_walphaxietag1} with $w^0 = w_0\in V$ and $\alpha^0 = \alpha_0 \in Y$. 
	Under the assumptions of Theorem \ref{thm:mainresult}, the operator $\Lambda : X_T(a) \rightarrow X_T(a)$, defined by \eqref{eq:Lambda}, has a unique fixed-point, i.e., there exists a constant $0 \leq L < 1$ such that 
  \begin{equation*}
  	\norm{\Lambda \alpha^1_1 - \Lambda \alpha^1_2}_{X_T(a)} \leq L \norm{ \alpha^1_1 - \alpha^1_2}_{X_T(a)}.
  \end{equation*} 
\end{lem}
The proof of Lemma \ref{lemma:lambda_fixedpoint} is moved to Appendix \ref{appendix:lambda_fixedpoint} as it follows from the standard ODE arguments combined with the estimate \eqref{eq:estiamte_w_1}, and the assumptions \hyperref[assumptionG]{$\H{\mathcal{G}}$} and \hyperref[assumptionMNK]{$\H{MNK}$}.
\indent Next, investigating $n=2$. This implies

\begin{align*}
     \norm{w^2}_{L^\infty_TH}^2 + \norm{w^2}_{L^2_TV}^2 
		&\leq c(1 +\norm{w_0}_V^2+  \norm{f}_{L^2_TV^\ast}^2    + \norm{\alpha^1}_{L^\infty_TY}^2  ) 
        \\ \notag
         &+ \frac{\beta_{4\varphi} \norm{K}\norm{M} +\beta_{5\varphi} \norm{K}\norm{M} \norm{\alpha^1}_{L^\infty_TY}}{m_A  - m_j \norm{N}^2} \norm{w^1}_{L^2_TV}^2 \\
         &+ \frac{\beta_{4\varphi} \norm{K}\norm{M} +\beta_{5\varphi} \norm{K}\norm{M} \norm{\alpha^1}_{L^\infty_TY}}{m_A  - m_j \norm{N}^2}\norm{w^2}_{L^2_TV}^2, \notag\\
         \norm{\Dot{w}^2}_{L^2_T V^\ast}^2 &\leq c(1+  \norm{f}_{L^2_TV^\ast}^2 + \norm{w^{1}}_{L^2_TV}^2 
    		+\norm{\alpha^{1}}_{L^\infty_TY}^2 + \norm{w^2}_{L^2_TV}^2) \\
      &+ \big(\sqrt{2}\beta_{5\varphi} \norm{K}\norm{M}\norm{\alpha^{1}}_{L^\infty_TY}\big)^2\norm{w^{1}}_{L^2_TV}^2. \notag
\end{align*}
From the estimate \eqref{eq:alpha_estimate1} for $n=1$, we have that
\begin{align*}
    \norm{w^2}_{L^\infty_TH}^2 + &\norm{w^2}_{L^2_TV}^2 \leq c (1+\norm{w_0}_V^2 +   \norm{f}_{L^2_TV^\ast}^2 + 
\norm{\alpha^1}_{L^\infty_TY}^2 )
       \\
       &\hspace{0.2cm}+ \frac{\beta_{4\varphi} \norm{K}\norm{M}+\beta_{5\varphi} \norm{K}\norm{M}(\norm{\alpha_0}_{Y}+ T^k c(\norm{(f,w_0,\alpha_0)}_{L^2_TV^\ast \times V\times Y}))}{m_A  - m_j \norm{N}^2}
         \norm{w^1}_{L^2_TV}^{2}\\
         &\hspace{0.2cm}+ \frac{\beta_{4\varphi} \norm{K}\norm{M}+\beta_{5\varphi} \norm{K}\norm{M}(\norm{\alpha_0}_{Y}+ T^k c(\norm{(f,w_0,\alpha_0)}_{L^2_TV^\ast \times V\times Y}))}{m_A  - m_j \norm{N}^2}
         \norm{w^2}_{L^2_TV}^2,
         \\
         \norm{\Dot{w}^2}_{L^2_T V^\ast}^2 &\leq c(1+  \norm{f}_{L^2_TV^\ast}^2 + \norm{w^{1}}_{L^2_TV}^2 
    		+\norm{\alpha^{1}}_{L^\infty_TY}^2 + \norm{w^2}_{L^2_TV}^2) \\
      &+ \big(\sqrt{2}\beta_{5\varphi} \norm{K}\norm{M}(\norm{\alpha_0}_Y + T^k c(\norm{(f,w_0,\alpha_0)}_{L^2_TV^\ast \times V\times Y})))\big)^2\norm{w^{1}}_{L^2_TV}^2 
\end{align*}
for $k\geq 1/2$. Choosing $T>0$ such that 
\begin{align}\label{eq:T_k}
    T^k &\sim \frac{1}{c(\norm{(f,w_0,\alpha_0)}_{L^2_TV^\ast \times V\times Y})}
\end{align}
small enough for some $k\geq 1/2$. From the smallness-condition \eqref{eq:assumptionbound_mA_max} and the choice of $T>0$, we have that
\begin{align*}
    \frac{\sqrt{2}\beta_{5\varphi} \norm{K}\norm{M}}{m_A  - m_j \norm{N}^2-\sqrt{2}\beta_{4\varphi}\norm{K}\norm{M}}(\norm{\alpha_0}_{Y}+ T^kc(\norm{(f,w_0,\alpha_0)}_{L^2_TV^\ast \times V\times Y})) &<1,\\
    \sqrt{2}\beta_{5\varphi} \norm{K}\norm{M}(\norm{\alpha_0}_{Y}+ T^kc(\norm{(f,w_0,\alpha_0)}_{L^2_TV^\ast \times V\times Y})) &<m_A.
\end{align*}
Gathering the above and using \eqref{eq:estiamte_w_1} implies
\begin{align*}
    \norm{w^2}_{L^\infty_TH}^2 + \norm{w^2}_{\mathcal{W}^{1,2}_T}^2  \leq c (1+   \norm{f}_{L^2_TV^\ast}^2 + \norm{w_0}_V^2 +
    \norm{\alpha_0}_Y^2 ).
\end{align*} 
Moreover, verifying that $\alpha^2 \in C([0,T];Y)$ is indeed a solution to \eqref{eq:algorithm12} for $n=2$ follows by the same approach as for $n=1$.
\\
\indent
The induction step follows the same procedure as for $n=2$. Consequently, $(w^n, \alpha^n) \in \mathcal{W}^{1,2}_T \times C([0,T];Y)$ is the approximated solution of \eqref{eq:algorithm11}-\eqref{eq:algorithm12}. Further, we obtain the following uniform bound
\begin{align}\label{eq:uniformlybounded}
     \norm{w^n}_{L^\infty_TH}^2 + \norm{w^n}_{\mathcal{W}^{1,2}_T}^2    +\norm{\alpha^n}_{L^\infty_TY}^2 \leq c (1 + \norm{f}_{L^2_TV^\ast}^2 + \norm{w_0}_V^2 + \norm{\alpha_0}_Y^2) 
\end{align}
for all $n \in \mathbb{Z}_+$. Additionally, it follows from Proposition \ref{prop:integrationbypartsformula} that $\{ w^n \}_{n\geq 1} \subset C([0,T];H)$.
\\
\indent\textbf{Step \ref{convergence}.\ref{passinglimit}} \rsubtask{passinglimit} \textit{(Convergence of the approximated solution)}.
We first show that $\{(w^n,\alpha^n)\}_{n\geq 1}$ is a Cauchy sequence in $L^2_TV \times C([0,T];Y)$. This is summarized in the proposition below.
\begin{prop}\label{prop:cauchysequences}
	Let $w^0 = w_0 \in V$ and $\alpha^0 = \alpha_0 \in Y$. 
	Under the hypothesis of Theorem \ref{thm:mainresult}, let $\{(w^n,\alpha^n)\}_{n\geq 1} \subset \mathcal{W}^{1,2}_T \times C([0,T];Y)$ be the solution of \eqref{eq:algorithm11}-\eqref{eq:algorithm12}. Then, $\{(w^n,\alpha^n)\}_{n\geq 1}$ is a Cauchy sequence in $L^2_TV \times C([0,T];Y)$.	In addition, $\{\mathcal{R}w^n\}_{n\geq 1}$, $\{\mathcal{S}_\varphi w^n\}_{n\geq 1}$, and $\{\mathcal{S}_jw^n\}_{n\geq 1}$ are Cauchy sequences in $L^2_TV^\ast$, $L^2_TX$ and $L^2_TX$, respectively.
\end{prop} 
\begin{remark}
    To cover the case where we obtain global time of existence when $\beta_{1\varphi} = \beta_{4\varphi} = \beta_{5\varphi}= \beta_{6\varphi} = \beta_{7\varphi} = 0$, the proof needs to be slightly modified. This case is included in Corollary \ref{cor:cauchysequences}.
\end{remark}

\begin{proof}
Let $e_{\Dot{w}}^n =  \Dot{w}^n - \Dot{w}^{n-1}$, $e_{w}^n =  w^n - w^{n-1}$, and $e_\alpha^n =  \alpha^n - \alpha^{n-1}$.
To begin with, we add \eqref{eq:algorithm11} for two iterations at the levels  $n$ and $n-1$. Then, choosing $v=w^{n-1}$ and $v=w^n$ for the levels $n$ and $n-1$, respectively, implies
\begin{align*}
    \inner{e^n_{\Dot{w}}(t)}{e^n_{w}(t) } &+ \inner{  A(t, w^n(t)) - A(t, w^{n-1}(t)) }{w^n(t) - w^{n-1}(t) } \\ \notag
	&\leq \inner{ \mathcal{R}w^{n-2}(t) - \mathcal{R}w^{n-1}(t) }{e_{w}^n(t)}
	\\ \notag
	&+\varphi(t, \alpha^{n-1}(t),\mathcal{S}_\varphi w^{n-1}(t), Mw^{n-1}(t), Kw^{n-1}(t)) \\ \notag
        &-  \varphi(t, \alpha^{n-1}(t),\mathcal{S}_\varphi w^{n-1}(t), Mw^{n-1}(t), Kw^{n}(t)) \\ \notag
	&+  \varphi(t,\alpha^{n-2}(t),\mathcal{S}_\varphi w^{n-2}(t), Mw^{n-2}(t), Kw^{n}(t))  \\ \notag
        &-  \varphi(t,\alpha^{n-2}(t),\mathcal{S}_\varphi w^{n-2}(t), Mw^{n-2}(t), Kw^{n-1}(t))\\ \notag
	&+ j^\circ(t,\alpha^{n-1}(t),\mathcal{S}_j w^{n-1}(t),Nw^{n}(t); -Ne_w^{n}(t)) \\ \notag
        &+ j^\circ(t,\alpha^{n-2}(t),\mathcal{S}_j w^{n-2}(t), Nw^{n-1}(t); Ne_w^{n}(t) )  
\end{align*}
for a.e. $t \in (0,T)$ with $w^n(0) = w^{n-1}(0) = w_0$. We deduce from hypotheses  \hyperref[assumptionphi]{$\H{A}$}\ref{list:A_maximalmonotone}, \hyperref[assumptionphi]{$\H{\varphi}$}\ref{list:phi_estimate}, \hyperref[assumptionj]{$\H{j}$}\ref{list:j_estimate}, \hyperref[assumptionMNK]{$\H{MNK}$}, and the Cauchy-Schwarz inequality that
\begin{align*}	
	&\inner{e_{\Dot{w}}^n(t)}{e_w^{n}(t) } + (m_A-m_j\norm{N}^2) \norm{e_w^n(t)}_V^2  \\
	&\leq  \norm{\mathcal{R}w^{n-1}(t) - \mathcal{R}w^{n-2}(t)}_{V^\ast}\norm{e_w^n(t)}_V 
   \\
   &+(\Bar{m}_j \norm{N}+\beta_{2\varphi} \norm{K})  \norm{ e_\alpha^{n-1}(t)}_Y\norm{e_w^n(t)}_V 
   \\
   &+ \beta_{1\varphi} \norm{K} \norm{M} \norm{w^{n-1}(t)}_V \norm{e_\alpha^{n-1}(t)}_Y\norm{e_w^n(t)}_V 
        \\
    &+
	\beta_{3\varphi} \norm{K}  \norm{ \mathcal{S}_\varphi  w^{n-1}(t) - \mathcal{S}_\varphi  w^{n-2}(t)}_X\norm{e_w^n(t)}_V  \\
	&+  \beta_{4\varphi} \norm{K}\norm{M} \norm{e_w^{n-1}(t)}_V \norm{e_w^n(t)}_V \\
    &+ \beta_{5\varphi} \norm{K} \norm{M}\norm{\alpha^{n-2}(t)}_Y\norm{e_w^{n-1}(t)}_V \norm{e_w^n(t)}_V\\
    &+\beta_{6\varphi} \norm{K} \norm{M}\norm{w^{n-1}(t)}_V\norm{ \mathcal{S}_\varphi  w^{n-1}(t) - \mathcal{S}_\varphi  w^{n-2}(t)}_X\norm{e_w^n(t) }_V \\
         &+ \beta_{7\varphi} \norm{K}  \norm{\alpha^{n-1}(t)}_Y \norm{ \mathcal{S}_\varphi  w^{n-1}(t) - \mathcal{S}_\varphi  w^{n-2}(t)}_X\norm{e_w^n(t)}_V
       \\
	&+ \Bar{m}_j \norm{N} \norm{\mathcal{S}_j w^{n-1}(t) - \mathcal{S}_jw^{n-2}(t)}_X \norm{e_w^{n}(t)}_V
\end{align*}
for a.e. $t \in (0,T)$. Integrating over the time interval $(0,t')\subset (0,T)$, we observe after using the integration by parts formula in Proposition \ref{prop:integrationbypartsformula} (with $v_1=v_2=e_w^n(t)$ for a.e. $t\in(0,T)$) and the Cauchy-Schwarz inequality that
\begin{align*}
    &I_1 := \frac{1}{2}\norm{e_w^n(t')}_H^2 - \frac{1}{2}\norm{e_w^n(0)}_H^2   + (m_A- m_j \norm{N}^2)\int_0^{t'}  \norm{e_w^n(t)}_V^2 dt \\
	&\leq \int_0^{t'} \norm{ \mathcal{R}w^{n-1}(t) - \mathcal{R}w^{n-2}(t)}_{V^\ast}  \norm{e_w^n(t)}_V dt
	\\
	&+ (\bar{m}_{j} \norm{N} + \beta_{2\varphi} \norm{K})  \int_0^{t'} \norm{e_\alpha^{n-1}(t)}_Y\norm{e_w^n(t)}_V dt \\
        &+ \beta_{1\varphi} \norm{K} \norm{M} \int_0^{t'} \norm{w^{n-1}(t)}_V  \norm{e_\alpha^{n-1}(t)}_Y \norm{e_w^n(t)}_V dt\\
	&+
	\beta_{3\varphi} \norm{K} \int_0^{t'} \norm{ \mathcal{S}_\varphi  w^{n-1}(t) - \mathcal{S}_\varphi  w^{n-2}(t)}_X\norm{e_w^n(t)}_V dt \\
	&+  \beta_{4\varphi} \norm{K}\norm{M} \int_0^{t'} \norm{e_w^{n-1}(t)}_V \norm{e_w^n(t)}_V dt\\
        &+  \beta_{5\varphi} \norm{K} \norm{M}\int_0^{t'} \norm{\alpha^{n-2}(t)}_Y \norm{e_w^{n-1}(t) }_V \norm{e_w^n(t)}_V dt\\
	&+  \beta_{6\varphi} \norm{K} \norm{M}   \int_0^{t'} \norm{w^{n-1}(t)}_V \norm{ \mathcal{S}_\varphi  w^{n-1}(t) - \mathcal{S}_\varphi  w^{n-2}(t)}_X\norm{e_w^n(t)}_V dt\\
         &+ \beta_{7\varphi} \norm{K}  \int_0^{t'} \norm{\alpha^{n-1}(t)}_Y \norm{ \mathcal{S}_\varphi  w^{n-1}(t) - \mathcal{S}_\varphi  w^{n-2}(t)}_X\norm{e_w^n(t)}_V dt
        \\
        &+
	\bar{m}_{j} \norm{N} \int_0^{t'} \norm{ \mathcal{S}_j w^{n-1}(t) - \mathcal{S}_j w^{n-2}(t)}_X\norm{e_w^n(t)}_V dt\\
        &=: I_{1,1} + I_{1,2} + I_{1,3} + I_{1,4} + I_{1,5} + I_{1,6} + I_{1,7} + I_{1,8} + I_{1,9}  
\end{align*}
for a.e. $t'\in (0,T)$. We may apply the Cauchy-Schwarz inequality to $I_{1,1}$, $I_{1,2}$, $I_{1,4}$ , $I_{1,5}$, and $I_{1,9}$. For $I_{1,3}$ and $I_{1,6}+I_{1,8}$, we respectively apply H\"{o}lder's inequality with $\frac{1}{2} + \frac{1}{\infty}+ \frac{1}{2} = 1$ and $ \frac{1}{\infty}+\frac{1}{2} +\frac{1}{2} = 1$. To treat $I_{1,7}$, we observe by the Cauchy-Schwarz inequality and  \hyperref[assumptionS1]{$\H{\mathcal{S}_\varphi}$}\ref{list:S_hist11} that 
\begin{align}
    \label{eq:S_hist_est2}
		\norm{\mathcal{S}_\varphi w^{n-1}(t) - \mathcal{S}_\varphi  w^{n-2}(t)}_X^2 &\leq c_{\mathcal{S}_\varphi }^2 T 	\int_0^t  \norm{e_w^{n-1}(s)}_V^2 ds
\end{align}
for a.e. $t \in (0,T)$. Therefore, we may use H\"{o}lder's inequality with $\frac{1}{2}+ \frac{1}{\infty} +\frac{1}{2} = 1$. Similarly, by \hyperref[assumptionR]{$\H{\mathcal{R}}$}\ref{list:S_hist} and \hyperref[assumptionS2]{$\H{\mathcal{S}_j}$}\ref{list:S_hist22}, respectively, we obtain
\begin{align}
	\label{eq:R_hist_est}
	\norm{\mathcal{R}w^{n-1}(t) - \mathcal{R}w^{n-2}(t)}_{V^\ast}^2 &\leq c_\mathcal{R}^2 T \int_0^t  \norm{e_w^{n-1}(s)}_V^2 ds 
\\
	 \label{eq:S_hist_est1}
	\norm{\mathcal{S}_jw^{n-1}(t) - \mathcal{S}_j w^{n-2}(t)}_X^2 &\leq c_{\mathcal{S}_j}^2 T 	\int_0^t  \norm{e_w^{n-1}(s)}_V^2 ds 
\end{align}
for a.e. $t \in (0,T)$. In addition, using that $w^n(0) = w^{n-1}(0)$ and the smallness-assumption \eqref{eq:assumptionbound_mA_max}, then 
\begin{align*}
  (m_A- m_j \norm{N}^2)\int_0^{t'}  \norm{e_w^n(t)}_V^2 dt  \leq I_1
\end{align*}
for all $t' \in [0,T]$. Gathering the above and dividing by $\norm{e_w^n}_{L^2_TV}$, we have
\begin{align}\label{eq:I_2123}
	I_2 &:= 	(m_A -m_j \norm{N}^2 ) \bigg[ \int_0^{t'} \norm{e_w^n(t)}^2_V dt \bigg]^{1/2}\\
	&\leq T^k( c_\mathcal{R} + \beta_{3\varphi}
	\norm{K}  c_{\mathcal{S}_\varphi } + \Bar{m}_j \norm{N} c_{\mathcal{S}_j}   + \beta_{7\varphi}\norm{K}c_{\mathcal{S}_\varphi}\norm{\alpha^{n-1}}_{L^\infty_{T}Y}) \notag \\
 &\times \bigg[ 	\int_0^{t'}\int_0^t  \norm{e_w^{n-1}(s)}_V^2 ds dt \bigg]^{1/2}  \notag \\
	&+(\bar{m}_{j}\norm{N}+\beta_{2\varphi}\norm{K})\bigg[ \int_0^{t'}\norm{e_\alpha^{n-1}(t)}_Y^2dt \bigg]^{1/2} \notag \\
    &+ \beta_{1\varphi} \norm{K}   \norm{M} \norm{w^{n-1}}_{L^2_TV}
 \norm{e_\alpha^{n-1}}_{L^\infty_{t'}Y} \notag\\
	&+  ( \beta_{4\varphi}+\beta_{5\varphi} \norm{\alpha^{n-2}}_{L^\infty_{T}Y})\norm{K}\norm{M}  \bigg[ \int_0^{t'} \norm{ e_w^{n-1}(t) }^2_V dt\bigg]^{1/2} \notag\\
       &+ T^k\beta_{6\varphi} c_{\mathcal{S}_\varphi} \norm{K} \norm{M}  \norm{w^{n-1}}_{L^2_{T}V}  \bigg[ 	\int_0^{t'}  \norm{e_w^{n-1}(t)}_V^2 dt \bigg]^{1/2} \notag \\
  	&=: I_{2,1} + I_{2,2} + I_{2,3} + I_{2,4} + I_{2,5} \notag
\end{align}
for all $t'\in [0,T]$ and $k\geq 1/2$. Next, subtracting \eqref{eq:algorithm12} for two iterations at the levels  $n-1$ and $n-2$. Utilizing Minkowski's inequality and \hyperref[assumptionG]{$\H{\mathcal{G}}$}\ref{list:G_Lipschitz} yields
\begin{align}\label{eq:est_alpha123}
	\norm{e_\alpha^{n-1}(t)}_Y &\leq L_\mathcal{G} \int_0^t \norm{e_\alpha^{n-1}(s)}_Y ds + L_\mathcal{G} \norm{M} \int_0^t \norm{e_w^{n-1}(s)}_V ds
\end{align}
for a.e. $t\in(0,T)$. Applying a standard Gr\"{o}nwall argument and the Cauchy-Schwarz inequality  reads
\begin{align}\label{eq:etimate_alpha}
	\norm{e_\alpha^{n-1}(t)}_Y 
	&\leq cT^k(1+cT\mathrm{e}^{cT})\Big[ \int_0^t\norm{e_w^{n-1}(s)}^2_V ds\Big]^{1/2}
\end{align}
for a.e. $t \in (0,T)$ and $k\geq 1/2$. From \eqref{eq:uniformlybounded}, $I_{2,4}$ becomes
\begin{align*}
    I_{2,4} &\leq (\beta_{4\varphi}+\beta_{5\varphi} \norm{\alpha_0}_Y)\norm{K}\norm{M} \bigg[ \int_0^{t'} \norm{ e_w^{n-1}(t) }^2_V dt\bigg]^{1/2}\\
    &+ T^k c(\norm{(f,w_0,\alpha_0)}_{L^2_TV^\ast \times V\times Y})
    \bigg[ \int_0^{t'} \norm{ e_w^{n-1}(t) }^2_V dt\bigg]^{1/2}\\
    &=: I_{2,4,1} + I_{2,4,2}.
\end{align*}
for some $k\geq 1/2$.  Combining $I_{2,4,2}$ and $I_{2,5}$, we define
\begin{align*}
    I_{2,6} := T^k (\beta_{6\varphi} c_{\mathcal{S}_\varphi}  \norm{w^{n-1}}_{L^2_{T}V} +  c(\norm{(f,w_0,\alpha_0)}_{L^2_TV^\ast \times V\times Y}) )\norm{K}\norm{M}\bigg[ \int_0^{t'} \norm{ e_w^{n-1}(t) }^2_V dt\bigg]^{1/2}
\end{align*}
for some $k\geq 1/2$. 
Accordingly, we have that
\begin{align*}
    I_2 \leq  I_{2,1} + I_{2,2} + I_{2,3} + I_{2,4,1} + I_{2,6}.
\end{align*}
Applying Young's inequality to $I_{2,4,1}$ and $I_{2,1} + I_{2,2} + I_{2,3}  + I_{2,6}$ and then the arithmetic-quadratic mean inequality to the latter term, we obtain 
\begin{align*}
	(I_2)^2 &\leq (I_{2,1} + I_{2,2} + I_{2,3} + I_{2,4,1} + I_{2,6} )^2 \\
	&\leq  2(I_{2,4,1})^2 + 2(I_{2,1}+ I_{2,2} + I_{2,3}+I_{2,6})^2  \notag\\
        &\leq 2(I_{2,4,1})^2 + 8\Big[(I_{2,1})^2 +  (I_{2,2})^2 + (I_{2,3})^2+(I_{2,6})^2\Big].
        \notag
\end{align*}
From \eqref{eq:etimate_alpha} and Young's inequality, the terms $(I_{2,2})^2$ and $(I_{2,3})^2$ becomes
\begin{align*}
    (I_{2,2})^2 &\leq T^k c(1+cT\mathrm{e}^{2cT})
   \int_0^{t'} \norm{e_w^{n-1}(t)}_V^2 dt  \\
    (I_{2,3})^2 &\leq  T^k c(1+cT^2\mathrm{e}^{2cT})
    \int_0^{t'} \norm{e_w^{n-1}(t)}_V^2 dt 
\end{align*}
for some $k\geq 1/2$. Gathering the above estimates and noting that $m_A - m_j \norm{N}^2>0$ by the smallness-assumption \eqref{eq:assumptionbound_mA_max} reads
\begin{align}\label{eq:estimate_before_iteration}
   I_3 &:= \int_0^{t'}  \norm{e_w^{n}(t)}_V^2 dt 
    \\ \notag
    &\leq 
     T^k \frac{c}{(m_A -m_j \norm{N}^2)^2} 
    \int_0^{t'}\int_0^t  \norm{e_w^{n-1}(s)}_V^2 ds dt  \\  \notag
    &+ \frac{2(\beta_{4\varphi} + \beta_{5\varphi}\norm{\alpha_0}_Y)^2 \norm{K}^2 \norm{M}^2}{(m_A -m_j \norm{N}^2)^2}  \int_0^{t'}\norm{e_w^{n-1}(t)}^2_V dt \\ \notag
    &+T^k\frac{c(1+cT^2\mathrm{e}^{2cT})}{(m_A -m_j \norm{N}^2)^2} \int_0^{t'}\norm{e_w^{n-1}(t) }^2_V dt    \\  
    &=: \Big(T^kR + \frac{2(\beta_{4\varphi} + \beta_{5\varphi}\norm{\alpha_0}_Y)^2 \norm{K}^2 \norm{M}^2}{(m_A -m_j \norm{N}^2)^2} \Big)  \int_0^{t'}\norm{e_w^{n-1}(t) }^2_V dt  \notag
\end{align}
for all $t' \in [0,T]$ and some $k\geq 1/2$. Here, 
\begin{align*}
    R &= \frac{c(\norm{(f,w_0,\alpha_0)}_{L^2_TV^\ast\times V \times Y})}{(m_A - m_j\norm{N}^2)^2}.
\end{align*} 
The aim is to get the right-hand side of \eqref{eq:estimate_before_iteration} to go to zero as $n\rightarrow \infty$. We will combine the smallness condition \eqref{eq:assumptionbound_mA_max} and the assumption on the final time $T$. \\
\indent
Iterating over $n\in \mathbb{Z}_+$, we have that
\begin{align*}
   I_3 \leq  c\bigg(T^kR + \frac{2(\beta_{4\varphi} + \beta_{5\varphi}\norm{\alpha_0}_Y)^2 \norm{K}^2 \norm{M}^2}{(m_A -m_j \norm{N}^2)^2} \bigg)^n (\norm{w^1}_{L^2_TV}^2 + \norm{w_0}_V^2)
\end{align*}
for some $k\geq 1/2$. Choosing $T>0$ such that
\begin{align*}
    T^k \sim \frac{1}{R}
\end{align*}
small enough for some $k\geq 1/2$,  the smallness-assumption \eqref{eq:assumptionbound_mA_max}
implies that 
\begin{align}\label{eq:sometimething}
    T^kR + \frac{2(\beta_{4\varphi} + \beta_{5\varphi}\norm{\alpha_0}_Y)^2 \norm{K}^2 \norm{M}^2}{(m_A -m_j \norm{N}^2)^2} <1.
\end{align}
Hence, \eqref{eq:estiamte_w_1} and then passing the limit $n\rightarrow \infty$ gives us
\begin{align*}
     \lim_{n\rightarrow\infty} c \bigg(T^kR + \frac{2(\beta_{4\varphi} + \beta_{5\varphi}\norm{\alpha_0}_Y)^2 \norm{K}^2 \norm{M}^2}{(m_A -m_j \norm{N}^2)^2} \bigg)^n =0
\end{align*}
as desired.
\\
\indent
Consequently, iterating over $n \in \mathbb{Z}_+$ in \eqref{eq:estimate_before_iteration}, and then passing the limit $n \rightarrow \infty$ gives us that $\{w^n\}_{n\geq1}$ is a Cauchy sequence in $L^2_TV$, and \eqref{eq:etimate_alpha} implies that $\{\alpha^n\}_{n\geq 1}$ is a Cauchy sequence in $C([0,T];Y)$. Moreover, $\{\mathcal{R}w^n\}_{n\geq 1}$ is a Cauchy sequence in $L^2_TV^\ast$ by \eqref{eq:R_hist_est}. Similarly, $\{\mathcal{S}_\varphi w^n\}_{n\geq 1}$ is a Cauchy sequence in $L^2_TX$ by \eqref{eq:S_hist_est2}, and $\{\mathcal{S}_j w^n\}_{n\geq 1}$ is a Cauchy sequence in $L^2_TX$ by \eqref{eq:S_hist_est1}. Concluding the proof.
\end{proof}

\begin{cor}\label{cor:cauchysequences}
	Let $w^0 = w_0 \in V$ and $\alpha^0 = \alpha_0 \in Y$. 
	Under the hypothesis of Theorem \ref{thm:mainresult} with $\beta_{1\varphi} = \beta_{4\varphi} = \beta_{5\varphi}= \beta_{6\varphi} = \beta_{7\varphi} = 0$, let $\{(w^n,\alpha^n)\}_{n\geq 1} \subset \mathcal{W}^{1,2}_T \times C([0,T];Y)$ be the solution of \eqref{eq:algorithm11}-\eqref{eq:algorithm12} for any time $T>0$. Then, $\{(w^n,\alpha^n)\}_{n\geq 1}$ is a Cauchy sequence in $L^2_TV \times C([0,T];Y)$. In addition, $\{\mathcal{R}w^n\}_{n\geq 1}$, $\{\mathcal{S}_\varphi w^n\}_{n\geq 1}$, and $\{\mathcal{S}_jw^n\}_{n\geq 1}$ are Cauchy sequences in $L^2_TV^\ast$, $L^2_TX$ and $L^2_TX$, respectively.
\end{cor} 
The proof  of Corollary \ref{cor:cauchysequences} can be found in Appendix \ref{appendix:cor_cauchy}.
\\
\\
\indent\textbf{Step \ref{convergence}.\ref{weaksol}} \rsubtask{weaksol} \textit{(Passing the limit in \eqref{eq:algorithm11}-\eqref{eq:algorithm12})}. From Proposition \ref{prop:cauchysequences}, it follows as $n\rightarrow \infty$ that
\begin{subequations}
	\begin{align}\label{eq:strong_convergences1}
		w^n \rightarrow w \text{ strongly in } &L^2_TV,
		&\alpha^n \rightarrow \alpha \text{ strongly in } C([0,T];Y),
		\\
		\mathcal{S}_\varphi w^n \rightarrow \mathcal{S}_\varphi w \text{ strongly in } &L^2_TX, &\mathcal{S}_j w^n \rightarrow \mathcal{S}_j  w \text{ strongly in } L^2_TX, \\
			\mathcal{R}w^n \rightarrow \mathcal{R}w \text{ strongly in }& L^2_TV^\ast.
		\label{eq:strong_convergences2}
	\end{align}
\end{subequations}
We are now in a position to pass the limit in \eqref{eq:algorithm11}-\eqref{eq:algorithm12}. First, by \eqref{eq:uniformlybounded}, we have that $\{w^n\}_{n\geq 1}$ and $\{\Dot{w}^n\}_{n\geq 1}$ are uniformly bounded in $L^2_TV$ and $L^2_TV^\ast$, respectively. Then, by the Eberlein–\v{S}mulian theorem, as $L^2_TV$ is a reflexive Banach space, we have, upon passing to a subsequence, that
\begin{align*}
	w^n \rightarrow w& \ \ \text{ weakly in } L^2_TV,
	&\Dot{w}^n \rightarrow \Dot{w} \ \ \text{ weakly in } L^2_TV^\ast.
\end{align*}
Now, with \eqref{eq:strong_convergences1}-\eqref{eq:strong_convergences2} in mind, we find by a proof of contradiction (assuming that the below does not hold thereby obtaining a contradiction with \eqref{eq:strong_convergences1}-\eqref{eq:strong_convergences2}) that
\begin{subequations}
	\begin{align}
		\label{eq:strong_convergences_t1}
		w^n(t) \rightarrow w(t)  \text{ strongly in } &V \ \  \text{ for a.e. }t\in (0,T),\\
		\alpha^n(t) \rightarrow \alpha(t) \text{ strongly in } &Y \ \ \text{ for a.e. }t\in (0,T),
		\label{eq:strong_convergences_talpha}\\
		\mathcal{S}_\varphi w^n(t) \rightarrow \mathcal{S}_\varphi  w(t) \text{ strongly in } &X \ \ \text{ for a.e. }t\in (0,T),
		\label{eq:strong_convergences_S2}\\
		\mathcal{S}_jw^n(t) \rightarrow  \mathcal{S}_j w(t) \text{ strongly in } &X \ \ \text{ for a.e. }t\in (0,T),
		\label{eq:strong_convergences_S1}\\
		\mathcal{R}w^n(t) \rightarrow \mathcal{R}w(t) \text{ strongly in } &V^\ast \ \ \text{ for a.e. }t\in (0,T),\label{eq:strong_convergences_tR}
	\end{align}
	as $n \rightarrow \infty$. By similar arguments, we find by \eqref{eq:uniformlybounded} that $\{w^n(t)\}_{n\geq 1}$ for a.e. $t\in (0,T)$ is uniformly bounded in $V$ (see, e.g., the first part of the proof of \cite[Lemma 13]{Zeng2018}). Since $V$ is a reflexive Banach space, 
    it follows by the Eberlein–\v{S}mulian  theorem, up to a subsequence, that $w^n(t) \rightarrow \Tilde{w}(t)$ weakly in $V$ for a.e.  $t\in (0,T)$. By uniqueness of limits, we have by \eqref{eq:strong_convergences_t1} that $\Tilde{w}(t) = w(t)$ for a.e. $t\in (0,T)$. Following the same reasoning shows that $\{\Dot{w}^n(t)\}_{n\geq 1}$ for a.e. $t\in (0,T)$ is uniformly bounded in $V^\ast$. Thus, as $n \rightarrow \infty$
	\begin{align}\label{eq:weak_wdot}
		\Dot{w}^n(t) \rightarrow \Dot{w}(t) \text{ weakly in } &V^\ast \ \  \text{ for a.e. }t\in (0,T).
	\end{align}
\end{subequations}
From \eqref{eq:strong_convergences_t1} and \hyperref[assumptionA]{$\H{A}$}\ref{list:A_demicont}, we get
\begin{equation}\label{eq:Aconvergenceweak}
 	\lim_{n\rightarrow \infty} \inner{A(t, w^{n}(t)) }{v}  = \inner{A(t, w(t))}{v} \ \ \text{ for all } v\in V, \text{ a.e. } t\in (0,T).
\end{equation}
Moreover, we utilize \hyperref[assumptionj]{$\H{j}$}\ref{list:j_convergence}, \eqref{eq:strong_convergences_t1}-\eqref{eq:strong_convergences_talpha}, \eqref{eq:strong_convergences_S1}, and \hyperref[assumptionMNK]{$\H{MNK}$} to obtain
\begin{align*}
	\limsup_{n\rightarrow \infty} \ &	j^\circ(t,\alpha^{n-1}(t), \mathcal{S}_jw^{n-1}(t),Nw^{n}(t);Nv-Nw^{n}(t)) \\
	&\leq j^\circ(t,\alpha(t),\mathcal{S}_jw(t),Nw(t);Nv-Nw(t))
\end{align*}
for all $v \in V$, a.e. $t\in(0,T)$. Next, by \eqref{eq:strong_convergences_t1}, \eqref{eq:strong_convergences_tR}-\eqref{eq:weak_wdot}, \eqref{eq:Aconvergenceweak}, and the Cauchy-Schwarz inequality, we can find that 
\begin{align*}
	\lim_{n\rightarrow \infty} \inner{A(t, w^{n}(t)) }{v- 		w^{n}(t) }  &= \inner{A(t, w(t))}{v- w(t)}, \\
	\lim_{n\rightarrow \infty} 		\inner{\Dot{w}^{n}(t)}{v-w^{n}(t)} &= \inner{\Dot{w}(t)}{v-w(t)},\\
	\lim_{n\rightarrow \infty} 		\inner{\mathcal{R}w^{n}(t)}{v-w^{n}(t)} &= \inner{\mathcal{R}w(t)}{v-w(t)},\\
	\lim_{n\rightarrow \infty}\inner{f(t)}{v-w^{n}(t)} &= 	\inner{f(t)}{v-w(t)}
\end{align*}
for all $v\in V$, a.e. $t\in (0,T)$.
Furthermore, let $\mathbb{E} = Y \times X \times U$ be equipped with the norm $\norm{(x,y,z)}_{\mathbb{E}} = \norm{x}_Y + \norm{y}_X + \norm{z}_U$. We wish to deduce that $\varphi(t,\cdot,\cdot,\cdot,\cdot)$ is continuous on $\mathbb{E} \times Z$ for a.e. $t\in (0,T)$ by applying Lemma \ref{lemma:phicontinuous}. The conditions \ref{list:phicont1} and \ref{list:phicont3} are directly obtained by \hyperref[assumptionphi]{$\H{\varphi}$}\ref{list:phi_cont},\ref{list:phi_bounded}. Lastly, we find that condition  \ref{list:phicont2} holds by Proposition \ref{prop:convex_lsc_implies_locallyLipschitz}. Indeed, as
$\varphi$ is lower semicontinuous and convex in its last argument, by \hyperref[assumptionphi]{$\H{\varphi}$}\ref{list:phi_convex_lsc}, and the fact that $\varphi$ is finite (does not take the values $\pm\infty$), it then follows by \eqref{eq:strong_convergences_t1}-\eqref{eq:strong_convergences_S2} and \hyperref[assumptionMNK]{$\H{MNK}$} that
\begin{align*}
	&\lim_{n\rightarrow \infty} \big[\varphi(t,\alpha^{n-1}(t),\mathcal{S}_\varphi w^{n-1}(t),Mw^{n-1}(t),Kv) \\
	&\qquad- \varphi(t,\alpha^{n-1}(t),\mathcal{S}_\varphi w^{n-1}(t),Mw^{n-1}(t),Kw^{n-1}(t)) \big]\\
	&\qquad= \varphi(t,\alpha(t),\mathcal{S}_\varphi  w(t),Mw(t),Kv) - \varphi(t,\alpha(t),\mathcal{S}_\varphi  w(t),Mw(t),Kw(t))
\end{align*}
for all $v\in V$, a.e. $t\in(0,T)$.
Next, we have by \hyperref[assumptionG]{$\H{\mathcal{G}}$}\ref{list:G_Lipschitz}, that $\mathcal{G}(t,\cdot,\cdot)$ is continuous on $Y \times U$ for a.e. $t\in (0,T)$. From \eqref{eq:est_on_G}, \hyperref[assumptionG]{$\H{\mathcal{G}}$}\ref{list:G_00}, and \eqref{eq:uniformlybounded}, we obtain the desired bound, i.e., integrable and independent of $n$.
Combining
\eqref{eq:strong_convergences_t1}, \eqref{eq:strong_convergences_talpha}, and \hyperref[assumptionMNK]{$\H{MNK}$}, we may apply the dominated convergence theorem to conclude
\begin{align*}
	\lim_{n\rightarrow \infty}\int_0^t \mathcal{G}(s,\alpha^n(s),Mw^n(s))ds 
	= \int_0^t \mathcal{G}(s,\alpha(s),Mw(s))ds
\end{align*}
for all $t\in [0,T]$. Thus, passing the upper limit $n\rightarrow \infty$ in \eqref{eq:algorithm11}-\eqref{eq:algorithm12} gives us that $(w,\alpha) \in \mathcal{W}^{1,2}_T \times C([0,T];Y)$ is indeed a solution to Problem \ref{prob:fullproblem}.
\\
\indent

\textbf{Step \ref{continuous}\rtask{continuous}} \textit{(Continuous dependence on initial data)}.
For simplicity of notation, we define $\mathbb{U} = L^2_{T/2}V \times C([0,T/2];Y)$ equipped with the norm $\norm{(x,y)}_\mathbb{U}^2 = \norm{x}_{L^2_{T/2}V}^2 +\norm{y}_{L^\infty_{T/2}Y}^2$.
Consider two sets of initial data $(w_{01},\alpha_{01}),(w_{02},\alpha_{02}) \in V \times Y$, we aim to prove that for all $\lambda >0$ there exists a  $\delta>0$, which will be fixed later, such that
 \begin{subequations}
     \begin{align}\label{eq:continuous1}
		\norm{(w_{01},\alpha_{01})-(w_{02},\alpha_{02})}_{V \times Y} < \delta
	\end{align}
	implies
	\begin{align}\label{eq:continuous2}
		\norm{(w_1,\alpha_1) - (w_2,\alpha_2)}_\mathbb{U} < \lambda.
	\end{align}
  \end{subequations}
    We consider the time interval $(0,T/2)$ with $T$ as in \eqref{eq:times_T} to guarantee that $\norm{(w_{01},\alpha_{01})-(w_{02},\alpha_{02})}_{V \times Y} < \delta$. 
  Here,  $(w_1,\alpha_1),(w_2,\alpha_2) \in \mathcal{W}^{1,2}_{T/2} \times C([0,T/2];Y)$ are two solutions to Problem \ref{prob:fullproblem} corresponding to $(w_{01},\alpha_{01}),(w_{02},\alpha_{02})$. That is,  for $i=1,2$, $(w_i,\alpha_i) \in \mathcal{W}^{1,2}_{T/2} \times C([0,T/2];Y)$ is the solution to
\begin{subequations}
	\begin{align*}
	&\alpha_i(t) = \alpha_{0i} + \int_0^t \mathcal{G}(s,\alpha_i(s),Mw_i(s))ds,\\
	&\inner{\Dot{w}_i (t) + A(t, w_i(t))- f(t) + \mathcal{R}w_i(t) }{v- w_i(t) }   
 \\
	&+ \varphi(t,\alpha_i(t), \mathcal{S}_\varphi  w_i(t), Mw_i(t), Kv)   -  \varphi(t, \alpha_i(t), \mathcal{S}_\varphi  w_i(t), Mw_i(t), Kw_i(t)) \notag\\
	&+ j^\circ(t,\alpha_i(t),\mathcal{S}_j w_i(t), Nw_i(t); Nv -Nw_i(t))\geq 0
	\notag
	\end{align*}
	for all $v\in V$, a.e. $t\in (0,T/2)$ with
	\begin{align*}
	w_i(0) = w_{0i}.
	\end{align*}
\end{subequations}
        Let $\{(w_i^n,\alpha^n_i)\}_{n\geq 1} \subset \mathcal{W}_{T/2}^{1,2} \times C([0,T/2];Y)$ be two solutions of \eqref{eq:algorithm11}-\eqref{eq:algorithm12} corresponding to the data $(w_{0i},\alpha_{0i}) \in  V \times Y$ for $i=1,2$. Then observe that
        \begin{align*}
            &\norm{(w_1,\alpha_1) - (w_2,\alpha_2)}_{\mathbb{U}}\\
            &\leq \norm{(w_1^n,\alpha_1^n) - (w_1,\alpha_1)}_{\mathbb{U}} +\norm{(w_2^n,\alpha_2^n) - (w_2,\alpha_2)}_{\mathbb{U}} +\norm{(w_1^n,\alpha_1^n) - (w_2^n ,\alpha_2^n )}_{\mathbb{U}}\\
            &=: B_1 + B_2 + B_3.
        \end{align*}
        By Step \ref{convergence}.\ref{passinglimit}, we have that $(w^n,\alpha^n) \rightarrow (w,\alpha)$ strongly in $\mathbb{U}$ when $n\rightarrow \infty$. Consequently, $B_1$ and $B_2$ must at least satisfy the estimate
        \begin{align*}
            B_1 + B_2 < \delta.
        \end{align*}
        While for $B_3$, we need the continuity of the flow map with respect to \eqref{eq:algorithm11}-\eqref{eq:algorithm12}. We add the two inequalities and choose $v=w^n_k(t)$ for a.e. $t\in(0,T/2)$ and $k=1,2$, $i\neq k$. Let $(w^{n-1}_1,\alpha^{n-1}_1),(w^{n-1}_2,\alpha^{n-1}_2) \in  \mathcal{W}^{1,2}_{T/2} \times C([0,T/2];Y)$ be given, and $W^0 = W_0 := w_{01} - w_{02}  \in V$ and $\Sigma^0 = \Sigma_0 := \alpha_{01} - \alpha_{02}\in Y$. Then $W^n  = w^n_1-w^n_2$ and $\Sigma^n= \alpha_1^n - \alpha_2^n$ solves
\begin{subequations}
\begin{align}\label{eq:ineq_W_n}
	&\inner{\Dot{W}^n(t) + A(t, w^n_1(t)) - A(t, w^n_2(t)) + \mathcal{R}w^{n-1}_1(t)-\mathcal{R}w^{n-1}_2(t) }{W^n(t)} \\
    &+ \varphi(t,\alpha^{n-1}_1(t), \mathcal{S}_\varphi w^{n-1}_1(t), Mw^{n-1}_1(t), Kw^n_2(t)) \notag\\
    &-  \varphi(t,\alpha^{n-1}_1(t), \mathcal{S}_\varphi w^{n-1}_1(t), Mw^{n-1}_1(t), Kw^n_1(t)) \notag\\
 &+ \varphi(t,\alpha^{n-1}_2(t), \mathcal{S}_\varphi w^{n-1}_2(t), Mw^{n-1}_2(t), Kw^n_1(t)) \notag \\
 &-  \varphi(t,\alpha^{n-1}_2(t), \mathcal{S}_\varphi w^{n-1}_2(t), Mw^{n-1}_2(t), Kw^n_2(t)) \notag\\
 &+j^\circ(t,\alpha^{n-1}_1(t), \mathcal{S}_j w^{n-1}_1(t), Nw^n_1(t); -W^n(t) ) \notag\\
 &+j^\circ(t,\alpha^{n-1}_2(t),  \mathcal{S}_j w^{n-1}_2(t), Nw^n_2(t); W^n(t) ) 
 \geq 0 \notag
\end{align}
for a.e. $t\in (0,T/2)$ with
\begin{equation}\label{eq:bc_W_n}
W^n (0) = W_0 ,
\end{equation}
and
\begin{equation}\label{eq:ineq_A_n}
    \Sigma^n (t) = \Sigma_0 + \int_0^t[\mathcal{G}(s,\alpha^n_1(s),Mw^n_1(s))- \mathcal{G}(s,\alpha^n_2(s),Mw^n_2(s))]ds
\end{equation}
\end{subequations}
for a.e. $t\in (0,T/2)$. To find the desired estimates, we need the following lemma.
\begin{lem}\label{lemma:est_W_n}
    Let $w^0_i = w_{0i} \in V$ and $\alpha^0_i = \alpha_{0i} \in Y$ for $i=1,2$. Under the assumptions of Theorem \ref{thm:mainresult}, let $\{(W^n,\Sigma^n)\}_{n\geq 1} \subset \mathcal{W}_{T/2}^{1,2} \times C([0,T/2];Y)$ be the solution to \eqref{eq:ineq_W_n}-\eqref{eq:ineq_A_n}. Then  
 \begin{align}\label{eq:estU}
     & \norm{W^n}_{L^2_{T/2}V}^2 \\
	&\leq c(\norm{W_0}_V^2 
	+  \norm{\Sigma^{n-1}}_{L^\infty_{T/2}Y}^2
         + \norm{w^{n-1}_1}_{L^2_{T/2}Y}^2 \norm{\Sigma^{n-1}}_{L^\infty_{T/2}Y}^2
         \notag \\
         &+
         T^k \norm{w_1^{n-1}}_{L^2_{T/2}V}^2\norm{W^{n-1}}_{L^2_{T/2}V}^2
             +  T^k \norm{\alpha_1^{n-1}}_{L^\infty_{T/2}Y}^2 \norm{W^{n-1}}_{L^2_{T/2}V}^2 )
            \notag 
             \\
             &+    \frac{\beta_{4\varphi} \norm{K} \norm{M}+ \beta_{5\varphi} \norm{K} \norm{M}\norm{\alpha_2^{n-1}}_{L^\infty_{T/2}Y}}{m_A - m_j\norm{N}^2} 
            \norm{W^{n-1}}_{L^2_{T/2}V}^2  \\
            &+ \frac{\beta_{4\varphi} \norm{K} \norm{M}+ \beta_{5\varphi} \norm{K} \norm{M}\norm{\alpha_2^{n-1}}_{L^\infty_{T/2}Y}}{m_A - m_j\norm{N}^2}  \norm{W^n}_{L^2_{T/2}V}^2\notag
 \end{align}
 for all $n\in \mathbb{Z}_+$ and some $k\geq 1/2$.
\end{lem}
The proof of Lemma \ref{lemma:est_W_n} is postponed to Appendix \ref{appendix:proof_W_n}. From \eqref{eq:T_k} and \eqref{eq:uniformlybounded}, the estimate  \eqref{eq:estU} becomes
        \begin{align*}
             \norm{W^n}_{L^2_{T/2}V}^2 &\leq c(\norm{W_0}_V^2 
         +\norm{\Sigma^{n-1}}_{L^\infty_{T/2}Y}^2)\\
                  &+      \frac{\beta_{4\varphi} \norm{K} \norm{M}+ \beta_{5\varphi} \norm{K} \norm{M}\norm{\alpha_2^{n-1}}_{L^\infty_{T/2}Y}}{m_A - m_j\norm{N}^2} 
            \norm{W^{n-1}}_{L^2_{T/2}V}^2  \\
            &+ \frac{\beta_{4\varphi} \norm{K} \norm{M}+ \beta_{5\varphi} \norm{K} \norm{M}\norm{\alpha_2^{n-1}}_{L^\infty_{T/2}Y}}{m_A - m_j\norm{N}^2}  \norm{W^n}_{L^2_{T/2}V}^2\notag
        \end{align*}
        In a similar manner as we obtained \eqref{eq:alpha_estimate2}, we see from \eqref{eq:ineq_A_n}, a standard Gr\"{o}nwall argument, the Cauchy-Schwarz inequality, and Young's inequality that
        \begin{align*}
            \norm{\Sigma^{n-1}}_{L^\infty_{T/2}Y}^2 
            	\leq c(\norm{\Sigma_0}_Y^2 + T^k \norm{W^{n-1}}_{L^2_{T/2}V}^2 ).
        \end{align*}
        By induction of $n$, it follows the same procedure as in Step \ref{convergence}.\ref{welldefinedscehem} that
        \begin{align*}
             \norm{(W^n,\Sigma^n)}_{L^2_{T/2}V \times L^\infty_{T/2}Y}^2 
            &\leq c(\norm{(W_0,\Sigma_0)}_{ V \times Y}^2 ) 
        \end{align*}
        for all $n\in \mathbb{Z}_+$. 
       Consequently, we may choose a $\delta >0$ in \eqref{eq:continuous1} to obtain \eqref{eq:continuous2}.\\
\indent
\textbf{Step \ref{mainproof}\rtask{mainproof}} \textit{(Proof of Theorem \ref{thm:mainresult})}.
	We now have all the tools to prove the main theorem.
	\begin{proof}[Proof of Theorem \ref{thm:mainresult}]
		Combining Step \ref{linearization}-\ref{continuous} gives us well-posedness of Problem \ref{prob:fullproblem}.
	\end{proof}

\section{Viscoelastic frictional contact problems}\label{sec:applications}
\noindent
We will present two applications to frictional contact; the first considering contact with normal compliance, and the second contact with normal damped response. Moreover, in Section \ref{sec:rateandstate}, we introduce a first-order approximation of a rate-and-state friction law that is covered by our framework. Let $u : \Omega \times [0,T] \rightarrow \R^d$ denote the displacement, $\sigma : \Omega \times [0,T] \rightarrow \mathbb{S}^d$ the stress tensor, and $\alpha : \Gamma_C \times [0,T] \rightarrow \R$ the external state variable. In addition, $f_0$ denotes the body forces, $f_N$ the surface traction, and $\rho$ the density.

\begin{figure}[H]
	\centering
	\includegraphics[scale=1]{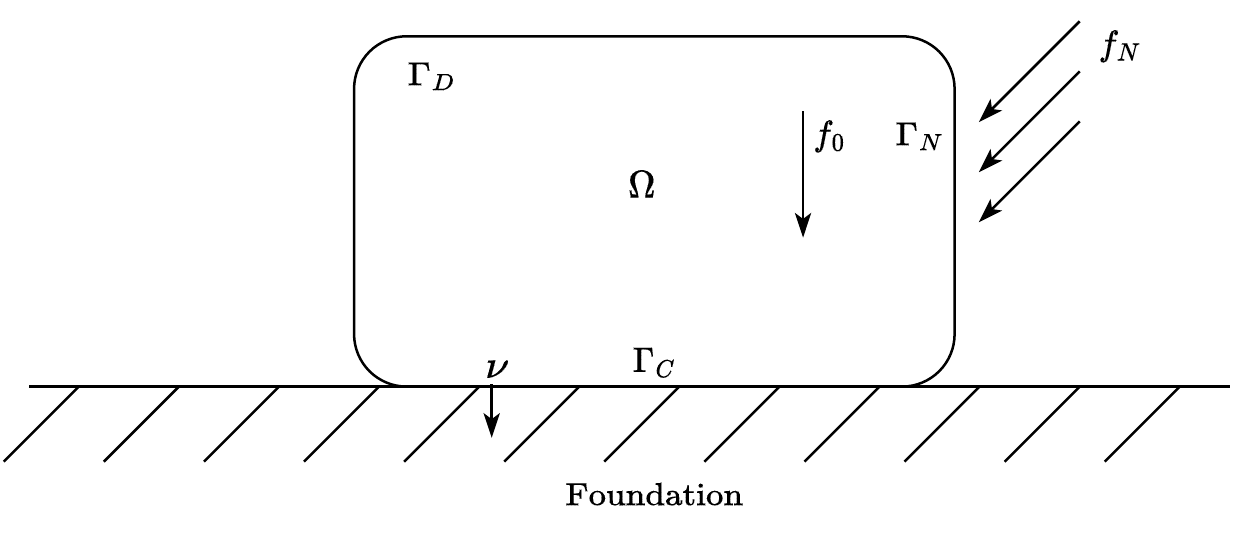}
	\captionsetup{justification=centering}
	\caption{A standard illustration of a sliding block.}
	\label{fig:klosse}
\end{figure}\noindent
We let the spaces $H=L^2(\Omega;\R^d)$, $Q=L^2(\Omega; \mathbb{S}^d)$,  $V$, and $\mathcal{W}^{1,2}_T$ be defined by \eqref{eq:H_space}, \eqref{eq:Q_space}, \eqref{eq:V_space}, and \eqref{eq:spaces_VastW}, respectively. We refer to Section \ref{sec:sobolev_spaces}-\ref{sec:bochner_spaces} for further definitions of the function spaces. Further, let $X=L^4(\Gamma_C)$ and $U=L^4(\Gamma_C;\mathbb{R}^d)$.
Let $\gamma_\nu : V \rightarrow X$ denote the normal trace operator, and $\gamma_\tau : V \rightarrow  U$ denote the tangential trace operator. It then follows by Theorem \ref{thm:trace} that $\gamma_\tau$ and $\gamma_\nu$ are well-defined for $d=2,3$. For all $v\in V$, we let $v_\nu = \gamma_\nu v = v  \cdot \nu$ denote the normal components on $\Gamma$, and $v_\tau =\gamma_\tau v = v - v_\nu\nu$ the tangential components on $\Gamma$. Similarly, let $\sigma_\nu = (\sigma\nu) \cdot \nu$, and $\sigma_\tau = \sigma\nu - \sigma_\nu\nu$ be the normal and tangential components of the tensor $\sigma$ on $\Gamma$, respectively.

\subsection{Dynamic frictional contact problem with normal compliance}\label{sec:application1}
In this section, we present a system of equations describing the evolution of a viscoelastic body in frictional contact with a foundation. Viscoelastic contact problems with normal compliance and friction are discussed in, e.g., \cite[Section 8.3]{shillor2004}. The normal compliance condition is used as an approximation of the Signorini non-penetration  condition. More on this can be found in \cite[Chapter 5]{han2002}, \cite[Chapter 11]{Kikuchi1988} and \cite{Sofonea2012}. We wish to study the following problem:
\begin{prob}\label{prob:application}
	Find the displacement $u: \Omega \times [0,T] \rightarrow \mathbb{R}^d$ and the external state variable $\alpha : \Gamma_C \times [0,T] \rightarrow \mathbb{R}$ such that
	\begin{subequations}
	\begin{align}
		\label{eq:sigma}
		\sigma(t) &= \mathcal{A} \varepsilon (\Dot{u}(t))  + \mathcal{B} \varepsilon (u(t)) + \int_0^t \mathcal{C}(t-s, \varepsilon(\Dot{u}(s))) ds 
		& \text{ on } \Omega \times (0,T)\\
		\label{eq:momentumeq}
		\rho \Ddot{u}(t)
		&= \nabla \cdot \sigma(t) + f_0(t) & \text{ on } \Omega \times (0,T)\\
		u(t) &= 0 & \text{ on } \Gamma_D \times (0,T)
		\label{eq:direchelet}
		\\
		\sigma(t) {\nu} &= f_N(t) & \text{ on } \Gamma_N \times (0,T)
		\label{eq:traction}
		\\
		-\sigma_\nu(t) &= p(u_\nu(t)) & \text{ on } \Gamma_C \times (0,T)
		\label{eq:sigma_nu}
		\\
		|\sigma_\tau(t)| &\leq  \mu (|\Dot{u}_\tau(t)|, \alpha(t))  |\sigma_\nu(t)|
            & \text{ on } \Gamma_C \times (0,T)
		\label{eq:prob1_friction_eq2}
		\\
		-\sigma_\tau(t) &=  \mu (|\Dot{u}_\tau(t)|, \alpha(t)) |\sigma_\nu(t)| \frac{\Dot{u}_\tau (t)}{|\Dot{u}_\tau(t)|}, \ \ \text{ if } \Dot{u}_\tau(t) \neq 0 & \text{ on } \Gamma_C \times (0,T)
		\label{eq:prob1_friction_eq}\\
		\Dot{\alpha}(t) &= G (\alpha(t), |\Dot{u}_\tau(t)|) & \text{ on } \Gamma_C \times (0,T)
		\label{eq:alphaeq}
		\end{align}
		with the initial conditions
		\begin{align}
			u(0) &= u_0, \ \ \Dot{u}(0) = w_0 & \text{ on } \Omega
			\label{eq:bbc1} \\
			\alpha (0) &= \alpha_{0}   & \text{ on } \Gamma_C.
			\label{eq:bbc}
		\end{align}
	\end{subequations}
\end{prob}
In the above problem, \eqref{eq:sigma} is a general viscoelastic constitutive law, where $\mathcal{A}$ is a viscosity operator, $\mathcal{B}$ an elasticity operator, and $\mathcal{C}$ is referred to as a relaxation tensor.  We note that $\mathcal{A}\varepsilon(\Dot{u})$ and $\mathcal{B}\varepsilon(u)$ are short-hand notation for $\mathcal{A}(x,\varepsilon(\Dot{u}))$ and $\mathcal{B}(x,\varepsilon(u))$, respectively. Moreover, \eqref{eq:momentumeq} is a momentum balance equation, \eqref{eq:direchelet} denotes the Dirichlet boundary conditions, and \eqref{eq:traction} the traction applied to the surface. The equation \eqref{eq:sigma_nu} is a contact condition, where $p$ is a prescribed function describing the penetration condition. Next, \eqref{eq:prob1_friction_eq2}-\eqref{eq:prob1_friction_eq} denotes a generalized Coulomb's friction law, and \eqref{eq:alphaeq} describes the evolution of the external state variable, see Section \ref{sec:intro_physics} for a discussion on this equation. Lastly, \eqref{eq:bbc1}-\eqref{eq:bbc} are initial conditions. We wish to investigate \eqref{eq:sigma}-\eqref{eq:bbc} under the following assumptions:
\\
$\underline{\H{\mathcal{A}}}$: \label{assumptionAcal}  $\mathcal{A} : \Omega  \times \mathbb{S}^d \rightarrow \mathbb{S}^d \text{ is such that }$
\begin{enumerate}[labelindent=0pt,labelwidth=\widthof{\ref{last-item}},label=(\roman*),itemindent=1em,leftmargin=!]
	\item For any $ \varepsilon \in \mathbb{S}^d, \ x\mapsto \mathcal{A} (x,\varepsilon)$ is measurable on $\Omega$.
	\item There exists $L_\mathcal{A} > 0$ such that $ |\mathcal{A} (x,\varepsilon_1)  - \mathcal{A} (x,\varepsilon_2)| \leq L_\mathcal{A} |\varepsilon_1 - \varepsilon_2|$ for all $\varepsilon_1,\varepsilon_2 \in \mathbb{S}^d$, a.e. $x \in \Omega$.
	\label{list:Acal_bounded}
	\item There exists $m_\mathcal{A}>0$ such that $ (\mathcal{A}(x,\varepsilon_1) - \mathcal{A}(x,\varepsilon_2)): (\varepsilon_1 - \varepsilon_2) \geq m_\mathcal{A} |\varepsilon_1 - \varepsilon_2|^2$, for all $\varepsilon_1, \varepsilon_2 \in \mathbb{S}^d$,  a.e. $x \in \Omega$.
	\label{list:Acal_maximalmonotone}
	\label{list:Acal_measurable}
	\item $\mathcal{A} (x,0) = 0$ for a.e. $x\in \Omega$.
	\label{list:Acal_0}
\end{enumerate}
\noindent
$\underline{\H{\mathcal{B}}}$: \label{assumptionB}  $\mathcal{B} : \Omega \times \mathbb{S}^d \rightarrow \mathbb{S}^d \text{ is such that }$
\begin{enumerate}[labelindent=0pt,labelwidth=\widthof{\ref{last-item}},label=(\roman*),itemindent=1em,leftmargin=!]
	\item For any $ \varepsilon \in \mathbb{S}^d, \ x\mapsto \mathcal{B} (x,\varepsilon)$ is measurable on $\Omega$.
	\label{list:Bcal_measurable}
	\item There exists $L_\mathcal{B} > 0$ such that $|\mathcal{B} (x,\varepsilon_1)  - \mathcal{B} (x,\varepsilon_2)| \leq L_\mathcal{B} |\varepsilon_1 - \varepsilon_2|$ for all $\varepsilon_1,\varepsilon_2 \in \mathbb{S}^d$, a.e. $x\in \Omega$.
	\label{list:Bcal_bounded}
	\item $\norm{\mathcal{B} (\cdot,0)}_Q < \infty $.
	\label{list:Bcal_0}
\end{enumerate}
\noindent
$\underline{\H{\mu}}$: \label{assumptionmu}  $ \mu : \Gamma_C \times \mathbb{R} \times \mathbb{R} \rightarrow \mathbb{R} \text{ is such that }$
\begin{enumerate}[labelindent=0pt,labelwidth=\widthof{\ref{last-item}},label=(\roman*),itemindent=1em,leftmargin=!]
	\item The mapping $ x\mapsto  \mu(x,r,y)$ is measurable on $\Gamma_C$ for all $r ,y \in \mathbb{R}$.
	\label{list:mu1_measurable}
	\item There exist $L_{1\mu},L_{2\mu},L_{3\mu} \geq 0$ such that 
    $|\mu(x,r_1,y_1)  - \mu (x,r_2,y_2)| \leq ( L_{1\mu} + L_{2\mu}|y_2|)|r_1-r_2| +  L_{3\mu}|r_1||y_1-y_2|  $ for all $r_1,r_2\in \mathbb{R}$, $y_1,y_2\in \mathbb{R}$, a.e. $x\in \Gamma_C$.
	\label{list:mu1_Lipschitz}
	\item There exist $\kappa_1,\kappa_2,\kappa_3 \geq 0$ such that $ |\mu(x,r,y) | \leq   \kappa_1 + \kappa_2|y|+\kappa_3|r|$ for all $r \in \mathbb{R}$, $y\in \mathbb{R}$, a.e. $x\in \Gamma_C$. 
	\label{list:mu1_est}
\end{enumerate}
\noindent
$\underline{\H{p}}$: \label{assumptionp}  $ p : \Gamma_C \times \mathbb{R} \rightarrow \mathbb{R} \text{ is such that }$
\begin{enumerate}[labelindent=0pt,labelwidth=\widthof{\ref{last-item}},label=(\roman*),itemindent=1em,leftmargin=!]
	\item The mapping $ x\mapsto  p(x,r)$ is measurable on $\Gamma_C$ for all $r \in \mathbb{R}$.
	\label{list:p_measurable}
	\item There exists $L_p > 0$ such that $|p(x,r_1)  - p (x,r_2)| \leq L_p |r_1-r_2|$ for all $r_1,r_2 \in \mathbb{R}$, a.e. $x\in \Gamma_C$.
	\label{list:p_Lipschitz}
	\item There exists $ p^\ast >0 $ such that $p(x,r) \leq p^\ast$ for all $r\in \R$, a.e. $x\in \Gamma_C$.
	\label{list:p_bounded}
\end{enumerate}
\noindent
$\underline{\H{G}}$: \label{assumptionGapp}  $G : \Gamma_C \times \mathbb{R} \times \mathbb{R} \rightarrow \mathbb{R} \text{ is such that }$
\begin{enumerate}[labelindent=0pt,labelwidth=\widthof{\ref{last-item}},label=(\roman*),itemindent=1em,leftmargin=!]
	\item The mapping $x\mapsto  G(x,\alpha,r)$ is measurable on $\Omega$ for all $\alpha,r \in \mathbb{R}$.
	\label{list:G_app_measurable}
	\item There exists $L_G > 0$ such that $|G(x,\alpha_1,r_1)  - G (x,\alpha_2,r_2)| \leq  L_G (|\alpha_1-\alpha_2| + |r_1-r_2|)$ for all $\alpha_1,\alpha_2,r_1,r_2 \in \mathbb{R}$, a.e. $x\in \Gamma_C$.
	\label{list:G_app_Lipschitz}
	\item  $\norm{G(\cdot,0,0)}_{L^2(\Gamma_C)} < \infty$.
	\label{list:G_app_0}
\end{enumerate}
\noindent
$\underline{\H{\mathcal{C}}}$: \label{assumptionC}  $\mathcal{C} :  \Omega \times (0,T)  \times \mathbb{S}^d \rightarrow \mathbb{S} \text{ is such that }$
\begin{enumerate}[labelindent=0pt,labelwidth=\widthof{\ref{last-item}},label=(\roman*),itemindent=1em,leftmargin=!]
	\item $\mathcal{C}(x,t,\varepsilon) = (c(x,t) \varepsilon)$ for all $\varepsilon \in \mathbb{S}^d$, a.e. $(x,t)\in  \Omega \times (0,T)$. 
	\item $c(x,t) = (c_{ijkl}(x,t))$ with $c_{ijkl} = c_{jikl} = c_{lkij} \in L^\infty_TL^\infty(\Omega)$.
\end{enumerate}
\begin{equation}\label{eq:denisty}
	\text{The mass density is assumed to be a positive constant } \rho>0 \text{ and}
\end{equation}
\begin{equation}\label{eq:assumptiononsourceterms}
	f_0 \in L^2_TH, \ \  f_N \in L^2_TL^2(\Gamma_N;\mathbb{R}^d)
\end{equation}
with the initial data satisfying
\begin{subequations}\label{eq:assumptionondata}
    \begin{align}\label{eq:assumptionondata1}
	   u_0,\ &w_0 \in V, \\
        \alpha_0 &\in  L^2(\Gamma_C).
        \label{eq:assumptionondata2}
    \end{align}
\end{subequations}

\begin{remark}
	Similar assumptions on the operators and data are found in, e.g., \cite{Migorski2022,Patrulescu2017,sofonea2017}. In comparison,  the assumptions  \hyperref[assumptionmu]{$\H{\mu}$}\ref{list:mu1_Lipschitz}-\ref{list:mu1_est} are generalized. In particular, we relaxed the boundedness assumption on $\mu$.
\end{remark}
We refer the reader to Appendix \ref{appendix:comments_app} for a discussion on applications under these assumptions.

\subsubsection{Variational formulation} \label{sec:variationalformulation_1}
We find a formal derivation of the variational formulation of Problem \ref{prob:application}, i.e., assuming sufficiently regular functions, as we only are interested in a mild solution (see Definition \ref{def:sols}). We refer to, e.g., \cite[Section 5.2]{shillor2004} for a more detailed derivation, especially how to deal with the contact conditions. Inserting \eqref{eq:momentumeq} into Green's formula \eqref{eq:greensformulasigma} yields
\begin{align*}
	\int_{\Omega} \rho \Ddot{u}(t) \cdot \big[v-\Dot{u}(t)\big] dx + \int_{\Omega}  \sigma(t) : \big[\varepsilon (v)-\varepsilon (\Dot{u}(t))\big] dx  &= \int_{\Omega} f_0(t) \cdot \big[v-\Dot{u}(t)\big] dx \\
	&+ \int_{\Gamma}  \sigma(t)\nu \cdot \big[v-\Dot{u}(t)\big] da
\end{align*}
for all $v\in V$, a.e. $t\in (0,T)$. From the terms on $\Gamma_C$, we deduce
\begin{align*}
	&\int_{\Gamma_C}  \sigma_\tau(t) \cdot \big[v_\tau - \Dot{u}_\tau(t)\big] da  + \int_{\Gamma_C}  \sigma_\nu(t) \big[v_\nu - \Dot{u}_\nu(t) \big] da \\
	&\geq  \int_{\Gamma_C}    \mu (|\Dot{u}_\tau(t)|, \alpha(t)) p(u_\nu(t)) \big[ |\Dot{u}_\tau(t)| - |v_\tau| \big] da + \int_{\Gamma_C} p(u_\nu(t))   \big[\Dot{u}_\nu(t) - v_\nu \big] da
\end{align*}
for all $v\in V$, a.e. $t \in (0,T)$.
Combining the above reads
\begin{align*}
	&\int_{\Omega} \rho \Ddot{u}(t) \cdot  \big[v-\Dot{u}(t)\big] dx + \int_{\Omega}  \sigma (t) : \big[\varepsilon (v)- \varepsilon (\Dot{u}(t))\big] dx   \\ \notag
	&+ \int_{\Gamma_C}    \mu (|\Dot{u}_\tau(t)|, \alpha(t)) p(u_\nu(t)) \big[|v_\tau |-|\Dot{u}_\tau(t) |\big]  da + \int_{\Gamma_C} p(u_\nu(t))   \big[v_\nu - \Dot{u}_\nu(t) \big] da \\
	& \geq \int_{\Gamma_N} f_N(t) \cdot \big[v-\Dot{u}(t)\big] da + \int_{\Omega} f_0(t) \cdot \big[v-\Dot{u}(t)\big] dx \notag
\end{align*}
for all $v\in V$, a.e. $t\in (0,T)$.
We write the above inequality slightly  more compactly. We observe that the map $t \mapsto \int_{\Gamma_N} f_N(t) \cdot v da + \int_{\Omega} f_0(t) \cdot v dx$ is linear and bounded in $V$.  Consequently, the Riesz representation theorem implies the existence of $f(t) \in V^\ast$ such that
\begin{align}\label{eq:f_inner}
	\inner{f(t)}{v} = \scalarprod{f_N(t)}{v}_{L^2(\Gamma_N;\R^d)} 
	 + \scalarprod{ f_0(t)}{v}_{H} \ \ \text{ for all } v\in V, \ \ \text{ a.e. } t\in (0,T).
\end{align}
As mentioned, we are interested in a mild solution of  \eqref{eq:alphaeq} (see Definition \ref{def:sols}), so we integrate \eqref{eq:alphaeq} over the time interval $(0,t)$ and use the initial condition \eqref{eq:bbc} to obtain this equation on the desired form. We may now formulate a variational inequality of Problem \ref{prob:application}. 
\begin{prob}\label{prob:weaksol}
	Find $u :  \Omega \times [0,T] \rightarrow \mathbb{R}^d$ and $\alpha: \Gamma_C \times [0,T]\rightarrow \mathbb{R}$ such that
 \begin{subequations}
	\begin{align}\label{eq:alphayo}
		&\alpha(t) = \alpha_0 + \int_0^t G (\alpha(s),  |\Dot{u}_\tau(s)|) ds,\\
		&\int_{\Gamma_C}    \mu (|\Dot{u}_\tau(t)|, \alpha(t)) p(u_\nu(t)) \big[|v_\tau |-|\Dot{u}_\tau(t) |\big] da + \int_{\Gamma_C} p(u_\nu(t))   \big[v_\nu - \Dot{u}_\nu(t) \big] da \label{eq:restofeq} \\
		&+\int_{\Omega} \rho \Ddot{u}(t) \cdot  \big[v-\Dot{u}(t)\big] dx + \scalarprod{\mathcal{A} \varepsilon (\Dot{u}(t))}{ \varepsilon (v)- \varepsilon (\Dot{u}(t)) }_Q  \notag\\
		&+ \scalarprod{\mathcal{B} \varepsilon (u(t)) +  \int_0^t \mathcal{C}(t-s, \varepsilon(\Dot{u}(s))) ds  }{ \varepsilon (v)- \varepsilon (\Dot{u}(t))}_Q \geq  \inner{f(t)}{v-\Dot{u}(t)} \notag
	\end{align}
 \end{subequations}
	for all $v\in V$, a.e. $t\in (0,T)$ with
	\begin{equation*}
		u(0) = u_0, \ \ \Dot{u}(0) = w_0.
	\end{equation*}
\end{prob}\noindent 
\begin{remark}
	Conversely, under the assumption of sufficient regularity, showing that \eqref{eq:alphayo} is equivalent to \eqref{eq:alphaeq} together with \eqref{eq:bbc} is an application of the fundamental theorem of calculus. Moreover, choosing the test functions $v=\Dot{u} \pm \Tilde{v}$ with $\Tilde{v} \in C^\infty_c(\Omega)$ in \eqref{eq:restofeq} implies that Problem \ref{prob:weaksol} is indeed equivalent to Problem \ref{prob:application}, see, e.g., \cite[Section 2.6]{Pipping2015_phd} or \cite[Section 5.2]{shillor2004}.
\end{remark}

The well-posedness result for Problem \ref{prob:weaksol} is summarized below.
\begin{thm}\label{thm:finalthm}
	Assume that \hyperref[assumptionAcal]{$\H{\mathcal{A}}$}, \hyperref[assumptionB]{$\H{\mathcal{B}}$}, \hyperref[assumptionC]{$\H{\mathcal{C}}$}, \hyperref[assumptionmu]{$\H{\mu}$}, \hyperref[assumptionp]{$\H{p}$}, \hyperref[assumptionGapp]{$\H{G}$}, and \eqref{eq:denisty}-\eqref{eq:assumptionondata} holds. Then, there exists a $T>0$ satisfying \eqref{eq:times_T} such that Problem \ref{prob:weaksol} has at most one solution $(u,\alpha)$ under the smallness-assumption
	\begin{align}\label{eq:smallness1}
		m_\mathcal{A} > \sqrt{2} \: p^\ast &(L_{1\mu}\sqrt{\mathrm{meas}(\Gamma_C)}  + L_{2\mu}\norm{\alpha_0}_Y)\\
		&\times \norm{(\gamma_\tau,\gamma_\nu)}_{\mathcal{L}(V,L^4(\Gamma_C;\R^d) \times L^4(\Gamma_C))}\norm{\gamma_\tau}_{\mathcal{L}(V,L^4(\Gamma_C;\R^d))}. \notag
	\end{align}
	In addition $(u,\alpha)$ has the following regularity: 
	\begin{equation*}
		u\in W^{1,2}(0,T;V),\ \ \Dot{u} \in \mathcal{W}^{1,2}_T \subset C([0,T];H), 
        \ \ \alpha \in C([0,T];L^2(\Gamma_C)).
	\end{equation*}
	Moreover, there exists a neighborhood around $(u_0,w_0,\alpha_0)$ so that the flow map $\Tilde{F} :   V \times V \times L^2(\Gamma_C) \rightarrow W^{1,2}(0,T/2;V) \times C([0,T/2]; L^2(\Gamma_C))$ defined by $(u_0,w_0,\alpha_0) \mapsto (u,\alpha)$ is continuous.
\end{thm}

\begin{remark}
	A similar smallness-assumption \eqref{eq:smallness1} is used in, e.g., \cite{Patrulescu2017}.
\end{remark}

\begin{remark}\label{remark:rho}
	We will show that there exists a solution for $\rho \equiv 1$. For $\rho>0$, we may take $w(t) = w(\rho t)$. 
\end{remark}

\subsubsection{Proof of Theorem \ref{thm:finalthm}}
Our aim is to use Theorem \ref{thm:mainresult} to prove Theorem \ref{thm:finalthm}. Our first task is to rewrite Problem \ref{prob:weaksol} in the same form as Problem \ref{prob:fullproblem}. Then, we will verify the hypothesis of Theorem \ref{thm:mainresult}. 
\\
\indent
Let
$Y=L^2(\Gamma_C)$
and $Z=U\times X$.
We then define the operators $A : (0, T) \times V \rightarrow V^\ast $, $\mathcal{R} : L^2_TV \rightarrow L^2_TV^\ast$, $\mathcal{S}_\varphi  : L^2_TV \rightarrow L^2_T X$,  $\mathcal{G} : (0,T) \times Y \times U \rightarrow Y$, $M  : V \rightarrow U$, and $N : V \rightarrow X$, respectively, by
\begin{subequations}
	\begin{align}\label{eq:defining_ARS}
		\inner{A(t,w)}{v} &= \int_\Omega \mathcal{A} \varepsilon (w) : \varepsilon(v) dx,  \  \text{ for } w,v\in V, \text{ a.e. } t\in (0,T),\\
		\inner{\mathcal{R}w(t)}{v} &= \int_\Omega \mathcal{B} (
		\varepsilon(\int_0^t w(s) ds  + u_0)) : \varepsilon(v) dx \label{eq:definition_R}\\
		&+ \int_\Omega  \int_0^t \mathcal{C}(t-s, \varepsilon(w(s))) ds :  \varepsilon(v) dx, \notag \\
		&  \text{for } w\in L^2_TV, \ v\in V,  \text{ a.e. } t\in (0,T), \notag\\
		\mathcal{S}_\varphi w(t) &= \gamma_\nu \Big( \int_0^t w(s) ds + u_{0} \Big) , \ \text{ for } w\in L^2_TV, \text{ a.e. } t\in (0,T), \label{eq:definition_S}\\
		\mathcal{G}(t,\alpha,Mw) &= G(\alpha,|w_{\tau}|), \text{ for } \alpha \in Y, \ w \in V,  \text{ a.e. } t\in (0,T),
		\label{eq:and_G}\\
		 Mv = &v_\tau, \ \ Nv = v_\nu, \ \ \text{ for  } v\in V.
		\label{eq:def_MN}
	\end{align}
We define $K : V \rightarrow Z$ by 
\begin{equation}\label{eq:def_K}
	Kv = (Mv,Nv), \ \ \text{ for } v\in V,
\end{equation}
which is linear and bounded in both arguments, i.e., $K \in \mathcal{L}(V,Z)$. We consider the functional $\varphi : (0, T) \times Y \times X \times U \times Z \rightarrow \mathbb{R}$ by
\begin{equation}\label{eq:defintion_phi}
	\varphi(t,y,z,\Tilde{w},\Tilde{v}) = \int_{\Gamma_C} p(z)  v^{(2)} da + \int_{\Gamma_C}   \mu (|\Tilde{w}| , y) p(z) |v^{(1)}| da
\end{equation}
for $y \in Y$, $z\in X$, $\Tilde{w} \in U$, $\Tilde{v}=(v^{(1)},v^{(2)}) \in Z$, a.e. $t\in (0,T)$. We let $f:(0,T) \rightarrow V^\ast$ be defined as in \eqref{eq:f_inner}.
\end{subequations}
This yields the following generalization of Problem \ref{prob:weaksol}:
\begin{prob}\label{prob:soondone}
	Find $w\in \mathcal{W}^{1,2}_T$ and $\alpha \in C([0,T];Y)$ 	such that
	\begin{align*}
			&\alpha(t) =  \alpha_0 + \int_0^t \mathcal{G}(s,\alpha(s),Mw(s))ds,\\
		&\inner{\rho\Dot{w}(t)}{v-w(t)}+ \inner{A(t,w(t))}{v-w(t)}  + \inner{\mathcal{R}w(t)}{v-w(t)} \\
		&+ \varphi(t,  \alpha(t),\mathcal{S}_\varphi  w(t),Mw(t),Kv) - \varphi(t,  \alpha(t),\mathcal{S}_\varphi  w(t),Mw(t),Kw(t)) \geq \inner{f(t)}{v-w(t)}
	\end{align*}
	for all $v\in V$, a.e. $t\in (0,T)$,
	with $w(0) = w_0$.
\end{prob}
\begin{remark}
	Since Problem \ref{prob:weaksol} is contained in Problem \ref{prob:soondone}, it suffices to show well-posedness of Problem \ref{prob:soondone}. 
\end{remark}
\begin{lem}\label{lemma:assumptionphi2}
	Under the assumptions of Theorem \ref{thm:finalthm},	the hypothesis of Theorem \ref{thm:mainresult} holds for \eqref{eq:defining_ARS}-\eqref{eq:defintion_phi}. Here,  $c_{0\varphi}(t) = \kappa_1(\mathrm{meas}(\Gamma_C))^{3/2}$, $c_{1\varphi}= \kappa_2\mathrm{meas}(\Gamma_C)$, $c_{2\varphi} = L_p(\mathrm{meas}(\Gamma_C))^{5/4}  $, $c_{3\varphi} = \kappa_3(\mathrm{meas}(\Gamma_C))^{5/4}$,  $\beta_{1\varphi}= p^\ast L_{3\mu}$, $\beta_{3\varphi} = L_p(1+ \kappa_1) \sqrt{\mathrm{meas}(\Gamma_C)} $, $\beta_{4\varphi} = p^\ast  L_{1\mu} \sqrt{\mathrm{meas}(\Gamma_C)}$, $\beta_{5\varphi} = p^\ast L_{2\mu}$, $\beta_{6\varphi} = \kappa_3 L_p(\mathrm{meas}(\Gamma_C))^{1/4}$, $\beta_{7\varphi} =  \kappa_2 L_p$,   and $c_{4\varphi}=\beta_{2\varphi} = 0$.
\end{lem}
To maintain the flow of the article, the proof of Lemma \ref{lemma:assumptionphi2} is placed in Appendix \ref{appendix:proofs_app1}. We are now ready to prove Theorem \ref{thm:finalthm}.

\begin{proof}[Proof of Theorem \ref{thm:finalthm}]
	The proof relies on Theorem \ref{thm:mainresult}
	with $j^\circ \equiv 0$.  In light of Lemma \ref{lemma:assumptionphi2}, the hypotheses of Theorem \ref{thm:mainresult} are fulfilled.	Consequently, $(w,\alpha)\in \mathcal{W}^{1,2}_T \times C([0,T];Y)$ is a unique solution of Problem \ref{prob:soondone}. Moreover, we define the function $u : [0,T] \rightarrow V$ by
	\begin{equation}\label{eq:definition_of_u}
		u(t) = u_0 + \int_0^t w(s) ds  
	\end{equation}
        for all $t\in[0,T]$. As a consequence of Bochner space theory, the fact that $w \in L^2_TV \subset L^1_T V$ and \eqref{eq:definition_of_u}, we have that $u \in C([0,T];V)$ and $\Dot{u} = w$.        Consequently, $\Dot{u} \in \mathcal{W}^{1,2}_T$, which implies $u \in W^{1,2}(0,T;V)$. Furthermore, we define the set 
	\begin{equation*}
		\mathbb{Y} =   V \times V \times Y,
	\end{equation*}
	and show that the flow map $\Tilde{F} : \mathbb{Y} \rightarrow W^{1,2}(0,T/2;V)  \times C([0,T/2];Y)$ defined by $(u_{0},w_{0},\alpha_{0}) \mapsto(u,\alpha) $ is continuous. That is, we claim that for all $\lambda >0$, there exists a $\delta>0$, chosen later,
	such that 
	\begin{subequations}
		\begin{equation}\label{eq:cont_1111}
			\norm{(u_{01},w_{01},\alpha_{01})- 	(u_{02},w_{02},\alpha_{02})}_\mathbb{Y} < \delta
		\end{equation}
		implies
		\begin{equation}\label{eq:lastlast123}
			\norm{(u_1,\alpha_1) - (u_2,\alpha_2)}_{W^{1,2}(0,T/2;V) \times L^\infty_{T/2}Y} < 	\lambda.
		\end{equation}
	\end{subequations}
	To check this, let us use the continuous dependence result in Theorem \ref{thm:mainresult}. 	We observe that
	\begin{align}\label{eq:oneestleft}
		&\norm{(w_{01},\alpha_{01})- 	(w_{02},\alpha_{02})}_{  V \times Y}^2 \leq \norm{(u_{01},w_{01},\alpha_{01})-	(u_{02},w_{02},\alpha_{02})}_{\mathbb{Y}}^2 .
	\end{align}
	From \eqref{eq:cont_1111}, \eqref{eq:oneestleft}, and Theorem \ref{thm:mainresult}\ref{list:continuous_dependence_on_inital_data}, we obtain
	\begin{equation}\label{eq:estwest}
		\norm{(w_1,\alpha_1) - (w_2,\alpha_2)}_{L^2_{T/2}V \times L^\infty_{T/2}Y }^2 < \delta.
	\end{equation}
	Next, by \eqref{eq:definition_of_u}, the triangle inequality, Minkowski's inequality, Young's inequality, the Cauchy-Schwarz inequality, and integrating over the time interval $(0,T/2)$, we have
	\begin{align}\label{eq:thisisit}
		\norm{u_1 - u_2}_{L^2_{T/2}V}^2 \leq c(T\norm{u_{01}- u_{02}}_V^2  + T^2\norm{w_1 - w_2}_{L^2_{T/2}V}^2).
	\end{align}
	Combining  \eqref{eq:estwest}-\eqref{eq:thisisit}, while remembering that $w_i = \Dot{u}_i$ for $i=1,2$, we may choose a $\delta  >0$ so that \eqref{eq:lastlast123} is acquired.
	\end{proof}

\subsection{Dynamic frictional contact problem with normal damped response}\label{sec:application2}
In our second application, we consider contact with
normal damped response, i.e., a wet material or some lubrication between the foundation and the reference configuration of a viscoelastic body (see, e.g., \cite{Sofonea2012}). We present the problem:
\begin{prob}\label{prob:application3}
	Find the displacement $u: \Omega \times [0,T] \rightarrow \mathbb{R}^d$ and the external state variable $\alpha : \Gamma_C \times [0,T] \rightarrow \mathbb{R}$ such that
	\begin{subequations}
		\begin{align}\label{eq:constitutive} 
			\sigma(t) &= \mathcal{A} \varepsilon (\Dot{u}(t))  + \mathcal{B} \varepsilon (u(t)) + \int_0^t \mathcal{C}(t-s, \varepsilon(\Dot{u}(s))) ds 
			& \text{ on } \Omega \times (0,T)\\
			\label{eq:momentumeq3}
			\rho \Ddot{u}(t)
			&= \nabla \cdot \sigma(t) + f_0(t) & \text{ on } 	\Omega \times (0,T)\\
			 u(t) &= 0 & \text{ on } \Gamma_D \times (0,T)\\
			\sigma(t) {\nu} &= f_N(t) & \text{ on } \Gamma_N 	\times (0,T)
			\label{eq:sigmanu2}
			\\
			-\sigma_\nu(t) &\in  \partial 	j_\nu(\Dot{u}_\nu(t)) & \text{ on } \Gamma_C \times (0,T)
			\label{eq:sigma_nu2}
			\\
			|\sigma_\tau(t)| &\leq  \mu (|\Dot{u}_\tau(t)|, \alpha(t))
                & \text{ on } \Gamma_C \times (0,T)
			\label{eq:prob1_friction_eq4}
			\\
			-\sigma_\tau(t) &= \mu (|\Dot{u}_\tau(t)|, \alpha(t)) 	\frac{\Dot{u}_\tau (t)}{|\Dot{u}_\tau(t)|},  \ \ \text{ if } \Dot{u}_\tau(t) \neq 0 & \text{ on } \Gamma_C \times (0,T)
			\label{eq:prob1_friction_eq3}\\
			\Dot{\alpha}(t) &= G (\alpha(t), |\Dot{u}_\tau(t)|) & 	\text{ on } \Gamma_C \times (0,T)
			\label{eq:alpha2}
		    \end{align}
		    with
		    \begin{align}
			u(0) &= u_0, \ \ \Dot{u}(0) = w_0 & \text{ on } \Omega\\ 
			\alpha (0) &= \alpha_{0}   & \text{ on } \Gamma_C . 
			\label{eq:laaaaaaaaast}
		\end{align}
	\end{subequations}
\end{prob}
We summarized \eqref{eq:constitutive}-\eqref{eq:sigmanu2}, and \eqref{eq:alpha2}-\eqref{eq:laaaaaaaaast} underneath Problem \ref{prob:application}. However, \eqref{eq:sigma_nu2} is a general form of the contact condition for normal damped response describing the contact with a lubricated foundation \cite[Section 6.3]{Migorski2012}. The equations \eqref{eq:prob1_friction_eq4}-\eqref{eq:prob1_friction_eq3} is a version of Coulomb's law of dry friction, where \eqref{eq:prob1_friction_eq4}-\eqref{eq:prob1_friction_eq3} is a generalization of \cite[Problem 68, p.268]{sofonea2017}.
We  investigate Problem \ref{prob:application3} under the hypotheses of \hyperref[assumptionA]{$\H{A}$}, \hyperref[assumptionB]{$\H{\mathcal{B}}$}, \hyperref[assumptionmu]{$\H{\mu}$}, \hyperref[assumptionGapp]{$\H{G}$}, \hyperref[assumptionC]{$\H{\mathcal{C}}$}, and \eqref{eq:denisty}-\eqref{eq:assumptionondata}. In addition, we require an assumption on $j_\nu$ in \eqref{eq:sigma_nu2}:\\
\\
$\underline{\H{j_\nu}}$: \label{assumptionjnu}  $ j_\nu : \Gamma_C  \times \mathbb{R} \rightarrow \mathbb{R} \text{ is such that }$
\begin{enumerate}[labelindent=0pt,labelwidth=\widthof{\ref{last-item}},label=(\roman*),itemindent=1em,leftmargin=!]
	\item $j_\nu(\cdot,r)$ is measurable on $\Gamma_C$ for all $r\in \mathbb{R}$, and there exists $\Bar{e} \in L^4(\Gamma_C)$ such that $j_\nu(\cdot,\Bar{e}(\cdot)) \in L^1(\Gamma_C )$.
	\item $j_\nu(x,\cdot)$ is locally Lipschitz on $\mathbb{R}$ for a.e. $ x\in \Gamma_C $. 
	\item $|\partial j_\nu (x,r)| \leq \Bar{c}_0 + \Bar{c}_1 |r| $ for all $r \in \mathbb{R}$, a.e. $x \in \Gamma_C$ with $\Bar{c}_0, \Bar{c}_1 \geq 0$.
 \label{list:j_nu_bounded}
	\item $j_\nu^0 (x, r_1; r_2 - r_1) + j_\nu^0 (x, r_2; r_1 - r_2) \leq m_{j_\nu} |r_1 - r_2|^2$ for all $r_1, r_2 \in \R$, a.e. $ x\in \Gamma_C $ with $m_{j_\nu} \geq 0$. 
	\label{list:j_nu_est}
\end{enumerate}
We refer the reader to Appendix \ref{appendix:comments_app} for a small discussion on the assumptions on the operators and functions in Problem \ref{prob:application3}.

\subsubsection{Variational formulation}
We make use of the derivations in Section \ref{sec:application1}, but include the new term for the normal stress. By definition of the Clarke subgradient (see Definition \ref{def:subdiffernetial}) and \eqref{eq:sigma_nu2}, we have
\begin{align*}
	-\sigma_\nu(t) \Tilde{v}_\nu \leq j_\nu^0 (\Dot{u}_\nu(t);\Tilde{v}_\nu) \ \ \text{ for all } \Tilde{v}\in V, \ \text{a.e. } t\in (0,T).
\end{align*} 
Integrating over $\Gamma_C$, and choosing $\Tilde{v}_\nu = v_\nu-\Dot{u}_\nu(t)$ gives us
\begin{align}\label{eq:ineq_jnu0}
	\int_{\Gamma_C}\sigma_\nu(t) [v_\nu - \Dot{u}_\nu(t)] da \leq \int_{\Gamma_C} j_\nu^0 (\Dot{u}_\nu(t);v_\nu-\Dot{u}_\nu(t))da
\end{align}
for all $v\in V$, a.e. $t\in(0,T)$.
Combining \eqref{eq:ineq_jnu0} with the calculations from Section \ref{sec:variationalformulation_1}, we have the following problem:
\begin{prob}\label{prob:weaksol2}
	Find $u :  \Omega \times [0,T] \rightarrow \mathbb{R}^d$ and $\alpha: \Gamma_C \times [0,T]\rightarrow \mathbb{R}$ such that
	\begin{subequations}
		\begin{align}\label{eq:alpha_as_an_eq}
			&\alpha(t) = \alpha_0 +  \int_0^t G (\alpha(s),  |\Dot{u}_\tau(s)|) ds,\\ 
			\label{eq:as_an_ineq}
			&\int_{\Gamma_C}    \mu (|\Dot{u}_\tau(t)| , \alpha(t))[|v_\tau |-|\Dot{u}_\tau(t)|] da
			+ \int_{\Gamma_C} j^\circ_\nu (\Dot{u}_\nu(t); v_\nu - \Dot{u}_\nu(t)) da \\ \notag
			&+\int_{\Omega} \rho \Ddot{u}(t) \cdot  [v-\Dot{u}(t)] dx + \scalarprod{\mathcal{A} \varepsilon (\Dot{u}(t)) }{\varepsilon (v)- \varepsilon (\Dot{u}(t))}_Q\\ 
			 & + \scalarprod{\mathcal{B} \varepsilon (u(t)) +  \int_0^t \mathcal{C}(t-s, \varepsilon(\Dot{u}(s))) ds  }{ \varepsilon (v)- \varepsilon (\Dot{u}(t)) }_Q \geq  \inner{f(t)}{v-\Dot{u}(t)} \notag
		\end{align}
		for all $v\in V$, a.e. $t\in (0,T)$ with
		\begin{equation*}
			u(0) = u_0, \ \ \Dot{u}(0) = w_0.
		\end{equation*}
	\end{subequations}

\end{prob}\noindent 

The well-posedness result for Problem \ref{prob:weaksol2} is stated in the following theorem.

\begin{thm}\label{thm:wellposed_app2}
	Assume that  \hyperref[assumptionAcal]{$\H{\mathcal{A}}$}, \hyperref[assumptionB]{$\H{\mathcal{B}}$}, \hyperref[assumptionmu]{$\H{\mu}$}, \hyperref[assumptionGapp]{$\H{G}$}, \hyperref[assumptionC]{$\H{\mathcal{C}}$}, \eqref{eq:denisty}-\eqref{eq:assumptionondata}, and \hyperref[assumptionjnu]{$\H{j_\nu}$} holds. Then, there exists a $T>0$ satisyfing \eqref{eq:times_T} so that Problem \ref{prob:weaksol2} has a unique  solution $(u,\alpha)$
	under the smallness-condition 
	\begin{align}\label{eq:smallness_app2}
		m_\mathcal{A} &> m_{j_\nu} \sqrt{\mathrm{meas}(\Gamma_C)} \norm{\gamma_\nu}_{\mathcal{L}(V,L^4(\Gamma_C))}^2\\
		&+  \sqrt{2} \:( L_{1\mu}\sqrt{\mathrm{meas}(\Gamma_C)}+ L_{2\mu}\norm{\alpha_0}_Y) \norm{\gamma_\tau}^2_{\mathcal{L}(V,L^4(\Gamma_C;\R^d))}. \notag
	\end{align}
	In addition, we have the regularity: 
	\begin{equation}\label{eq:regularity_second_application}
		u\in W^{1,2}(0,T;V),\ \ \ \Dot{u} \in \mathcal{W}^{1,2}_T \subset C([0,T];H),  \ \ \ \alpha \in C([0,T];L^2(\Gamma_C)).
	\end{equation}
	Moreover, the flow map depends continuously on the initial data. 
\end{thm}
\begin{remark}
	The constraint \eqref{eq:smallness_app2} can also be found in, e.g., \cite[Theorem 4.4]{han} for  $\alpha_\varphi c_\alpha^2 = \sqrt{2} \:( L_{1\mu} \sqrt{\mathrm{meas}(\Gamma_C)}+ L_{2\mu}\norm{\alpha_0}_Y) \norm{\gamma_\tau}^2_{\mathcal{L}(V,L^4(\Gamma_C;\R^d))}$ and 
 $\alpha_j c_j^2 =m_{j_\nu} \sqrt{\mathrm{meas}(\Gamma_C)} \norm{\gamma_\nu}_{\mathcal{L}(V,L^4(\Gamma_C))}^2$.
\end{remark}
\begin{remark}
	We show that there exists a solution for $\rho \equiv 1$, and then take $w(t) = w(\rho t)$. 
\end{remark}
\subsubsection{Proof of Theorem \ref{thm:wellposed_app2}}
We use the same approach as in Section \ref{sec:application1}, meaning that we will use Theorem \ref{thm:mainresult} to prove Theorem \ref{thm:wellposed_app2}. But to use this theorem, we need to first rewrite Problem \ref{prob:weaksol2} into the same form as Problem \ref{prob:fullproblem}. Then we will verify the hypothesis of Theorem \ref{thm:mainresult}. 
\\
\indent
First recall that $X=L^4(\Gamma_C)$ and  $U=L^4(\Gamma_C; \R^d)$. Then take $Y=L^2(\Gamma_C)$ and $Z=U$.
We define  $A : (0, T) \times V \rightarrow V^\ast $, $\mathcal{R} : L^2_TV \rightarrow L^2_TV^\ast$, $\mathcal{G} : (0,T) \times Y \times U \rightarrow Y$, and $f : (0,T) \rightarrow V^\ast$ as in \eqref{eq:defining_ARS}-\eqref{eq:and_G} and \eqref{eq:f_inner}, respectively. We choose $M :  V \rightarrow U$ and $N : V \rightarrow X$ to be as in \eqref{eq:def_MN} and $K\equiv M$. Moreover, we define the functional $\varphi : (0, T) \times Y \times U \times Z \rightarrow \mathbb{R}$ by
\begin{subequations}
	\begin{equation}\label{eq:defintion_phi_2}
		\varphi(t,y,\Tilde{w},\Tilde{v}) = \int_{\Gamma_C}   \mu (|\Tilde{w}|,y) |\Tilde{v}| da \  \ \text{ for  } y\in Y, \ \Tilde{w} \in U,\ \Tilde{v} \in Z, \ \text{ a.e. } t\in (0,T)
	\end{equation}
	and the functional $j : (0,T)  \times X \rightarrow \mathbb{R}$ by
	\begin{align}\label{eq:j_nu}
		j(t,\Tilde{w}) = \int_{\Gamma_C} j_\nu (\Tilde{w}) da \  \ \text{ for  } \Tilde{w} \in X, \ \text{ a.e. } t\in (0,T).
	\end{align}
\end{subequations}
The problem is then on the following form.
\begin{prob}\label{prob:done_afterthis}
	Find $(w,\alpha)\in \mathcal{W}^{1,2}_T \times C([0,T];Y)$ such that
	\begin{align*}
		&\alpha(t) = \alpha_0 + \int_0^t \mathcal{G}(s,\alpha(s),Mw(s))ds,\\
		&\inner{\rho\Dot{w}(t)}{v-w(t)}+ \inner{A(t,w(t))}{v-w(t)} + \inner{\mathcal{R}w(t)}{v-w(t)}  + j^\circ(t,Nw(t);Nv-Nw(t))\\
		& + \varphi(t, \alpha(t),Mw(t),Kv) - \varphi(t, \alpha(t),Mw(t),Kw(t)) \geq \inner{f(t)}{v-w(t)}
	\end{align*}
	for all $v\in V$, a.e. $t\in (0,T)$ with $w(0) = w_0$.
\end{prob}
To see that it suffices to prove existence of a solution to Problem \ref{prob:done_afterthis} in order for Problem \ref{prob:weaksol2} to have a solution, we introduce the following result, which is of a similar form as found in \cite[Lemma 8, p.126]{sofonea2017} (see also \cite[Theorem 3.47]{Migorski2012}). The result will also be useful to prove uniqueness. 
\begin{cor}\label{cor:j}
	Assume that \hyperref[assumptionjnu]{$\H{j_\nu}$} holds. Then, the functional $j$ defined by \eqref{eq:j_nu} has the following properties:
		\begin{enumerate}[labelindent=0pt,labelwidth=\widthof{\ref{last-item}},label=(\roman*),itemindent=1em,leftmargin=!]
		\item $j(\cdot,v)$ is measurable on $(0,T)$ for all $v\in X$.
		\label{list:finite}
		\item $j(t, \cdot)$ is locally Lipschitz on $X$ for a.e. $t\in (0,T)$.
		\label{list:locally}
		\item For all $\Tilde{w},v \in X$, we have
		$ j^\circ (t,\Tilde{w};v) \leq \int_{\Gamma_C} j_\nu^\circ (\Tilde{w};v) da$.
		\label{list:equality}
	\end{enumerate}
\end{cor}
\begin{lem}\label{lemma:assumptionon_j}
	Under the assumptions of Theorem \ref{thm:wellposed_app2}, \hyperref[assumptionphi]{$\H{\varphi}$} holds for $\varphi$ defined by \eqref{eq:defintion_phi_2} for $c_{0\varphi}(t) = \kappa_1 (\mathrm{meas}(\Gamma_C))^{3/2}$, $c_{1\varphi}= \kappa_2  \mathrm{meas}(\Gamma_C)$, $ c_{3\varphi} = \kappa_3 (\mathrm{meas}(\Gamma_C))^{5/4}$,  $\beta_{1\varphi} =L_{3\mu}$,\\ $\beta_{4\varphi}  =L_{1\mu}\sqrt{\mathrm{meas}(\Gamma_C)}$, $\beta_{5\varphi} =L_{2\mu}$, and $c_{4\varphi}=\beta_{2\varphi} = \beta_{3\varphi}= \beta_{6\varphi} = \beta_{7\varphi}=0$. \\
 \indent Moreover, $j$ defined by \eqref{eq:j_nu} satisfies \hyperref[assumptionj]{$\H{j}$} for $c_{0j}(t) = 2^{1/4} (\mathrm{meas}(\Gamma_C))^{3/4}\Bar{c}_0$, $c_{3j} = 2^{1/4} (\mathrm{meas}(\Gamma_C))^{2/3}\Bar{c}_1$, $m_j = m_{j_\nu}\sqrt{\mathrm{meas}(\Gamma_C)}$, and $\Bar{m}_j = c_{1j} = c_{2j} = 0$.
\end{lem}
The proof of Lemma \ref{lemma:assumptionon_j} is placed in Appendix \ref{appendix:assumption_phiandj}. We may now prove Theorem \ref{thm:wellposed_app2}.
\begin{proof}[Proof of Theorem \ref{thm:wellposed_app2}]
	We wish to utilize Theorem \ref{thm:mainresult}.	The hypothesis of Theorem \ref{thm:mainresult} holds by \eqref{eq:smallness_app2}, Lemma \ref{lemma:assumptionphi2} and \ref{lemma:assumptionon_j}, and Corollary \ref{cor:j}\ref{list:finite}-\ref{list:locally}. With the help of Corollary \ref{cor:j}\ref{list:equality}, we may conclude that there exists a solution to Problem \ref{prob:weaksol2}. Moreover,  the fact that the flow map depends continuously on the initial data  follows by the same approach as in the proof of Theorem \ref{thm:finalthm}.
	To obtain uniqueness, we let $(u_1,\alpha_1),(u_2,\alpha_2) \in W^{1,2}(0,T;V) \times C([0,T];Y)$ be two pairs of solutions to Problem \ref{prob:weaksol2}. Choosing the test functions $v=\Dot{u}_2(t)$ and  $v=\Dot{u}_1(t)$ for a.e. $t\in (0,T)$, respectively, in \eqref{eq:as_an_ineq} yields
	\begin{align}\label{eq:inequality}
		&\scalarprod{ \mathcal{A} \varepsilon (\Dot{u}_1(t)) - \mathcal{A} \varepsilon (\Dot{u}_2(t)) }{ \varepsilon (\Dot{u}_1(t))- \varepsilon (\Dot{u}_2(t)) }_Q  \\ \notag
		&+ \int_{\Omega} \rho \big[\Ddot{u}_1(t)- \Ddot{u}_2(t)\big] \cdot  \big[\Dot{u}_1(t)-\Dot{u}_2(t)\big] dx\\ \notag
		&\leq
			\int_{\Gamma_C}    \big[\mu (|\Dot{u}_{1\tau}(t)| , \alpha_1(t)) - \mu (|\Dot{u}_{2\tau}(t)| , \alpha_2(t))\big]\big[|\Dot{u}_{2\tau}(t) |-|\Dot{u}_{1\tau}(t)|\big] da\\ \notag
			&+ \int_{\Gamma_C} [j^\circ_\nu (\Dot{u}_{1\nu}(t); \Dot{u}_{2\nu}(t) - \Dot{u}_{1\nu}(t)) +  j^\circ_\nu (\Dot{u}_{2\nu}(t); \Dot{u}_{1\nu}(t) - \Dot{u}_{2\nu}(t))] da\\ \notag
			&+ \scalarprod{ \mathcal{B} \varepsilon (u_1(t))- \mathcal{B} \varepsilon (u_2(t))}{ \varepsilon (\Dot{u}_2(t))- \varepsilon (\Dot{u}_1(t)) }_Q \\
			&+  \scalarprod{\int_0^t[ \mathcal{C}(t-s, \varepsilon(\Dot{u}_1(s)))  - \mathcal{C}(t-s, \varepsilon(\Dot{u}_2(s)))] ds  }{ \varepsilon (\Dot{u}_2(t))- \varepsilon (\Dot{u}_1(t)) }_Q \notag
	\end{align}
	for a.e. $t\in (0,T)$.  Using \hyperref[assumptionGapp]{$\H{G}$}\ref{list:G_app_Lipschitz}, a standard Gr\"{o}nwall argument, the Cauchy-Schwarz inequality, and Minkowski's inequality in \eqref{eq:alpha_as_an_eq} implies
	\begin{align}\label{eq:alphaineq}
		\int_{\Gamma_C}|\alpha_1(t)-\alpha_2(t)|^2 dx 
		&\leq c\int_0^t \int_{\Omega} |\Dot{u}_{1}(s) - \Dot{u}_{2}(s)|^2  dx ds
	\end{align}
	for a.e. $t\in (0,T$). We next utilize \hyperref[assumptionAcal]{$\H{\mathcal{A}}$}\ref{list:Acal_maximalmonotone}, \hyperref[assumptionB]{$\H{\mathcal{B}}$}\ref{list:Bcal_bounded}, \hyperref[assumptionC]{$\H{\mathcal{C}}$}, \hyperref[assumptionjnu]{$\H{j_\nu}$}\ref{list:j_nu_est}, \hyperref[assumptionmu]{$\H{\mu}$}\ref{list:mu1_Lipschitz}, \eqref{eq:denisty}, the Cauchy-Schwarz inequality, and Young's inequality to \eqref{eq:inequality}:
    \begin{align*}
        & m_A \norm{\varepsilon (\Dot{u}_1(t)) -  \varepsilon (\Dot{u}_2(t))}_Q^2   + \int_{\Omega} \rho \big[\Ddot{u}_1(t)- \Ddot{u}_2(t)\big] \cdot  \big[\Dot{u}_1(t)-\Dot{u}_2(t)\big] dx\\ \notag
		&\leq L_{1\mu} \sqrt{\mathrm{meas}(\Gamma_C)} \norm{\Dot{u}_{1\tau}(t) - \Dot{u}_{2\tau}(t)}_{L^4(\Gamma_C;\R^d)}^2
        \\
  &+ L_{2\mu} \norm{\alpha_1(t)}_{L^2(\Gamma_C)} \norm{\Dot{u}_{1\tau}(t) - \Dot{u}_{2\tau}(t)}_{L^4(\Gamma_C;\R^d)}^2\\
  &+L_{3\mu} \norm{\Dot{u}_{2\tau}(t)}_{L^4(\Gamma_C)} \norm{\alpha_1(t)-\alpha_2(t)}_{L^2(\Gamma_C)} \norm{\Dot{u}_{1\tau}(t) - \Dot{u}_{2\tau}(t)}_{L^4(\Gamma_C;\R^d)} \\
  & + m_{j\nu} \sqrt{\mathrm{meas}(\Gamma_C)}\norm{\Dot{u}_{1\nu}(t) - \Dot{u}_{2\nu}(t)}_{L^4(\Gamma_C)}^2 \\
  &+ L_\mathcal{B} \norm{\varepsilon (u_1(t)) -  \varepsilon (u_2(t))}_Q\norm{\varepsilon (\Dot{u}_1(t)) -  \varepsilon (\Dot{u}_2(t))}_Q\\
			&+ \norm{\mathcal{C}}_{L^\infty_TL^\infty(\Omega;\mathbb{S}^d)}    \norm{\int_0^t [ \varepsilon(\Dot{u}_1(s))  - \varepsilon(\Dot{u}_2(s))  ]ds}_Q \norm{ \varepsilon (\Dot{u}_2(t))- \varepsilon (\Dot{u}_1(t)) }_Q \notag
    \end{align*}
    for a.e. $t\in(0,T)$.    Next, we use the fact that if $\Dot{u}\in \mathcal{W}^{1,2}_T$, then $\frac{d}{dt}\norm{\Dot{u}(t)}_H^2  = 2 \inner{\Ddot{u}(t)}{\Dot{u}(t)}$ \cite[Theorem 3 in Section 5.9.2]{evans}, \eqref{eq:definition_of_u}, and integrating over the time interval $(0,t') \subset (0,T)$. We observe by \eqref{eq:smallness_app2}   that  $m_A -m_{j_\nu} \sqrt{\mathrm{meas}(\Gamma_C)} \norm{\gamma_\nu}_{\mathcal{L}(V,X)}^2>0$. In addition, we use \eqref{eq:alphaineq} and then apply H\"{o}lder's inequality and Minkowski's inequality to obtain
    \begin{align*}
         \norm{\Dot{u}_1 -  \Dot{u}_2}_{L^2_{t'}V}^2   &\leq \frac{(L_{1\mu} \sqrt{\mathrm{meas}(\Gamma_C)}    + L_{2\mu}\norm{\alpha_1}_{L^\infty_TY})\norm{\gamma_\tau}_{\mathcal{L}(V,U)}^2 }{m_A -m_{j_\nu} \sqrt{\mathrm{meas}(\Gamma_C)} \norm{\gamma_\nu}_{\mathcal{L}(V,X)}^2 }  \norm{\Dot{u}_1 -  \Dot{u}_2}_{L^2_{t'}V}^2 \\
  &+c \norm{\Dot{u}_{2}}_{L^2_TV} \norm{\alpha_1-\alpha_2}_{L^\infty_{t'}Y} \norm{\Dot{u}_1 -  \Dot{u}_2}_{L^2_{t'}V}\\
  &+ c
  \Big[ \int_0^{t'} \int_0^t \norm{\Dot{u}_1(s) - \Dot{u}_2(s)}_V^2 ds dt \Big]^{1/2}\norm{ \Dot{u}_1- \Dot{u}_2 }_{L^2_{t'} V}
   \\
   &=: I + II + III
    \end{align*}
    for a.e. $t'\in (0,T)$. Using  Minkowski's inequality, \hyperref[assumptionGapp]{$\H{G}$}\ref{list:G_app_Lipschitz}, and  the Cauchy-Schwarz inequality implies
    \begin{align*}
        \norm{\alpha_1(t)}_{L^2(\Gamma_C)} &\leq \norm{\alpha_0}_{L^2(\Gamma_C)} + L_G\int_0^t [\norm{\alpha_1(s)}_{L^2(\Gamma_C)} + \sqrt{\mathrm{meas}(\Gamma_C)} \norm{\Dot{u}_{1\tau}(s)}_{L^4(\Gamma_C)}]ds\\
       &\leq \norm{\alpha_0}_{L^2(\Gamma_C)} + L_G\int_0^t \norm{\alpha_1(s)}_{L^2(\Gamma_C)}ds + cT^k \Big(\int_0^t  \norm{\Dot{u}_1(s)}^2_Vds\Big)^{1/2}
    \end{align*}
    for a.e. $t\in (0,T)$ and $k\geq 1/2$. From Gr\"{o}nwall's inequality, we deduce
    \begin{align*}
        \norm{\alpha_1}_{L^\infty_TL^2(\Gamma_C)}
       &\leq (\norm{\alpha_0}_{L^2(\Gamma_C)} + cT^k\norm{\Dot{u}_1}_{L^2_TV})(1+cT\mathrm{e}^{cT})
    \end{align*}
    for some $k\geq 1/2$. Utilizing Young's inequality to $I$ and $II+III$ and then the arithmetic-quadratic mean inequality to the latter term while keeping in mind \eqref{eq:smallness_app2}-\eqref{eq:regularity_second_application} and \eqref{eq:alphaineq}, we obtain 
    \begin{align*}
         \norm{\Dot{u}_1 -  \Dot{u}_2}_{L^2_{t'}V}^2   
		&\leq\frac{ 2(L_{1\mu} \sqrt{\mathrm{meas}(\Gamma_C)}   + L_{2\mu}(\norm{\alpha_0}_Y + T^kc) )^2\norm{\gamma_\tau}_{\mathcal{L}(V,U)}^4}{(m_A -m_{j_\nu} \sqrt{\mathrm{meas}(\Gamma_C)} \norm{\gamma_\nu}_{\mathcal{L}(V,X)}^2 )^2 } \norm{\Dot{u}_1 -  \Dot{u}_2}_{L^2_{t'}V}^2 \\
        &+ c   \int_0^{t'} \int_0^t \norm{\Dot{u}_1(s) - \Dot{u}_1(s)}_V^2 ds dt 
    \end{align*}
    for all $t'\in (0,T)$. 
    A consequence of \eqref{eq:smallness_app2}  and choosing $T>0$ such that
    \begin{align*}
        T^k \sim \frac{m_A -m_{j_\nu} \sqrt{\mathrm{meas}(\Gamma_C)} \norm{\gamma_\nu}_{\mathcal{L}(V,X)}^2}{cL_{2\mu}\norm{\gamma_\tau}_{\mathcal{L}(V,U)}^2}
    \end{align*}
    small enough for some $k\geq 1/2$
    implies
    \begin{equation*}
      \frac{\sqrt{2}(L_{1\mu} \sqrt{\mathrm{meas}(\Gamma_C)}    + L_{2\mu}(\norm{\alpha_0}_Y + T^kc))\norm{\gamma_\tau}_{\mathcal{L}(V,U)}^2  }{m_A -m_{j_\nu} \sqrt{\mathrm{meas}(\Gamma_C)} \norm{\gamma_\nu}_{\mathcal{L}(V,X)}^2 } < 1,
    \end{equation*}
    we obtain
    \begin{align*}
         \norm{\Dot{u}_1 -  \Dot{u}_2}_{L^2_{t'}V}^2   \leq  c   \int_0^{t'} \norm{\Dot{u}_1 - \Dot{u}_2}_{L^2_tV}^2 dt 
    \end{align*}
    for all $t'\in (0,T)$.     Applying a standard Gr\"{o}nwall argument reads
	\begin{align}\label{eq:ineq_to_go}
		\norm{\Dot{u}_1 - \Dot{u}_2}_{L^2_{T}V}^2 \leq 0.
	\end{align}
	By the definition of $u$, i.e., \eqref{eq:definition_of_u},  the smallness-condition \eqref{eq:smallness_app2} and \eqref{eq:alphaineq}-\eqref{eq:ineq_to_go}, we conclude that $(u,\alpha)$ is the unique solution to Problem \ref{prob:weaksol2}.
\end{proof}

\subsection{Application to rate-and-state friction} \label{sec:rateandstate}
The coupling \eqref{eq:mu_alpha} are standard in geophysical applications in earth sciences, where the experimentally derived Dieterich-Ruina laws are commonly used. We refer to \cite{Marone1998} for an overview and comparison of some commonly used laws. There have been physical issues with the standard rate law, e.g., $|\Dot{u}_\tau(t)| \rightarrow 0$ resulting in a negative friction coefficient. This was repaired by using the regularized or truncated law (see \cite[Section 1.1-1.3]{Pipping2015_phd} and references therein), which are, respectively, given by
\begin{subequations}\label{eq:both1}
\begin{align}\label{eq:regularized}
	\mu(|\Dot{u}_\tau(t)|,\alpha(t)) &= a \: \mathrm{arcsinh} \bigg(\frac{|\Dot{u}_\tau(t)|}{2 v_\alpha(t)} \bigg), 
	\\
	\mu(|\Dot{u}_\tau(t)|,\alpha(t)) &=a\log^+ \bigg(\frac{|\Dot{u}_\tau(t)|}{v_\alpha(t)}\bigg) \ \  \text{ with } \log^+v= 
	\begin{dcases}
	\log v, \ \ \text{ if } v\geq 1,\\
	0,  \ \ \text{ otherwise},
	\end{dcases}
	\label{eq:truncated}
\end{align}	
\end{subequations}
where $v_\alpha(t) = v_0e^{-\frac{1}{a}(\mu_0 + b\alpha(t))}$. The coefficients $a$, $b$, $v_0$, and $\mu_0$ are system parameters (see, e.g., \cite{Marone1998,Helmstetter},\cite[Section 1.2]{Pipping2015_phd}). Another regularization used in literature can be found in \cite{Roubicek2014}. The most popular rate-and-state friction laws are the aging and slip laws, respectively, described by
\begin{subequations}\label{eq:both2}
    \begin{align}\label{eq:aginglaw}
	\Dot{\alpha} (t) &= \frac{v_0\mathrm{e}^{-\alpha(t)}- |\Dot{u}_\tau(t)|}{L},\\
    \Dot{\alpha}(t) &= -\frac{|\Dot{u}_\tau(t)|}{L}
    \Big[\log\Big(\frac{|\Dot{u}_\tau(t)|}{v_0}\Big) + \alpha(t)\Big],
     \label{eq:sliplawrate}
\end{align}
\end{subequations}
with $L$ being a system parameter (see, e.g., \cite{Helmstetter}). 
In the framework presented in this paper, we are not able to include \eqref{eq:both1}-\eqref{eq:both2}. The main issue is that  $|\Dot{u}_\tau(t)| \not\in L^\infty(\Gamma_C)$. We therefore consider a first-order Taylor approximation of $\mathrm{e}^{ \Hat{c} \alpha(t)}$ around $\alpha_0$ with $\Hat{c} = -1, \frac{b}{a}$, which is used in \eqref{eq:regularized} and \eqref{eq:aginglaw}. This leads to a first-order approximation of \eqref{eq:regularized} and \eqref{eq:aginglaw}.\\
\indent
The first-order Taylor approximation of $\mathrm{e}^{\Hat{c}\alpha(t)}$ around $\alpha_0$ reads
\begin{align}\label{eq:expo}
    \mathrm{e}^{\Hat{c}\alpha(t)} =  \mathrm{e}^{\Hat{c}\alpha_0}  (1 + \Hat{c}  (\alpha(t)-\alpha_0) )+ \mathcal{O} ((\alpha(t)-\alpha_0)^2).
\end{align}
Using the above approximation in \eqref{eq:regularized} and \eqref{eq:aginglaw} gives us
\begin{subequations}
    \begin{align}\label{eq:alpha_approx}
        G(\alpha(t),|\Dot{u}_\tau(t)|) &= \frac{v_0 \mathrm{e}^{-\alpha_0}(1- \alpha(t)  +\alpha_0 ) 
    - |\Dot{u}_\tau(t)|}{L}, \\ 
        \mu(|\Dot{u}_\tau(t)|,\alpha(t)) &=a\: \mathrm{arcsinh} \bigg(\frac{  \mathrm{e}^{\frac{\mu_0 + b \alpha_0 }{a}}|\Dot{u}_\tau(t)|(1+\frac{b}{a}(\alpha(t)-\alpha_0))}{2v_0} \bigg).
        \label{eq:mu_approx}
    \end{align}
\end{subequations}
We will make a formal argument to justify that these are first-order approximations. 
\begin{proof}[Formal augmentations of \eqref{eq:alpha_approx}-\eqref{eq:mu_approx}]
    For simplicity in notation, we let $y= \alpha(t)$ and $r=\Dot{u}_\tau(t)$. We wish to have the same order of error for the approximation of \eqref{eq:regularized} and \eqref{eq:aginglaw} as in \eqref{eq:expo}. 
    For \eqref{eq:alpha_approx}, we are considering $\Hat{c} = -1$. 
    We directly obtain
    \begin{align*}
       \frac{v_0}{L}(\mathrm{e}^{-y}-\mathrm{e}^{-\alpha_0}(1-(y-\alpha_0))) =    \mathcal{O} ((y-\alpha_0)^2).
    \end{align*}
    Next, considering \eqref{eq:mu_approx}, we are interested in the approximation \eqref{eq:expo} for $\Hat{c} = \frac{b}{a}$. We let $g=  \frac{1  }{2v_0}\mathrm{e}^{\frac{\mu_0+ by}{a}}$ and $f =\frac{1}{2v_0}\mathrm{e}^{\frac{\mu_0 + b\alpha_0}{a}}(1+ \frac{b}{a}(y-\alpha_0))$.  Then, by the mean value theorem
    \begin{align*}
        \mathrm{arcsinh} (|r|g) -  \mathrm{arcsinh} (|r|f) 
        &= \frac{|r|}{\sqrt{1 + z^2}} (g-f)
    \end{align*}
    with $z\in (|r|f,|r|g)$ where $f\leq g$. So, we have
    \begin{align*}
         \mathrm{arcsinh} (|r|g) =   \mathrm{arcsinh} (|r|f)  + \frac{|r|}{\sqrt{1 + z^2}} \mathcal{O}((y-\alpha_0)^2).
    \end{align*}
   We also note that 
   \begin{align*}
         0 \leq \frac{|r|}{\sqrt{1 + z^2}} \leq \frac{|r|}{\sqrt{1 + (|r|f)^2}} \leq \frac{1}{|f|} = \frac{1}{|\Tilde{c}||1+\frac{b}{a}(y-\alpha_0)|}   ,
   \end{align*}
   where $f \sim \Tilde{c}$ if $y$ is close to $\alpha_0$.
   Consequently, we have the same order of error as in \eqref{eq:expo} as desired. 
\end{proof}

\begin{remark}\label{remark:ratestate1}
Above, we gave formal arguments that our model is a first-order expansion around the initial value $\alpha_0$. We will now investigate if our approximated model has the same qualitative behavior as the original model. Following \cite{rice2001rate}, the key restriction on $G$ is that when the slip rate $|\dot{u}_\tau|$ is constant, the equation $\dot{\alpha} = G(\alpha, |\dot{u}_\tau|)$  has a stable solution that evolves monotonically towards a definite value of $\alpha$, denoted $\alpha^\ast=\alpha^\ast(|\dot{u}_\tau|)$, at which $G(\alpha^\ast,|\dot{u}_\tau|) = 0$. This holds true if $\dfrac{\partial G}{\partial \alpha} < 0$, which is easily verified for the approximation of $G$ given by \eqref{eq:alpha_approx}. 
\\
\indent
For the friction term, again following \cite{rice2001rate}, we seek $\dfrac{\partial \mu}{\partial |\dot{u}_\tau|} > 0$ and $\dfrac{\partial \mu}{\partial \alpha} >0$. The first condition is consistent with the experimental observations and holds if $\alpha$ is close to $\alpha_0$. The latter condition agrees with the established convention for the state variable; larger values mean greater strength. This is also consistent with the usual interpretation of $\alpha$ as a measure of contact maturity and the fact that more mature contact is stronger. Consequently, this shows that qualitatively our approximate model has the same behavior as the original model problem. One can also see in Figure \ref{fig:sub1}-\ref{fig:sub2} that there is a neighborhood where the first-order approximations \eqref{eq:alpha_approx}-\eqref{eq:mu_approx} are close to the original equations for the values used in Table \ref{tab:table1}. 
\end{remark}

\begin{table}[h]
  \begin{center}
      \caption{Parameters used for the experiments are found in \cite{parameters}.}    
    \label{tab:table1}
    \begin{tabular}{c|c} 
      \textbf{Symbol} & \textbf{Value}\\
      \hline
      $a$ & $0.011$  \\
      $b$ & $0.014$ \\
      $L$ & $5\cdot 10^{-5}$ m \\
      $v_0$ & $10^{-9}$ m/s \\
      $\mu_0$ & $0.7$ \\
      $|\Dot{u}_\tau(t=0)|$ & $10^{-9}$ m/s\\
      $\alpha_0$ & $\ln\big(\frac{v_0}{L}\big)$ \\
    \end{tabular}
  \end{center}
\end{table}

\begin{figure}[H]
  \centering
  \includegraphics[scale=0.53]{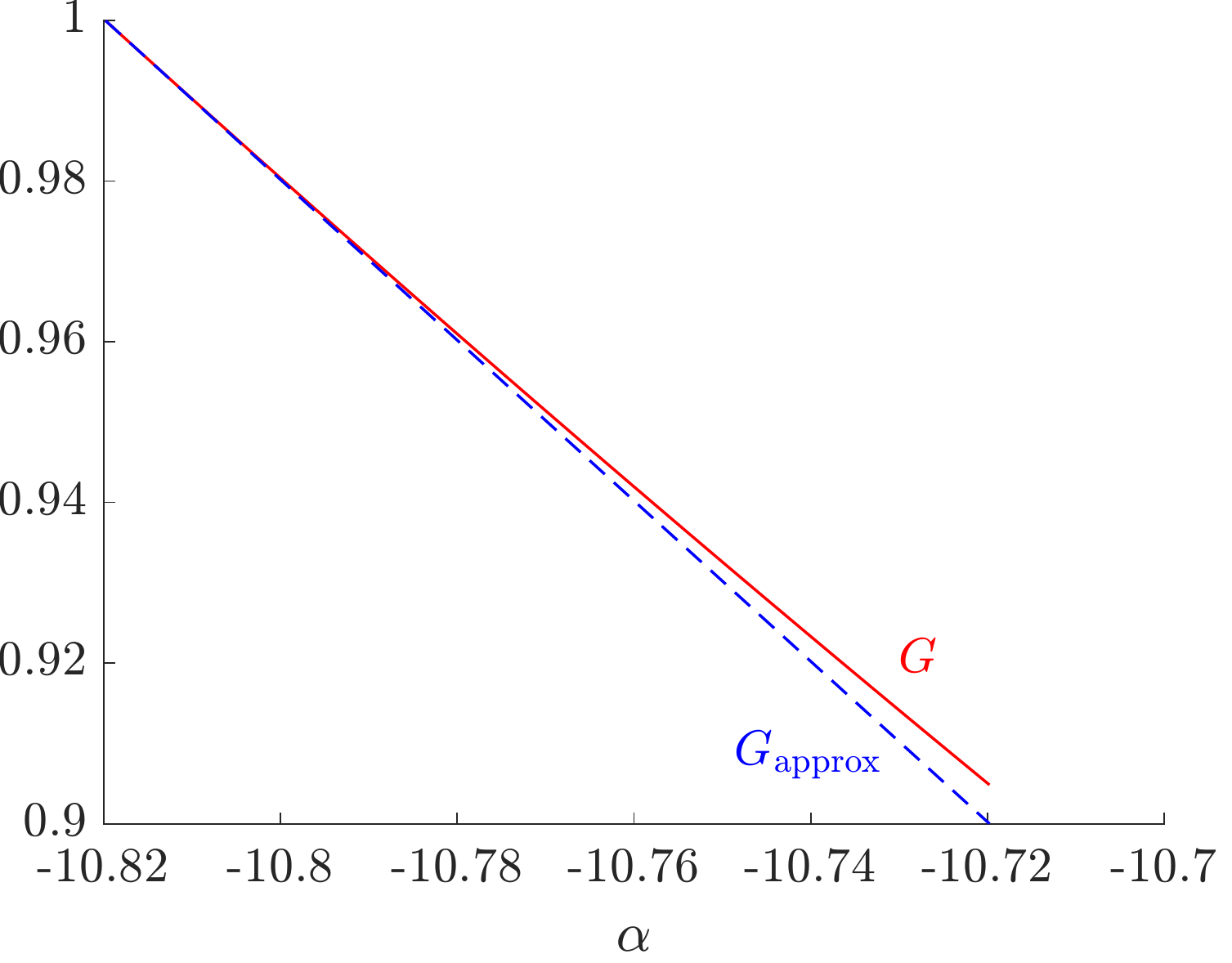}
  \caption{With the parameters in Table \ref{tab:table1}, the red line denotes $G = G(\alpha(t),|\Dot{u}_\tau(0)|)$ given by \eqref{eq:aginglaw}, and the blue dashed line denotes $G_{\mathrm{approx}}= G_{\mathrm{approx}}(\alpha(t),|\Dot{u}_\tau(0)|)$, the approximation of $G$, given by \eqref{eq:alpha_approx}.}
  \label{fig:sub1}
\end{figure}
\vspace{0.3cm}
\begin{figure}[H]
  \centering
  \includegraphics[scale=0.53]{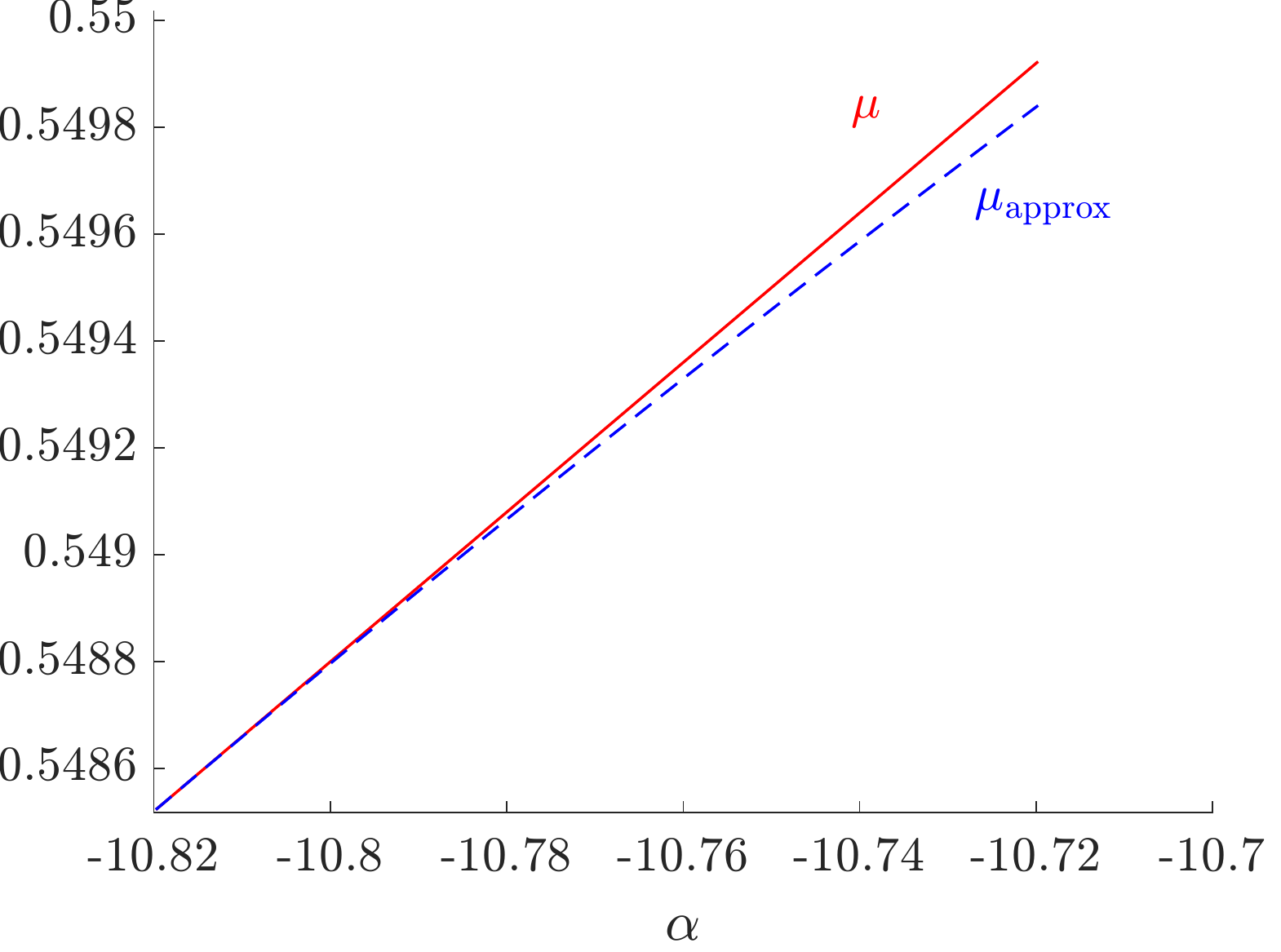}
  \caption{The friction coefficient $\mu = \mu(|\Dot{u}_\tau(0)|,\alpha(t))$ given by \eqref{eq:regularized} (red line) and $\mu_{\mathrm{approx}} = \mu_{\mathrm{approx}}(|\Dot{u}_\tau(0)|,\alpha(t))$, the approximated $\mu$, given by \eqref{eq:mu_approx} (blue dashed line) for the parameters in Table \ref{tab:table1}.}
  \label{fig:sub2}
\end{figure}

We require the following assumption on the coefficients in \eqref{eq:alpha_approx}-\eqref{eq:mu_approx}:
\begin{subequations}\label{eq:abLmu0}
    \begin{align}
        \alpha_0 &\in L^\infty (\Gamma_C), \\
       v_0^{-1}, L^{-1} &\in L^\infty(\Gamma_C),\\
         v_0,\mu_0,a,b&\in L^\infty(\Gamma_C),  \\
         L,a,&v_0  \: \text{ are nonzero}.
         \end{align}
\end{subequations}
Then, we have the following consequences of Theorem \ref{thm:finalthm} and \ref{thm:wellposed_app2}, respectively.

\begin{cor}\label{cor:cor}
   Assume that
   hypotheses \hyperref[assumptionAcal]{$\H{\mathcal{A}}$}, \hyperref[assumptionB]{$\H{\mathcal{B}}$}, \hyperref[assumptionp]{$\H{p}$}, \hyperref[assumptionC]{$\H{\mathcal{C}}$}, \eqref{eq:denisty}-\eqref{eq:assumptionondata1}, and \eqref{eq:abLmu0} holds. Then, there exists a $T>0$ satisfying \eqref{eq:times_T} such that Problem \ref{prob:weaksol} with \eqref{eq:alpha_approx}-\eqref{eq:mu_approx}  has a unique solution $(u,\alpha)$ under the smallness-conditions 
    \begin{align}\label{eq:cor1_smallness_ass1}
		m_\mathcal{A} &> \sqrt{2} \: p^\ast(\sqrt{\mathrm{meas}(\Gamma_C)} \norm{a-b\alpha_0}_{L^\infty(\Gamma_C)} + \norm{b}_{L^\infty(\Gamma_C)}\norm{\alpha_0}_{L^2(\Gamma_C)} )
		\\
		&\times \norm{\frac{1}{2v_0}\mathrm{e}^{\frac{\mu_0 + b \alpha_0 }{a}}}_{L^\infty(\Gamma_C)}   \norm{\gamma_\tau}_{ \mathcal{L}(V,L^4(\Gamma_C;\R^d))}
		\norm{(\gamma_\tau,\gamma_\nu)}_{\mathcal{L}(V,L^4(\Gamma_C;\R^d)) \times \mathcal{L}(V,L^4(\Gamma_C)) }.\notag
	\end{align}
	We obtain the following regularity
	\begin{equation*}
		u\in W^{1,2}(0,T;V),\ \ \Dot{u} \in \mathcal{W}^{1,2}_T \subset C([0,T];H), \ \ \alpha \in C([0,T];L^2(\Gamma_C)).
	\end{equation*}
	Moreover,  there exists a neighborhood around $(u_0,w_0,\alpha_0)$ so that the flow map $\Tilde{F} : V \times V \times L^\infty(\Gamma_C) \rightarrow W^{1,2}(0,T/2;V) \times C([0,T/2]; L^2(\Gamma_C))$, $(u_0,w_0,\alpha_0) \mapsto (u,\alpha)$ is continuous.
\end{cor}
\begin{cor}\label{cor:cor2}
    Assume that  \hyperref[assumptionAcal]{$\H{\mathcal{A}}$}, \hyperref[assumptionB]{$\H{\mathcal{B}}$}, \hyperref[assumptionmu]{$\H{\mu}$}, \hyperref[assumptionGapp]{$\H{G}$}, \hyperref[assumptionC]{$\H{\mathcal{C}}$}, \eqref{eq:denisty}-\eqref{eq:assumptionondata1}, \hyperref[assumptionjnu]{$\H{j_\nu}$}, and \eqref{eq:abLmu0} holds. Then, there exists a $T>0$ satisfying \eqref{eq:times_T} so that Problem \ref{prob:weaksol2} with \eqref{eq:alpha_approx}-\eqref{eq:mu_approx} has a unique  solution $(u,\alpha)$	under the smallness-condition \begin{align}\label{eq:cor2_smallness_ass1}
	    m_\mathcal{A} &>m_{j_\nu} \sqrt{\mathrm{meas}(\Gamma_C)} \norm{\gamma_\nu}_{\mathcal{L}(V,L^4(\Gamma_C))}^2+\sqrt{2}  \norm{\frac{1}{2v_0}\mathrm{e}^{\frac{\mu_0 + b \alpha_0 }{a}}}_{L^\infty(\Gamma_C)}  \\
	    &\times (\sqrt{\mathrm{meas}(\Gamma_C)} \norm{a-b\alpha_0}_{L^\infty(\Gamma_C)} + \norm{b}_{L^\infty(\Gamma_C)}\norm{\alpha_0}_{L^2(\Gamma_C)} )
		 \norm{\gamma_\tau}_{ \mathcal{L}(V,L^4(\Gamma_C;\R^d))}^2.
		\notag
	\end{align}
	In addition, we have the regularity: 
	\begin{equation*}
		u\in W^{1,2}(0,T;V),\ \ \ \Dot{u} \in \mathcal{W}^{1,2}_T \subset C([0,T];H), \ \ \ \alpha \in C([0,T];L^2(\Gamma_C)).
	\end{equation*}
	Moreover, the flow map depends continuously on the initial data. 
\end{cor}

\begin{remark}
    A \lq\lq practical fix\rq\rq \: to get a global well-posedness result would be through a truncation operator $R^{\ast}: \mathbb{R}_+ \rightarrow \mathbb{R}_+$ defined by
    \begin{equation*}
        R^{\ast}(r) = 
        \begin{cases}
            r \: \quad \text{if }  r \leq R
            \\ 
            R \quad \text{if } r > R.
        \end{cases}
    \end{equation*}
    Then if we replace $\mu$ with $R^{\ast}(\mu)$ we would have that $\mu \in L^{\infty}(\Gamma_C)$ and obtain the global well-posedness by Theorem \ref{thm:mainresult}\ref{list:global}. See also Remark \ref{remark:mubounded} on this point. 
\end{remark}

\begin{remark}
    We observe from \eqref{eq:cor1_smallness_ass1} and \eqref{eq:cor2_smallness_ass1} that either
    \begin{align*}
       \norm{\frac{1}{2v_0}\mathrm{e}^{\frac{\mu_0 + b \alpha_0 }{a}}}_{L^\infty(\Gamma_C)}  (\sqrt{\mathrm{meas}(\Gamma_C)} \norm{a-b\alpha_0}_{L^\infty(\Gamma_C)} + \norm{b}_{L^\infty(\Gamma_C)}\norm{\alpha_0}_{L^2(\Gamma_C)} )
    \end{align*}
    is small enough, or we must compensate by adding more viscosity. 
\end{remark}

\begin{remark}
    In \cite{Pipping2019,Pipping2015}, they study \eqref{eq:both1}-\eqref{eq:both2} in a time-discrete setting with $\mathcal{S}_\varphi w \equiv \text{constant}$ (the normal stresses are constant - referred to as Tresca friction), $\varphi$ being independent of $Mw$ in its third argument, relaxing the structure of $\varphi$ (see Remark \ref{remark:mubounded}), and putting $j^\circ \equiv 0$  in Problem \ref{prob:fullproblem}. 
\end{remark}

  \begin{proof}[Proof of Corollary \ref{cor:cor}]
    We only need to verify \hyperref[assumptionmu]{$\H{\mu}$} and \hyperref[assumptionGapp]{$\H{G}$} for \eqref{eq:alpha_approx}-\eqref{eq:mu_approx}, respectively. The rest follows by Theorem \ref{thm:finalthm} since $L^\infty(\Gamma_C) \subset L^2(\Gamma_C)$. We start by verifying \hyperref[assumptionGapp]{$\H{G}$} for \eqref{eq:alpha_approx}. From \eqref{eq:abLmu0}, we directly have that \hyperref[assumptionGapp]{$\H{G}$}\ref{list:G_app_Lipschitz}-\ref{list:G_app_0} holds with $L_G = c(\norm{v_0}_{L^\infty(\Gamma_C)},\norm{L^{-1}}_{L^\infty(\Gamma_C)},\norm{\alpha_0}_{L^\infty(\Gamma_C)})$ Next, we verify \hyperref[assumptionmu]{$\H{\mu}$} for \eqref{eq:mu_approx}. 
    \begin{claim}\label{claim:2}
        We have the following inequalities:
        \begin{subequations}
            \begin{align} \label{eq:arc2}
               \Big|\mathrm{arcsinh}(\beta \xi)| &\leq |\beta|+|\xi|,\\ \label{eq:arc_lip}
                 |\mathrm{arcsinh}(\beta_1\xi_1)-\mathrm{arcsinh}(\beta_2\xi_2)| &\leq |\beta_1| |\xi_1-\xi_2| + |\xi_2| |\beta_1-\beta_2|
            \end{align}
        \end{subequations}
        for all $\beta, \ \xi\in \R$ and $\beta_i , \ \xi_i \in \R$ for $i=1,2$.
    \end{claim}
   \indent
    By \eqref{eq:arc2}, we have
    \begin{align*}
         &|\mu(x,|r|,y)| \\
         &\leq  \norm{\frac{1}{\sqrt{2v_0}}\mathrm{e}^{\frac{\mu_0 + b \alpha_0 }{2a}}}_{L^\infty(\Gamma_C)}\Big[
         (\norm{a}_{L^\infty(\Gamma_C)} + \norm{b\alpha_0}_{L^\infty(\Gamma_C)} ) + \norm{b}_{L^\infty(\Gamma_C)} |y| +  \norm{a}_{L^\infty(\Gamma_C)} |r|\Big]
    \end{align*}
    which is \hyperref[assumptionmu]{$\H{\mu}$}\ref{list:mu1_est}   for
    $\kappa_1 =  \norm{\frac{1}{\sqrt{2v_0}}\mathrm{e}^{\frac{\mu_0 + b \alpha_0 }{2a}}}_{L^\infty(\Gamma_C)} (\norm{a}_{L^\infty(\Gamma_C)} + \norm{b\alpha_0}_{L^\infty(\Gamma_C)})$, \\$\kappa_2 =  \norm{\frac{1}{\sqrt{2v_0}}\mathrm{e}^{\frac{\mu_0 + b \alpha_0 }{2a}}}_{L^\infty(\Gamma_C)}\norm{b}_{L^\infty(\Gamma_C)}$, and $\kappa_3 =  \norm{\frac{1}{\sqrt{2v_0}}\mathrm{e}^{\frac{\mu_0 + b \alpha_0 }{2a}}}_{L^\infty(\Gamma_C)}\norm{a}_{L^\infty(\Gamma_C)}$.   Lastly, to verify \hyperref[assumptionmu]{$\H{\mu}$}\ref{list:mu1_Lipschitz}, we use \eqref{eq:arc_lip} and the triangle inequality 
    \begin{align*}
        &|\mu(x,|r_1|,y_1) - \mu(x,|r_2|,y_2)| \\
        &= |a| \: | \mathrm{arcsinh} \bigg(\frac{  \mathrm{e}^{\frac{\mu_0 + b \alpha_0 }{a}}|r_1|(1+\frac{b}{a}(y_1-\alpha_0))}{2v_0} \bigg)  -  \mathrm{arcsinh} \bigg(\frac{  \mathrm{e}^{\frac{\mu_0 + b \alpha_0 }{a}}|r_2|(1+\frac{b}{a}(y_2-\alpha_0))}{2v_0} \bigg) | \\
        &\leq \norm{\frac{1}{2v_0}  \mathrm{e}^{\frac{\mu_0 + b \alpha_0 }{a}}}_{L^\infty(\Gamma_C)} |a| \Big[|r_1|\frac{|b|}{|a|}|y_1 -y_2| + |(1+\frac{b}{a}(y_2-\alpha_0))| \: |r_1-r_2|  \Big]
        \\
        &\leq \norm{\frac{1}{2v_0}  \mathrm{e}^{\frac{\mu_0 + b \alpha_0 }{a}}}_{L^\infty(\Gamma_C)} \Big[|b||r_1| |y_1-y_2| + (|a-b\alpha_0|+|b||y_2|) \: |r_1-r_2|  \Big].
    \end{align*}
    Thus,  \hyperref[assumptionmu]{$\H{\mu}$}\ref{list:mu1_Lipschitz} holds with $L_{1\mu} =\norm{\frac{1}{2v_0}\mathrm{e}^{\frac{\mu_0 + b \alpha_0 }{a}}}_{L^\infty(\Gamma_C)} \norm{a-b\alpha_0}_{L^\infty(\Gamma_C)}$ and $L_{2\mu} = L_{3\mu}  = \norm{\frac{b}{2v_0}\mathrm{e}^{\frac{\mu_0 + b \alpha_0 }{a}}}_{L^\infty(\Gamma_C)}
    \norm{b}_{L^\infty(\Gamma_C)}$. We lastly observe that the smallness-condition \eqref{eq:smallness1} holds by \eqref{eq:cor1_smallness_ass1}.
     \begin{proof}[Proof of Claim \ref{claim:2}]
    Since $|\mathrm{arcsinh(\beta\xi)}|$ is symmetric for $\beta,\xi\in\R$, we may prove \eqref{eq:arc2} for $\beta,\xi>0$. By the mean-value theorem, it follows that
        \begin{align} \label{eq:arc456}
          \mathrm{arcsinh}(\beta)  \leq \beta.
     \end{align}
    By definition $\mathrm{arcsinh}(\beta) = \log(\beta + \sqrt{1+\beta^2})$, then by the fact that $(\beta+\sqrt{1+\beta^2})(\xi+\sqrt{1+\xi^2}) \geq \beta\xi+\sqrt{1+(\beta \xi)^2}$ and the increasing property of the logarithm, we have
     \begin{align*}
          \mathrm{arcsinh}(\beta \xi) \leq  \mathrm{arcsinh}(\beta) + \mathrm{arcsinh}(\xi).
     \end{align*}
    Then, from \eqref{eq:arc456}, we conclude that \eqref{eq:arc2} holds.
    Similarly, without loss of generality, we may assume that $\beta>\xi$, then by the mean-value theorem
     \begin{align}\label{eq:arc123}
          \mathrm{arcsinh}(\beta) -  \mathrm{arcsinh}(\xi)   \leq \beta-\xi.
     \end{align}
    Combining
     \begin{align*}
         &| \mathrm{arcsinh} (\beta_1\xi_1 )  -  \mathrm{arcsinh}(\beta_2\xi_2)| \\
        &\leq | \mathrm{arcsinh} (\beta_1\xi_1 )  -  \mathrm{arcsinh}(\beta_1\xi_2)| 
        + | \mathrm{arcsinh} (\beta_1\xi_2 )  -  \mathrm{arcsinh}(\beta_2\xi_2)|
     \end{align*}
     with \eqref{eq:arc123} implies \eqref{eq:arc_lip}.
 \end{proof}
\end{proof}

\begin{proof}[Proof of Corollary \ref{cor:cor2}]
      The proof follows by Theorem \ref{thm:wellposed_app2} and the verification of  \hyperref[assumptionmu]{$\H{\mu}$} and \hyperref[assumptionGapp]{$\H{G}$} for \eqref{eq:alpha_approx}-\eqref{eq:mu_approx}, which is done in the proof of Corollary \ref{cor:cor}. In addition, the smallness-condition \eqref{eq:smallness_app2} holds by \eqref{eq:cor2_smallness_ass1}.
\end{proof}

\SkipTocEntry\section*{Acknowledgements}
\noindent
This research was supported by the VISTA program, The Norwegian Academy of Science and Letters, and Equinor. N.S.T. would also like to thank Meir Shillor for many important comments that improved a previous version of this manuscript.

\appendix

\addcontentsline{toc}{section}{Appendix}

\SkipTocEntry\section{Comments on assumptions}\label{appendix:comments_app}
\noindent
We include a small discussion on applications fitting our assumptions:
\begin{itemize}
	\item We first consider the equation $G = G(\alpha, |\Dot{u}_\tau|)$ under the assumptions \hyperref[assumptionGapp]{$\H{G}$}. In Section \ref{sec:rateandstate}, we introduced two applications to rate-and-state friction that are included in the framework introduced in this paper.	Another application is, e.g., the slip law
	\begin{align*}
		G(\alpha, |\Dot{u}_\tau|) = |\Dot{u}_\tau|,
	\end{align*}
	which also fit the frameworks in \cite{Patrulescu2017,Migorski2022}.
    \item Neglecting $\alpha$, the friction coefficient reduces to $\mu = \mu(|\Dot{u}_\tau|)$. In this case, \hyperref[assumptionmu]{$\H{\mu}$} is the same as in, e.g., \cite[Section 6.3 and 8.1]{Migorski2012}.
	\item The assumption \hyperref[assumptionp]{$\H{p}$} is also used in, e.g., \cite{Patrulescu2017} and \cite[Section 10.3]{sofonea2017}. They hold for, e.g., a constant function and
	\begin{equation*}
		p(r) = 
		\begin{dcases}
			c_p(r^+)^m, \ \ \ \text{ if } r \leq r^\ast\\
			c_p(r^\ast)^m, \ \ \ \text{ if } r > r^\ast,
		\end{dcases}
	\end{equation*}
	where $r^\ast$ is a positive cut-off limit related to the wear and hardness of the material, $r^+= \max\{0,r\}$, $m\in \N$, and $c_p>0$ is a surface stiffness coefficient. We refer the reader to, e.g., \cite{Migorski2012,shillor2004} for more applications. We observe that $p$ is Lipschitz as we may write 
	\begin{align*}
	    r_1^m - r_2^m = (r_1-r_2)\sum_{k=0}^{m-1}r_{1}^{k}r_{2}^{m-1-k}
	\end{align*}
	and $p$ is bounded (see, e.g., \cite[Section 6.3]{Migorski2012}). 
     \item For applications of the function $j_{\nu}$ fitting the assumptions  \hyperref[assumptionjnu]{$\H{j_\nu}$}, we refer the reader to, e.g., \cite[Section 6.3]{Migorski2012} and \cite[p.185-187]{han}.
	\item One example in linear viscoelasticity where the operators $\mathcal{A}$ and $\mathcal{B}$ satisfy \hyperref[assumptionAcal]{$\H{\mathcal{A}}$} and \hyperref[assumptionB]{$\H{\mathcal{B}}$}, respectively, is the Kelvin-Voigt constitutive law
	\begin{equation*}
		\sigma_{ij} = a_{ijkl}\varepsilon_{kl}(\Dot{u}) +  b_{ijkl}\varepsilon_{kl}(u),
	\end{equation*}
	where $\sigma_{ij}$, $a_{ijkl}$, and $b_{ijkl}$ are the components of $\sigma$, $\mathcal{A}$, and $\mathcal{B}$, respectively, under the assumption $a_{ijkl} \in L^\infty(\Omega)$ and
	\begin{equation}\label{eq:sym}
		a_{ijkl} = a_{jikl} = a_{klij}.
	\end{equation}
	In addition, there exists $m_\mathcal{A}>0$ such that
	\begin{align*}
		a_{ijkl} \varepsilon_{ij}\varepsilon_{kl} \geq m_\mathcal{A} |\varepsilon|^2 \ \ \text{ for all } \varepsilon \in \mathbb{S}^d,
	\end{align*} 
	i.e., the usual ellipticity condition. This implies $a_0 = 0$ in      \hyperref[assumptionAcal]{$\H{\mathcal{A}}$}\ref{list:Acal_bounded}, and $a_1 = L_\mathcal{A}$.
	We also assume that $b_{ijkl} \in L^\infty(\Omega)$ has the same type of symmetry property as \eqref{eq:sym} (see, e.g., \cite[Remark 3.1]{han2015}).
\end{itemize}
For further applications, we refer the reader to, e.g., \cite[Section 6]{han2002},  \cite[Section 6]{Migorski2012}, and  \cite[Section 4]{Sofonea2012}.

\SkipTocEntry\section{Proof of Proposition \ref{prop:estimatewfirst}} \label{appendix:proof_part2}

\begin{proof}[Proof of Proposition \ref{prop:estimatewfirst}]
	For the simplicity of notation, let $w = w_{\alpha\xi\eta g\chi}$. We start off finding estimates on $A$, $\varphi$, and $j^\circ$. According to \hyperref[assumptionA]{$\H{A}$}\ref{list:A_maximalmonotone},
	\begin{align*}
		&\inner{A(t,w(t))}{w(t)} - \inner{A(t,0)}{w(t)} = \inner{A(t,w(t))- A(t,0)}{w(t)} \geq m_A \norm{w(t)}_V^2
	\end{align*}
	for a.e. $t\in (0,T)$. 
	Invoking \hyperref[assumptionA]{$\H{A}$}\ref{list:A_bounded}, i.e.,		$\norm{A(t,0)}_{V^\ast} \leq a_0 (t)$, and the Cauchy-Schwarz inequality gives 
	\begin{align}\label{eq:estimateonA}
		\inner{A(t,w(t))}{w(t)} \geq   m_A \norm{w(t)}_V^2 + \inner{A(t,0)}{w(t)} 
		\geq m_A \norm{w(t)}_V^2 - a_0(t)\norm{w(t)}_V
	\end{align}
	for a.e. $t\in (0,T)$. 
	Next, we take a closer look at $j^\circ$. Keeping in mind Definition \ref{def:subdiffernetial} and applying \hyperref[assumptionj]{$\H{j}$}\ref{list:j_bounded}-\ref{list:j_estimate} and the Cauchy-Schwarz inequality reads
	\begin{align}\label{eq:forsurethis}
		&j^\circ(t,\alpha(t),\chi(t),Nw(t);Nv-Nw(t)) \\ \notag
		&\hspace{1.1cm}\leq m_j \norm{N(w(t)-v)}_X^2 
		+  \Bar{m}_j(\norm{\alpha(t)}_Y +\norm{\chi(t)}_X) \norm{Nw(t)-Nv}_X  \notag\\
		&\hspace{1.1cm}- j^\circ(t,0,0,Nv;Nw(t-Nv)\notag\\
		&\hspace{1.1cm}\leq  m_j\norm{N}^2 \norm{w(t)-v}_V^2 +   \Bar{m}_j(\norm{\alpha(t)}_Y +\norm{\chi(t)}_X)\norm{N} \norm{w(t) - v}_V \notag \\
		&\hspace{1.1cm}+( c_{0j}(t) + c_{3j}\norm{N} \norm{v}_V )\norm{N} \norm{w(t) - v}_V  \notag
	\end{align}
	for all $v\in V$, a.e. $t\in (0,T)$. Similarly, we observe by Definition \ref{def:convex_subdiffernetial}, \hyperref[assumptionphi]{$\H{\varphi}$}\ref{list:phi_bounded}-\ref{list:phi_estimate}, and the Cauchy-Schwarz inequality that
	\begin{align}\label{eq:this2}
		&\varphi(t,\alpha(t),\eta(t),Mg(t),Kv) -  \varphi(t,\alpha(t),\eta(t),Mg(t),Kw(t)) \\ \notag
		&\leq 
        \beta_{2\varphi}\norm{K} \norm{\alpha(t)}_Y\norm{w(t)-v}_V 
        +\beta_{3\varphi} \norm{K}\norm{\eta(t)}_X\norm{w(t)-v}_V \\ \notag
        &+  \beta_{4\varphi}\norm{K}\norm{M}\norm{g(t)}_V\norm{w(t)-v}_V + \beta_{5\varphi}  \norm{K}\norm{M}\norm{\alpha(t)}_Y\norm{g(t)}_V\norm{w(t)-v}_V  \notag
  \\  \notag
         &+  \big[c_{0\varphi}(t) + c_{4\varphi}\norm{K}\norm{v}_V\big]\norm{K}\norm{w(t)-v}_V 
         \\
         &=: K_1 + K_2 + K_3 + K_4 + K_5 + K_6
         \notag
	\end{align}
	for all $v\in V$, a.e. $t\in (0,T)$. 
	We are now in a position to find the desired estimate.
	Choosing $v=0$ in Problem \ref{prob:first_step}, while keeping in mind \hyperref[assumptionMNK]{$\H{MNK}$}, reads
	\begin{align*}
		&\inner{\Dot{w}(t)}{w(t)} + \inner{A(t,w(t))}{w(t)} \\
		&\leq \inner{f(t)}{w(t) } - \inner{\xi(t) }{w(t) } + \varphi(t, \alpha(t),\eta(t), Mg(t), 0) -  \varphi(t,\alpha(t), \eta(t), Mg(t), Kw(t)) \\
		&+j^\circ(t, \alpha(t),\chi(t), Nw(t); -Nw(t) ) 
	\end{align*}
	for a.e. $t\in (0,T)$.
	Take $v=0$ in \eqref{eq:forsurethis}-\eqref{eq:this2} and 
	combine with \eqref{eq:estimateonA}.  Next, we integrate over the time interval $(0,T)$ and apply the Cauchy-Schwarz inequality to \eqref{eq:forsurethis}, $K_1$, $K_2$, $K_3$, and $K_5$. We apply H\"{o}lder's inequality with $\frac{1}{\infty} + \frac{1}{2} + \frac{1}{2} = 1$ to $K_4$. Applying the integration by parts formula in Proposition \ref{prop:integrationbypartsformula} (with $v_1=v_2=w(t)$ for a.e. $t\in(0,T)$) yields
	\begin{align*}
		&\norm{w(t)}_{H}^2 + (m_A- m_j\norm{N}^2)\norm{w}_{L^2_TV}^2  \\
		&\leq \norm{w(0)}_H^2 + \norm{f}_{L^2_TV^\ast}\norm{w}_{L^2_TV} + \norm{\xi}_{L^2_TV^\ast}\norm{w}_{L^2_TV}   + \norm{a_0}_{L^2(0,T)}\norm{w}_{L^2_TV}  \\
		&+ \Bar{m}_j\norm{N} \norm{\chi}_{L^2_TX}\norm{w}_{L^2_TV}   + \norm{N}  \norm{c_{0j}}_{L^2(0,T)}\norm{w}_{L^2_TV}  + \norm{K} \norm{c_{0\varphi}}_{L^2(0,T)}\norm{w}_{L^2_TV} 
		\\
    & + T^{1/2}(  \Bar{m}_j\norm{N} + \beta_{2\varphi}\norm{K}) \norm{\alpha}_{L^\infty_TY}\norm{w}_{L^2_TV} 
    +\beta_{3\varphi}\norm{K} \norm{\eta}_{L^2_TX}\norm{w}_{L^2_TV} 
    \\
    &+ \beta_{4\varphi}\norm{K}\norm{M} \norm{g}_{L^2_TV}\norm{w}_{L^2_TV}   + \beta_{5\varphi}\norm{K}\norm{M} \norm{\alpha}_{L^\infty_TY} \norm{g}_{L^2_TV}\norm{w}_{L^2_TV} 
	\end{align*}
	for a.e. $t\in (0,T)$.	
	Next, using the fact that $w(0) = w_0$, that $V\subset H$ is a continuous embedding, and applying Young's inequality implies that there exists $\epsilon >0$ such that
	\begin{align*}
		&\norm{w(t)}_{H}^2 + (m_A- m_j\norm{N}^2 - 4\epsilon )\norm{w}_{L^2_TV}^2  \\
		&\leq c(\norm{w_0}_V^2 + \frac{1}{2\epsilon}\norm{f}_{L^2_TV^\ast}^2 + \frac{1}{2\epsilon}\norm{\xi}_{L^2_TV^\ast}^2    +\frac{1}{2\epsilon} \norm{a_0}_{L^2(0,T)}^2  + \frac{1}{2\epsilon} \norm{\chi}_{L^2_TX}^2   \\
		&+  \frac{1}{2\epsilon} \norm{c_{0j}}_{L^2(0,T)}^2 + \frac{1}{2\epsilon} \norm{c_{0\varphi}}_{L^2(0,T)}^2 
		 + \frac{1}{2\epsilon} \norm{\alpha}_{L^\infty_TY}^2 
        + \frac{1}{2\epsilon} \norm{\eta}_{L^2_TX}^2  ) \\
        &+ 
        \frac{\beta_{4\varphi} \norm{K}\norm{M}+\beta_{5\varphi}\norm{K}\norm{M}\norm{\alpha}_{L^\infty_TY} }{2}
         \norm{g}_{L^2_TV}^2 \\
         &+  \frac{\beta_{4\varphi} \norm{K}\norm{M}+\beta_{5\varphi}\norm{K}\norm{M}\norm{\alpha}_{L^\infty_TY} }{2}\norm{w}_{L^2_TV}^2
	\end{align*}
	for a.e. $t\in (0,T)$. Now, we choose $4\epsilon  = \frac{1}{2}(m_A- m_j\norm{N}^2)>0$ which implies that $m_A - m_j \norm{N}^2 -4\epsilon  >0$ from \eqref{eq:assumptionbound_mA_max}.
	\\
	\indent
	It remains to find \eqref{eq:estwdot}. Let us first rearrange Problem \ref{prob:first_step}. Then the estimates \eqref{eq:forsurethis}-\eqref{eq:this2} and the Cauchy-Schwarz inequality implies
	\begin{align*}
    		&\inner{\Dot{w}(t)}{w(t)-v} + \inner{A(t,w(t))}{w(t)-v} \\
    		&\leq  \norm{f(t)}_{V^\ast} \norm{v-w(t)}_V + \norm{\xi(t)}_{V^\ast} \norm{v-w(t)}_V +  c_{0j}(t)\norm{N}\norm{v-w(t)}_V
    		\\
    		&+\Bar{m}_j\norm{N} \norm{\chi(t)}_X \norm{v-w(t)}_V 
    	 +(\bar{m}_j\norm{N}+\beta_{2\varphi}\norm{K} )\norm{\alpha(t)}_Y\norm{v-w(t)}_V  \notag
      \\
              &+\beta_{3\varphi} \norm{K}\norm{\eta(t)}_X\norm{v-w(t)}_V +  \beta_{4\varphi}\norm{K}\norm{M}\norm{g(t)}_V\norm{v-w(t)}_V  \notag
\\
&+ \beta_{5\varphi}  \norm{K}\norm{M}\norm{\alpha(t)}_Y\norm{g(t)}_V\norm{v-w(t)}_V  \notag
  +  \big[c_{0\varphi}(t) + c_{4\varphi}\norm{K}\norm{v}_V\big]\norm{K}\norm{v-w(t)}_V \notag
	\end{align*}
	for all $v\in V$, a.e. $t\in (0,T)$. Next, choosing $V \ni \Tilde{v} = v- w(t)$ with $v \in V$ arbitrary. By duality, we deduce
	\begin{align*}
		&\norm{\Dot{w}(t)}_{V^\ast} + \norm{A(t,w(t))}_{V^\ast} = \sup_{ \Tilde{v} \in V,\ \norm{\Tilde{v} }_V = 1} \Big(\inner{\Dot{w}(t)}{\Tilde{v} }_{V^\ast \times V} + \inner{A(t,w(t))}{\Tilde{v}}_{V^\ast \times V} \Big)\\
		&\leq
		\norm{f(t)}_{V^\ast} + \norm{\xi(t)}_{V^\ast}  +  \norm{N} c_{0j}(t)  + \Bar{m}_j\norm{N}\norm{\chi(t)}_X  + m_j\norm{N}^2 
		\\
	&+ (\bar{m}_j\norm{N} + \beta_{2\varphi}\norm{K}) \norm{\alpha(t)}_Y 
        +\beta_{3\varphi} \norm{K}\norm{\eta(t)}_X +  \beta_{4\varphi}\norm{K}\norm{M}\norm{g(t)}_V \\
        &+ \beta_{5\varphi} \norm{K}\norm{M}\norm{\alpha(t)}_Y\norm{g(t)}_V 
  +  c_{0\varphi}(t) + c_{4\varphi}\norm{K}+ c_{4\varphi}\norm{K}\norm{w}_V \notag
\\
&=:  \sum_{i=1}^{12} K_i
	\end{align*}
	for a.e. $t\in (0,T)$.  We square both sides and apply the arithmetic-quadratic mean inequality first to $K_9$ and then to $\overset{12}{\underset{\substack{i=1\\  i \neq 9}}{\sum}}K_i$, and then to the latter term implies
 \begin{align*}
     \norm{\Dot{w}(t)}_{V^\ast}^2 \leq  2 (K_9)^2 +c \overset{12}{\underset{\substack{i=1\\  i \neq 9}}{\sum}} (K_i)^2.
 \end{align*}
 We obtain the desired inequality integrating over the time interval $(0, T)$ and utilizing
H\"{o}lder’s inequality.
 \end{proof}

\SkipTocEntry\section{Proof of Lemma \ref{lemma:lambda_fixedpoint}} \label{appendix:lambda_fixedpoint}

\begin{proof}[Proof of Lemma \ref{lemma:lambda_fixedpoint}]
	The proof relies on the Banach fixed-point theorem. We, therefore, need to verify that the map is indeed well-defined and that it is a contractive mapping on $X_T(a)$. For the sake of presentation, we split the proof into two steps.
	\\
	\indent \textbf{Step i}  \textit{(The operator $\Lambda$ is well-defined on $X_T(a)$)}. 	Indeed, $\Lambda \alpha^1 \in X_T(a)$. We first prove that for given $\alpha^1 \in X_T(a)$, then $\norm{\Lambda \alpha^1}_{L^\infty_TY} \leq a$. We apply Minkowski's inequality to \eqref{eq:Lambda}, then we utilize the estimate \eqref{eq:est_on_G} with $\alpha^n= \alpha^1$ and $w^n= w^1$ together with \hyperref[assumptionMNK]{$\H{MNK}$} and \hyperref[assumptionG]{$\H{\mathcal{G}}$}\ref{list:G_00}. This yields
	\begin{align*}
		\norm{\Lambda \alpha^1(t)}_Y 
		&\leq \norm{\alpha_0}_Y + \int_0^t \norm{ \mathcal{G}(s,\alpha^1(s),Mw^1(s))}_Y ds \\
		&\leq \norm{\alpha_0}_Y + \int_0^t \big[L_\mathcal{G}  \norm{\alpha^1(s)}_Y + L_\mathcal{G} \norm{M} \norm{w^1(s)}_V \big]ds  + T\norm{\mathcal{G}(\cdot,0,0)}_{L^\infty_TY} 
	\end{align*}
	for a.e. $t\in (0,T)$.	From H\"{o}lder's inequality and the Cauchy-Schwarz inequality, we obtain
	\begin{align*}
		\norm{\Lambda \alpha^1(t)}_Y 
		&\leq \norm{\alpha_0}_Y +  L_\mathcal{G} T  \norm{\alpha^1}_{L^\infty_TY}   + L_\mathcal{G}   \norm{M} T^{1/2} \norm{w^1}_{L^2_TV}  + T \norm{\mathcal{G}(\cdot,0,0)}_{L^\infty_TY}.
	\end{align*}
	  We then see from Young's inequality and the estimate \eqref{eq:estiamte_w_1} that
	\begin{align*}
		&\norm{\Lambda \alpha^1}_{L^\infty_TY}^2  \leq c(1 + \norm{\alpha^1}_{L^\infty_TY}^2 ).
	\end{align*}
	Choosing $a$ such that it provides the desired upper bound concludes this part.
	\\
	\indent
	It remains to show that $\Lambda \alpha^1(t)$ is continuous in $Y$ for a.e. $t\in (0,T)$ for given $\alpha^1 \in  X_T(a)$ and $w^1 \in \mathcal{W}^{1,2}_T$. Let $t, t'\in [0,T]$, then with no loss of generality, we assume that $t < t'$. We then use Minkowski's inequality, the estimate \eqref{eq:est_on_G} (with $\alpha^n= \alpha^1$ and $w^n= w^1$), H\"{o}lder's inequality, the Cauchy-Schwarz inequality, \hyperref[assumptionMNK]{$\H{MNK}$}, and \hyperref[assumptionG]{$\H{\mathcal{G}}$}\ref{list:G_00} to deduce
	\begin{align*}
		&\norm{\Lambda\alpha^1(t') - \Lambda\alpha^1(t)}_Y \\
		&\leq \int_t^{t'} \norm{ \mathcal{G}(s,\alpha^1(s),Mw^1(s))}_Y ds\\
		&\leq \int_t^{t'} \big[  L_\mathcal{G} \norm{\alpha^1(s)}_Y  + L_\mathcal{G}\norm{M}\norm{w^1(s)}_V + \norm{\mathcal{G}(s,0,0)}_{Y} \big] ds\\
		&\leq  L_\mathcal{G} |t'-t| \norm{\alpha^1}_{L^\infty(t,t';Y)}    + L_\mathcal{G}|t'-t|^{1/2}\norm{M} \norm{w^1}_{L^2(t,t';V)} + |t'-t| \norm{\mathcal{G}(\cdot,0,0)}_{L^\infty(t,t';Y)} .
	\end{align*}
	The estimate \eqref{eq:estiamte_w_1} implies
	\begin{align*}
		\norm{\Lambda\alpha^1(t') - \Lambda\alpha^1(t)}^2_Y 
            &\leq c |t'-t|^{1/2}   + c|t'-t| \norm{\alpha^1}_{L^\infty_TY}  + |t'-t| \norm{\mathcal{G}(\cdot,0,0)}_{L^\infty_TY}
	\end{align*}
        for $k\in \mathbb{Z}^+$.
	Passing the limit $|t'-t| \rightarrow 0$, we have that indeed $\Lambda \alpha^1\in X_T(a)$.
	\\
	\indent\textbf{Step ii} \textit{(The application $\Lambda : X_T(a) \rightarrow X_T(a)$ is a contractive mapping)}. 	Let $\alpha^1_i \in X_T(a)$, $i=1,2$, and let $ w^1 \in \mathcal{W}^{1,2}_T$ be a unique solution to \eqref{eq:algorithm11}-\eqref{eq:algorithm12}. We introduce a new norm in $C([0,T];Y)$
	\begin{equation*}
		\norm{z}_\gamma = \max_{s\in[0,T]} \mathrm{e}^{-\gamma s} \norm{z(s)}_Y
	\end{equation*}
	where $\gamma>0$ is chosen later. We notice that $C([0,T];Y)$ with $\norm{\cdot}_\gamma$ is complete and $\norm{\cdot}_\gamma$  is  equivalent to the norm on $C([0,T];Y)$. 
 Then, from \eqref{eq:Lambda}, we have
	\begin{align*}
		\mathrm{e}^{-\gamma t} \norm{\Lambda\alpha^1_1(t) - \Lambda\alpha^1_2(t)}_Y &\leq  L_\mathcal{G} \mathrm{e}^{-\gamma t} \int_0^t \mathrm{e}^{\gamma s} \mathrm{e}^{-\gamma s}  \norm{\alpha^1_1(s) - \alpha^1_2(s)}_Y ds\\ 
		&\leq   L_\mathcal{G} \mathrm{e}^{-\gamma t} \norm{\alpha^1_1 - \alpha^1_2}_\gamma \int_0^t \mathrm{e}^{\gamma s}  ds \\
		&\leq \frac{L_\mathcal{G}}{\gamma}  \norm{\alpha^1_1 - \alpha^1_2}_\gamma . 
	\end{align*}
	Choosing $\gamma > L_\mathcal{G}$ implies that $\Lambda$ is a contraction on $X_T(a)$, and thus we may conclude by the Banach fixed-point theorem that $\alpha^1 \in X_T(a)$ is a unique fixed point to \eqref{eq:Lambda}. 

\end{proof}

\SkipTocEntry\section{Proof of Corollary \ref{cor:cauchysequences}} \label{appendix:cor_cauchy}

\begin{proof}[Proof of Corollary \ref{cor:cauchysequences}]
    We follow the proof of Proposition \ref{prop:cauchysequences} until \eqref{eq:I_2123}. Since $\beta_{1\varphi} = \beta_{4\varphi} =\beta_{5\varphi} =\beta_{6\varphi}=\beta_{7\varphi} = 0$, while keeping in mind \eqref{eq:assumptionbound_mA_max}, \eqref{eq:I_2123} becomes
    \begin{align*}
        \bigg[ \int_0^{t'} \norm{e_w^n(t)}^2_V dt \bigg]^{1/2} &\leq T^k\frac{c_\mathcal{R} + \beta_{3\varphi}
	\norm{K}  c_{\mathcal{S}_\varphi } + \Bar{m}_j \norm{N} c_{\mathcal{S}_j} }{ m_A -m_j \norm{N}^2 } \bigg[ 	\int_0^{t'}\int_0^t  \norm{e_w^{n-1}(s)}_V^2 ds dt \bigg]^{1/2}   \\
	&+\frac{\bar{m}_{j}\norm{N}+\beta_{2\varphi}\norm{K}}{ m_A -m_j \norm{N}^2 }\bigg[ \int_0^{t'}\norm{e_\alpha^{n-1}(t)}_Y^2dt \bigg]^{1/2}
    \end{align*}
    for all $t' \in (0,T)$ and some $k\geq 1/2$.  From a standard Gr\"{o}nwall argument (see, e.g., \cite[Lemma 3.2]{Sofonea2012}) combined with Young's inequality, and the Cauchy-Schwarz inequality applied to \eqref{eq:est_alpha123}, we obtain
    \begin{align}\label{eq:etimate_alpha}
    	\norm{e_\alpha^{n-1}(t)}_Y^2 
    	\leq c T^k \int_0^t\norm{e_w^{n-1}(s)}_V^2ds + cT^k\int_0^t \int_0^s e^{c(t-s)}  \norm{e_w^{n-1}(r)}_V^2 dr ds 
    \end{align}
for a.e. $t \in (0,T)$ and some $k\geq 1/2$. Applying Young's inequality and \eqref{eq:etimate_alpha}, we have
    \begin{align*}
        \int_0^{t} \norm{e_w^n(t_{n})}^2_V dt_{n}  &\leq c 	\int_0^t\int_0^{t_n}  \norm{e_w^{n-1}(t_{n-1})}_V^2 dt_{n-1} dt_{n} \\
	&+ c \int_0^{t} \int_0^{t_{n}} e^{c(t-t_{n})}  \norm{e_w^{n-1}(t_{n-1})}_V^2 dt_{n-1} dt_{n}.
    \end{align*}
    Iterating over $n \in \mathbb{Z}_+$ implies
    \begin{align*}
	\int_0^{t}& \norm{e_w^n(t_{n})}^2_V dt_{n} \\
	&\leq c (1+ e^{ct})  \int_0^{t}
	  \int_0^{t_{n}}  \int_0^{t_{n-1}}  \dots  \int_0^{t_3}  \int_0^{t_2} \norm{e_w^{1}(t_1)}_V^2 dt_1 dt_2 \dots dt_{n-2} 
        dt_{n-1}dt_{n} 
	\\
	&\leq c  \Big( \norm{w^{1}}_{L^2_TV}² + T\norm{w_0}_V^2\Big)(1+ e^{cT}) 
	 \int_0^t  \int_0^{t_{n}}  \int_0^{t_{n-1}} \dots  \int_0^{t_3}   dt_2 \dots dt_{n-2} dt_{n-1}  dt_n.
\end{align*}
We observe that
\begin{align*}
    \int_0^t  \int_0^{t_{n}}  \int_0^{t_{n-1}} \dots  \int_0^{t_3}   dt_2 \dots dt_{n-2}  dt_{n-1}  dt_n = \frac{t^{n-1}}{(n-1)!} \leq \frac{T^{n-1}}{(n-1)!}.
\end{align*}
Now, since $T<\infty$ and 
\begin{equation*}
n! \sim \sqrt{2\pi n}(n/e)^n,
\end{equation*}
we have 
\begin{align*}
\lim_{n\rightarrow \infty} c\frac{T^{n-1}}{(n-1)!} =0.
\end{align*}
We may therefore conclude the proof in the same way as in the proof of Proposition \ref{prop:cauchysequences}.
\end{proof}

\SkipTocEntry\section{Proof of Lemma \ref{lemma:est_W_n}} \label{appendix:proof_W_n}
 \begin{proof}[Proof of Lemma \ref{lemma:est_W_n}]
For \eqref{eq:ineq_W_n}, we utilize the conditions \hyperref[assumptionA]{$\H{A}$}\ref{list:A_maximalmonotone}, \hyperref[assumptionphi]{$\H{\varphi}$}\ref{list:phi_estimate}, \hyperref[assumptionj]{$\H{j}$}\ref{list:j_estimate}, and the Cauchy-Schwarz inequality to obtain
\begin{align*}
    &\inner{\Dot{W}^n(t)}{W^n(t)} + (m_A-m_j\norm{N}^2 ) \norm{W^n(t)}_V^2 \\
	&\leq \norm{\mathcal{R}w^{n-1}_1(t) - \mathcal{R}w^{n-1}_2(t) }_{V^\ast} \norm{W^n(t)}_V 
        \\
        &+\beta_{1\varphi}\norm{K} \norm{M}\norm{w_1^{n-1}(t)}_V\norm{\Sigma^{n-1}(t)}_Y \norm{W^n(t)}_V 
        \\
        &
        + 
        (\beta_{2\varphi}\norm{K} +\Bar{m}_j\norm{N})  \norm{\Sigma^{n-1}(t)}_Y \norm{W^n(t)}_V \\
        &+ \beta_{3\varphi}\norm{K}  \norm{\mathcal{S}_\varphi w_1^{n-1}(t)-\mathcal{S}_\varphi w_2^{n-1}(t)}_X\norm{W^n(t)}_V  \\
        &+ \beta_{4\varphi}\norm{K}\norm{M} \norm{W^{n-1}(t)}_V\norm{W^n(t)}_V
        + \beta_{5\varphi} \norm{K} \norm{M}\norm{\alpha_2^{n-1}(t)}_Y \norm{W^{n-1}(t)}_V \norm{W^n(t)}_V 
        \\
        &+
         \beta_{6\varphi} \norm{K} \norm{M}\norm{w_1^{n-1}(t)}_V\norm{\mathcal{S}_\varphi w_1^{n-1}(t)-\mathcal{S}_\varphi w_2^{n-1}(t)}_X\norm{W^n(t)}_V \\
         &+ \beta_{7\varphi}  \norm{K} \norm{\alpha^{n-1}_1(t)}_Y\norm{\mathcal{S}_\varphi w_1^{n-1}(t)-\mathcal{S}_\varphi w_2^{n-1}(t)}_X\norm{W^n(t)}_V \\
         &+ \Bar{m}_j\norm{N}\norm{\mathcal{S}_j w_1^{n-1}(t)-\mathcal{S}_j w_2^{n-1}(t)}_X\norm{W^n(t)}_V 
\end{align*}
for a.e. $t\in (0,T/2)$. First, integrating over the time interval $(0,t')\subset (0,T/2)$, applying the integration by parts formula in Proposition \ref{prop:integrationbypartsformula} (with $v_1=v_2=W(t)$ for a.e. $t\in (0,T/2)$), Young's inequality, H\"{o}lder's inequality, $W(0) = W_0$, and the fact that $V\subset H$ is a continuous embedding. Secondly, we apply the Cauchy-Schwarz inequality to \hyperref[assumptionR]{$\H{\mathcal{R}}$}, \hyperref[assumptionS1]{$\H{\mathcal{S}_\varphi}$}, and \hyperref[assumptionS2]{$\H{\mathcal{S}_j}$}, respectively, gives us the following estimates
\begin{align*}
	\norm{\mathcal{R}w_1^{n-1}(t) - \mathcal{R}w_2^{n-1}(t)}_{V^\ast}^2 &\leq c_\mathcal{R}^2 \frac{T}{2} \int_0^t  \norm{W^{n-1}(s)}_V^2 ds ,
	\\
		\norm{\mathcal{S}_\varphi w_1^{n-1}(t) - \mathcal{S}_\varphi  w^{n-1}_2(t)}_X^2 &\leq c_{\mathcal{S}_\varphi }^2 \frac{T}{2}	\int_0^t  \norm{W^{n-1}(s)}_V^2 ds,
	\\
    \norm{\mathcal{S}_jw_1^{n-1}(t) - \mathcal{S}_j w_2^{n-1}(t)}_X^2 &\leq c_{\mathcal{S}_j}^2 \frac{T}{2} 	\int_0^t  \norm{W^{n-1}(s)}_V^2 ds 
\end{align*}
for a.e. $t \in (0,T/2)$. This yields
  \begin{align*}
      & (m_A - m_j\norm{N}^2 - 4\epsilon)\norm{W^n}_{L^2_{T/2}V}^2 \\
	&\leq c(\norm{W_0}_V^2   
    + \frac{1}{2\epsilon} \norm{w_1^{n-1}}_{L^2_{T/2}V}^2
    \norm{\Sigma^{n-1}}_{L^\infty_{T/2}Y}^2
	+ \frac{1}{2\epsilon} \norm{\Sigma^{n-1}}_{L^\infty_{T/2}Y}^2
         \\
         & +  \frac{T^k}{2\epsilon}\norm{w_1^{n-1}}_{L^2_{T/2}V}^2 \norm{W^{n-1}}_{L^2_{T/2}V}^2
             +   \frac{T^k}{2\epsilon}  \norm{\alpha_1^{n-1}}_{L^\infty_{T/2}Y}^2 \norm{W^{n-1}}_{L^2_{T/2}V}^2 )
             \\
             &+    \frac{\beta_{4\varphi} \norm{K} \norm{M}+ \beta_{5\varphi} \norm{K} \norm{M}\norm{\alpha_2^{n-1}}_{L^\infty_{T/2}Y}}{2} 
            \norm{W^{n-1}}_{L^2_{T/2}V}^2  \\
            &+ \frac{\beta_{4\varphi} \norm{K} \norm{M}+ \beta_{5\varphi} \norm{K} \norm{M}\norm{\alpha_2^{n-1}}_{L^\infty_{T/2}Y}}{2}  \norm{W^n}_{L^2_{T/2}V}^2
  \end{align*}
  for $k\geq 1/2$. Keeping in mind \eqref{eq:assumptionbound_mA_max}, we may
    choose $4\epsilon = \frac{1}{2}(m_A - m_j\norm{N}^2) >0$ to obtain the desired result.
\end{proof}

\SkipTocEntry\section{Proof of Lemma \ref{lemma:assumptionphi2}} \label{appendix:proofs_app1}

\begin{proof}[Proof of Lemma \ref{lemma:assumptionphi2}]
    The assumptions on $A$, \hyperref[assumptionA]{$\H{A}$}, hold directly by the hypothesis \hyperref[assumptionAcal]{$\H{\mathcal{A}}$} with $m_A = m_\mathcal{A}$, $a_0 =0$, and $a_1 = L_\mathcal{A}$ (see, e.g., \cite[p.273]{sofonea2017}). Secondly, \hyperref[assumptionG]{$\H{\mathcal{G}}$} holds by \hyperref[assumptionGapp]{$\H{G}$} with $L_\mathcal{G} = L_G$ and $\mathcal{R}$ is history-dependent 
	with $c_\mathcal{R} = L_\mathcal{B} +  \norm{\mathcal{C}}_{L^\infty_TL^\infty(\Omega;\mathbb{S}^d)}$ by \hyperref[assumptionC]{$\H{\mathcal{C}}$} and \hyperref[assumptionB]{$\H{\mathcal{B}}$}\ref{list:Bcal_bounded}, see, e.g.,
	 \cite[p.275]{sofonea2017}. We observe that \hyperref[assumptionR]{$\H{\mathcal{R}}$}\ref{list:R0} holds by a duality argument, the Cauchy-Schwarz inequality, \hyperref[assumptionB]{$\H{\mathcal{B}}$}\ref{list:Bcal_bounded}, and \eqref{eq:assumptionondata1}.
    Moreover, it follows directly by Minkowski's inequality and the properties of the trace operator that \hyperref[assumptionS1]{$\H{\mathcal{S}_\varphi}$}\ref{list:S_hist11}-\ref{list:S011} holds with $c_{\mathcal{S}_\varphi } = \norm{\gamma_\nu}_{\mathcal{L}(V,Y)}$.   Additionally, \eqref{eq:initaldata}-\eqref{eq:assumptionbound_mA_max} is a consequence of \eqref{eq:assumptiononsourceterms}-\eqref{eq:assumptionondata} and a duality argument.
 \\
	\indent
	The verification of \hyperref[assumptionphi]{$\H{\varphi}$} requires some work.	Firstly, \hyperref[assumptionphi]{$\H{\varphi}$}\ref{list:phi_measurable} holds as the variables in  \eqref{eq:defintion_phi} do not explicitly depend on $t$.
	Next, we show that $\varphi(t,\cdot,\cdot,\cdot,\Tilde{v})$ is continuous on $Y \times X \times U$ for all $\Tilde{v}=(v^{(1)},v^{(2)})\in Z$, a.e. $t\in (0,T)$. That is, ensuring that \hyperref[assumptionphi]{$\H{\varphi}$}\ref{list:phi_cont} holds. 	Let $(y,z,\Tilde{w}),(y_0,z_0,\Tilde{w}_0) \in Y \times X \times U$ such that $\norm{(y,z,\Tilde{w})- (y_0,z_0,\Tilde{w}_0)}_{ Y \times X \times U} \rightarrow 0$. Then
	\begin{align*}
		&|\varphi(t, y,z,\Tilde{w},\Tilde{v}) - \varphi(t,  y_0,z_0,\Tilde{w}_0,\Tilde{v})| \\ 
		&\leq \int_{\Gamma_C}|p(z)-p(z_0)||v^{(2)}|da + \Big|\int_{\Gamma_C}[\mu(|\Tilde{w}|,y) p(z) - \mu(|\Tilde{w}_0|,y_0) p(z_0)] |v^{(1)}| da \Big| .
	\end{align*}
	For simplicity in notation, we let
	\begin{equation*}
		\mu = \mu (|\Tilde{w}| , y), \ \ \ \mu_0 = \mu (|\Tilde{w}_{0}| , y_0),
	\end{equation*}
	then
	\begin{align*}
	    &|\varphi(t, y,z,\Tilde{w},\Tilde{v}) - \varphi(t,  y_0,z_0,\Tilde{w}_0,\Tilde{v})| \\ 
		&\leq \int_{\Gamma_C}|p(z)-p(z_0)||v^{(2)}|da + \int_{\Gamma_C}|p(z)||\mu  - \mu_0| |v^{(1)}| da +  \int_{\Gamma_C}|\mu_0||p(z)  - p(z_0)| |v^{(1)}| da\\
		&=: I + II + III.
	\end{align*}
    We estimate $I$ directly by H\"{o}lder's inequality and \hyperref[assumptionp]{$\H{p}$}\ref{list:p_Lipschitz}, while for $II$ and $III$, we use H\"{o}lder's inequality, \hyperref[assumptionp]{$\H{p}$}\ref{list:p_bounded}, and \hyperref[assumptionmu]{$\H{\mu}$}\ref{list:mu1_Lipschitz}.     This reads
	\begin{align*}
	    I &\leq L_p\sqrt{\mathrm{meas}(\Gamma_C)}\norm{z-z_0}_X \norm{v}_Z, \\
	    II &\leq p^\ast  \big(L_{3\mu}\norm{\Tilde{w}}_U\norm{y - y_0}_Y\norm{v}_Z\\
	    &+ \big(L_{1\mu} \sqrt{\mathrm{meas}(\Gamma_C)}+  L_{2\mu}  \norm{y_0}_Y\big) \norm{\Tilde{w}-\Tilde{w}_0}_U  \norm{v}_Z, \\
		III &\leq L_p\big(\kappa_1 \sqrt{\mathrm{meas}(\Gamma_C)} + \kappa_2  \norm{y_0}_Y  +\kappa_3(\mathrm{meas}(\Gamma_C))^{1/4} \norm{\Tilde{w}_0}_U \big) \norm{z  - z_0}_X \norm{v}_Z.
	\end{align*}
	\indent
	To verify \hyperref[assumptionphi]{$\H{\varphi}$}\ref{list:phi_convex_lsc}, we first note that a continuous function is lower semicontinuous. So, it suffices to show that $\varphi(t,y,z,\Tilde{w},\cdot)$ is Lipschitz continuous on $Z$ for all $y\in Y$, $z\in X$, $\Tilde{w} \in U$, a.e. $t\in (0,T)$. Let $\Tilde{v}_i =(v^{(1)}_i, v^{(2)}_i)\in Z $, $i=1,2$, then \hyperref[assumptionp]{$\H{p}$}\ref{list:p_Lipschitz}-\ref{list:p_bounded}, \hyperref[assumptionmu]{$\H{\mu}$}\ref{list:mu1_est}, and H\"{o}lder's inequality yields
	\begin{align}\label{eq:star}
		&|\varphi(t, y,z,\Tilde{w},\Tilde{v}_1) - \varphi(t,  y,z,\Tilde{w},\Tilde{v}_2)| \\ \notag
		&\leq \int_{\Gamma_C}|p(z)||v^{(2)}_{1} - v^{(2)}_{2}|da +p ^\ast  \int_{\Gamma_C}|\mu(|\Tilde{w}|,y)| [|v^{(1)}_{1}| - |v^{(1)}_{2}|] da \\
		&\leq L_p \sqrt{\mathrm{meas}(\Gamma_C)}  \norm{z}_X \norm{\Tilde{v}_1 - \Tilde{v}_2}_Z \notag\\
            &+  ( \kappa_1(\mathrm{meas}(\Gamma_C))^{3/4} + \kappa_2(\mathrm{meas}(\Gamma_C))^{1/4} \norm{y_0}_Y +\kappa_3\sqrt{\mathrm{meas}(\Gamma_C)} \norm{\Tilde{w}_0}_U ) \norm{\Tilde{v}_1 - \Tilde{v}_2}_Z \notag
	\end{align}
	for all $y \in Y$, $z\in X$, $\Tilde{w}\in U$, a.e. $t\in (0,T)$. The triangle inequality and the linearity of the integral guarantee the convexity in the last argument of $\varphi$. Thus, we have that \hyperref[assumptionphi]{$\H{\varphi}$}\ref{list:phi_convex_lsc} holds.
	\\
	\indent
 From
  \eqref{eq:star} and $Z^\ast = U^\ast \times X^\ast = (L^4(\Gamma_C;\R^d))^\ast \times (L^4(\Gamma_C))^\ast = L^{4/3}(\Gamma_C;\R^d) \times L^{4/3}(\Gamma_C)$, we observe
  that  \hyperref[assumptionphi]{$\H{\varphi}$}\ref{list:phi_bounded} holds     for 
    $c_{0\varphi}(t) = \kappa_1(\mathrm{meas}(\Gamma_C))^{3/2}$, $c_{1\varphi}= \kappa_2\mathrm{meas}(\Gamma_C)$, $c_{2\varphi} = L_p(\mathrm{meas}(\Gamma_C))^{5/4}  $, $c_{3\varphi} = \kappa_3(\mathrm{meas}(\Gamma_C))^{5/4}$, and $c_{4\varphi} = 0$. 
    We lastly verify \hyperref[assumptionphi]{$\H{\varphi}$}\ref{list:phi_estimate}. Let $y_i\in Y$, $z_i \in X$, $\Tilde{w}_i\in U$, $\Tilde{v}_i =(v^{(1)}_i,v^{(2)}_i) \in Z$, for $i=1,2$. Then
	\begin{align*}
		&\varphi(t,y_1,z_1,\Tilde{w}_1,\Tilde{v}_2) -  \varphi(t,y_1,z_1,\Tilde{w}_1,\Tilde{v}_1) +  \varphi(t,y_2,z_2,\Tilde{w}_2,\Tilde{v}_1)
		- \varphi(t,y_2,z_2,\Tilde{w}_2,\Tilde{v}_2) \\
		&=  \int_{\Gamma_C} \big[p(z_1)- p(z_2)\big] \big[v^{(2)}_{2}-v^{(2)}_{1}\big] da + \int_{\Gamma_C}   \mu (|\Tilde{w}_{1}| , y_1) p(z_1) \big[|v^{(1)}_{2}|-|v^{(1)}_{1}|\big] da\\
		&+  \int_{\Gamma_C}   \mu (|\Tilde{w}_{2}| , y_2) p(z_2) \big[|v^{(1)}_{1}|- |v^{(1)}_2|\big] da.
	\end{align*}
	For simplicity, we define
	\begin{align*}
		\mu_1 &=  \mu (|\Tilde{w}_{1}| , y_1),  & \mu_2 =  \mu (|\Tilde{w}_{2}| , y_2),
	\end{align*}
	then by \hyperref[assumptionp]{$\H{p}$}\ref{list:p_bounded}, \hyperref[assumptionmu]{$\H{\mu}$}\ref{list:mu1_Lipschitz}, and the triangle inequality
	\begin{align*}
	    &\varphi(t,y_1,z_1,\Tilde{w}_1,\Tilde{v}_2) -  \varphi(t,y_1,z_1,\Tilde{w}_1,\Tilde{v}_1) +  \varphi(t,y_2,z_2,\Tilde{w}_2,\Tilde{v}_1)
		- \varphi(t,y_2,z_2,\Tilde{w}_2,\Tilde{v}_2) \\
		&=  \int_{\Gamma_C} \big[p(z_1)- p(z_2)\big] \big[v^{(2)}_{2}-v^{(2)}_{1}\big] da + \int_{\Gamma_C}   p(z_2)\big[\mu_1 -\mu_2\big] \big[|v^{(1)}_{2}|-|v^{(1)}_{1}|\big] da\\
		&+  \int_{\Gamma_C}   \mu_1\big[p(z_1)- p(z_2)\big] \big[|v^{(1)}_{2}|- |v^{(1)}_1|\big] da\\
		&\leq L_p\int_{\Gamma_C} |z_1- z_2| |v^{(2)}_{2}-v^{(2)}_{1}| da
         + p^\ast\int_{\Gamma_C}  L_{1\mu} |\Tilde{w}_1 - \Tilde{w}_2| |v^{(1)}_{2} - v^{(1)}_{1}| da
        \\
        &+ p^\ast  \int_{\Gamma_C} L_{2\mu}|y_2||\Tilde{w}_1 - \Tilde{w}_2| |v^{(1)}_{2} - v^{(1)}_{1}| da + p^\ast \int_{\Gamma_C}  L_{3\mu}|\Tilde{w}_1| |y_1 - y_2||v^{(1)}_{2} - v^{(1)}_{1}| da \\
        &+  L_p\int_{\Gamma_C} (\kappa_1 + \kappa_2|y_1|  + \kappa_3|\Tilde{w}_1|) |z_1- z_2| |v^{(1)}_{2}- v^{(1)}_1| da\\
		&=: I + II + III + IV + V.
	\end{align*}
	Next, we apply H\"{o}lder's inequality and \hyperref[assumptionmu]{$\H{\mu}$}\ref{list:mu1_est} to obtain 
	\begin{align*}
	    I &\leq L_p \sqrt{\mathrm{meas}(\Gamma_C)}  \norm{z_1- z_2}_X \norm{v^{(2)}_1-v^{(2)}_2}_X,\\
	    II &\leq p^\ast L_{1\mu} \sqrt{\mathrm{meas}(\Gamma_C)} \norm{\Tilde{w}_1 - \Tilde{w}_2}_U \norm{v^{(1)}_{1} - v^{(1)}_{2}}_U,\\
             III &\leq p^\ast  L_{2\mu}  \norm{y_2}_Y \norm{\Tilde{w}_1 - \Tilde{w}_2}_U \norm{v^{(1)}_1 -v^{(1)}_2}_U,\\
	    IV &\leq p^\ast  L_{3\mu} \norm{\Tilde{w}_1}_U \norm{y_1 - y_2}_Y \norm{v^{(1)}_1 -v^{(1)}_2}_U,\\
	    V &\leq \kappa_1  \sqrt{\mathrm{meas}(\Gamma_C)} L_p  \norm{z_1 - z_2}_X \norm{v^{(1)}_1 -v^{(1)}_2}_U
        +  \kappa_2L_p \norm{y_1}_Y     \norm{z_1 - z_2}_X  \norm{v^{(1)}_1 -v^{(1)}_2}_U\\
        &+  \kappa_3 (\mathrm{meas}(\Gamma_C))^{1/4} L_p \norm{\Tilde{w}_1}_U      \norm{z_1 - z_2}_X  \norm{v^{(1)}_1 -v^{(1)}_2}_U.
     \end{align*}
    Hence, \hyperref[assumptionphi]{$\H{\varphi}$}\ref{list:phi_estimate} holds with $\beta_{1\varphi}= p^\ast L_{3\mu}$, $\beta_{3\varphi} = L_p(1+ \kappa_1) \sqrt{\mathrm{meas}(\Gamma_C)} $,\\ $\beta_{4\varphi} = p^\ast  L_{1\mu} \sqrt{\mathrm{meas}(\Gamma_C)}$, $\beta_{5\varphi} = p^\ast L_{2\mu}$, $\beta_{6\varphi} = \kappa_3 L_p(\mathrm{meas}(\Gamma_C))^{1/4}$, $\beta_{7\varphi} =  \kappa_2 L_p$,   and $\beta_{2\varphi} = 0$.
\end{proof}

\SkipTocEntry\section{Proof of Lemma \ref{lemma:assumptionon_j}} \label{appendix:assumption_phiandj}

\begin{proof}[Proof of Lemma \ref{lemma:assumptionon_j}]
    We will first prove that $\varphi$ defined by \eqref{eq:defintion_phi_2} satisfies \hyperref[assumptionphi]{$\H{\varphi}$}. Firstly, \hyperref[assumptionphi]{$\H{\varphi}$}\ref{list:phi_measurable} holds as the variables in \eqref{eq:defintion_phi_2} do not explicitly depend on $t$. Next, \hyperref[assumptionphi]{$\H{\varphi}$}\ref{list:phi_cont} holds, i.e., $\varphi(t,\cdot,\cdot,\Tilde{v})$ is continuous on $Y \times U$ for all $\Tilde{v} \in Z$, a.e. $t\in (0,T)$. Indeed, let $(y,\Tilde{w}),(y_0,\Tilde{w}_0) \in Y \times U$ such that $\norm{(y,\Tilde{w}) - (y_0,\Tilde{w}_0)}_{Y \times U} \rightarrow 0$,	it then follows by \hyperref[assumptionmu]{$\H{\mu}$}\ref{list:mu1_Lipschitz}, and H\"{o}lder's inequality that
	\begin{align*}
	       |\varphi(t,y,\Tilde{w},\Tilde{v}) -       
                \varphi(t,y_0,\Tilde{w}_0,\Tilde{v})| &= | \int_{\Gamma_C}  \big[\mu (|\Tilde{w}|, y) - \mu (|\Tilde{w}_0|, y_0)\big] |\Tilde{v}| da| \\
		&\leq L_{1\mu} \norm{w}_U\norm{y- y_0}_Y\norm{\Tilde{v}}_Z\\
		&+
            (L_{2\mu}\sqrt{\mathrm{meas}(\Gamma_C)}+  L_{3\mu} \norm{y_0}_Y) \norm{ \Tilde{w}- \Tilde{w}_0}_U  \norm{\Tilde{v}}_Z
	\end{align*}
	for all $\Tilde{v} \in Z$.
	Moreover, convexity in the last argument of $\varphi$ follows by the triangle inequality and the linearity of the integral. Additionally, a continuous function is lower semicontinuous. Therefore, we will show that $\varphi$ is continuous in its last argument.  Let $\Tilde{v},\Tilde{v}_0\in Z$ such that $\norm{\Tilde{v} - \Tilde{v}_0}_Z \rightarrow 0$. Then, the Cauchy-Schwarz inequality, and \hyperref[assumptionmu]{$\H{\mu}$}\ref{list:mu1_est} gives
    \begin{align}\label{eq:varphi_est_app}
		&|\varphi(t,y,\Tilde{w},\Tilde{v}_1) -  \varphi(t,y,\Tilde{w},\Tilde{v}_2)|\\ \notag
		&\leq (\kappa_1 (\mathrm{meas}(\Gamma_C))^{3/4}  + \kappa_2 (\mathrm{meas}(\Gamma_C))^{1/4} \norm{y}_Y + \kappa_3  \sqrt{\mathrm{meas}(\Gamma_C)} \norm{\Tilde{w}}_U) \norm{\Tilde{v}_{1} - \Tilde{v}_{2}}_Z \notag
	\end{align}
	for $y\in Y$, $\Tilde{w} \in U$, $\Tilde{v}_{i} \in Z$, $i=1,2$, a.e. $t\in (0,T)$.
    This proves  \hyperref[assumptionphi]{$\H{\varphi}$}\ref{list:phi_convex_lsc}.	        Similarly as in the proof of Lemma \ref{lemma:assumptionphi2}, \eqref{eq:varphi_est_app} implies that \hyperref[assumptionphi]{$\H{\varphi}$}\ref{list:phi_bounded} holds for $c_{0\varphi}(t) = \kappa_1 (\mathrm{meas}(\Gamma_C))^{3/2}$, $c_{1\varphi}= \kappa_2  \mathrm{meas}(\Gamma_C)$, $ c_{3\varphi} = \kappa_3 (\mathrm{meas}(\Gamma_C))^{5/4}$, and  $ c_{4\varphi} = 0$. Lastly, we investigate \hyperref[assumptionphi]{$\H{\varphi}$}\ref{list:phi_estimate}. From hypothesis \hyperref[assumptionmu]{$\H{\mu}$}\ref{list:mu1_Lipschitz} and H\"{o}lder's inequality, we obtain
	\begin{align*}
		&\varphi(t,y_1,\Tilde{w}_1,\Tilde{v}_2) -  \varphi(t,y_1,\Tilde{w}_1,\Tilde{v}_1) +  \varphi(t,y_2,\Tilde{w}_2,\Tilde{v}_1)- \varphi(t,y_2,\Tilde{w}_2,\Tilde{v}_2) \\
		&\leq  (L_{1\mu} \sqrt{\mathrm{meas}(\Gamma_C)} + L_{2\mu}\norm{y_2}_Y)  \norm{\Tilde{w}_1 - \Tilde{w}_2}_U \norm{\Tilde{v}_1 - \Tilde{v}_2}_Z +  L_{3\mu} \norm{\Tilde{w}_1}_U \norm{y_1 - y_2}_Y \norm{\Tilde{v}_1 - \Tilde{v}_2}_Z 
	\end{align*}
for all $y_i \in Y$, $\Tilde{w}_i \in U$, $\Tilde{v}_{i} \in Z$ for $i=1,2$, a.e. $t\in (0,T)$.	We set $\beta_{1\varphi} =L_{3\mu}$, $\beta_{4\varphi}  = L_{1\mu}\sqrt{\mathrm{meas}(\Gamma_C)}$, $\beta_{5\varphi} =L_{2\mu}$, and $\beta_{2\varphi} = \beta_{3\varphi}= \beta_{6\varphi} = \beta_{7\varphi}=0$.
Lastly,  we prove that $j$ defined by \eqref{eq:j_nu} satisfies \hyperref[assumptionj]{$\H{j}$}. Corollary \ref{cor:j}\ref{list:finite}-\ref{list:locally} guarantees that \hyperref[assumptionj]{$\H{j}$}\ref{list:j_measurable}-\ref{list:j_locallyLipschitz}. It remain to show that \hyperref[assumptionj]{$\H{j}$}\ref{list:j_convergence}-\ref{list:j_estimate}. From Corollary \ref{cor:j}\ref{list:equality}, Proposition \ref{prop:chainrule_subdiff}, and Fatou's lemma, we obtain  \hyperref[assumptionj]{$\H{j}$}\ref{list:j_convergence}. We find that \hyperref[assumptionj]{$\H{j}$}\ref{list:j_bounded} holds by \hyperref[assumptionjnu]{$\H{j_\nu}$}\ref{list:j_nu_bounded}, Young's inequality, and H\"{o}lder's inequality. Indeed, 
\begin{align*}
    |\xi|^{4/3} &\leq (\bar{c}_0 + \bar{c}_1 |v|)^{4/3} \leq 2^{1/3}(\bar{c}_0^{4/3} + \bar{c}_1^{4/3} |v|^{4/3})
\end{align*}
implies
\begin{align*}
    \norm{\xi}_{X^\ast} \leq 2^{1/4} (\mathrm{meas}(\Gamma_C))^{3/4} \Bar{c}_0 + 2^{1/4}  (\mathrm{meas}(\Gamma_C))^{2/3} \Bar{c}_1 \norm{v}_X
\end{align*}
with $c_{0j}(t) = 2^{1/4} (\mathrm{meas}(\Gamma_C))^{3/4}\Bar{c}_0$, $c_{3j} = 2^{1/4} (\mathrm{meas}(\Gamma_C))^{2/3}\Bar{c}_1$, and $c_{1j} = c_{2j} = 0$.
For \hyperref[assumptionj]{$\H{j}$}\ref{list:j_estimate}, we use Corollary \ref{cor:j}\ref{list:equality} (see, e.g., \cite{Migorski2022}), \hyperref[assumptionjnu]{$\H{j_\nu}$}\ref{list:j_nu_est}, and the Cauchy-Schwarz inequality to obtain
\begin{align*}
    j^\circ (t,v_1 ; v_2 - v_1) +   j^\circ (t,v_2 ; v_1 - v_2) 
    &\leq  m_{j_\nu} \int_{\Gamma_C} |v_1 - v_2|^2 da \leq m_{j_\nu} \sqrt{\mathrm{meas}(\Gamma_C)} \norm{v_1 - v_2}_{X}^2,
\end{align*}
where $m_j =  m_{j_\nu} \sqrt{\mathrm{meas}(\Gamma_C)}$, concluding this proof.

\end{proof}

\end{document}